\documentclass{article}
\usepackage[symbol]{footmisc}

\usepackage{amsfonts}
\usepackage{amsmath}
\usepackage{a4wide}
\usepackage{amssymb}
\usepackage{hyperref}
\usepackage{tikz-cd}
\usetikzlibrary{cd}

\newtheorem{lem}{Lemma}[section]
\newtheorem{theo}[lem]{Theorem}
\newtheorem{coro}[lem]{Corollary}
\newtheorem{propo}[lem]{Proposition}
\newtheorem{rema}[lem]{Remark}
\newtheorem{defi}[lem]{Definition}

\newenvironment{proof}{\paragraph*{Proof}}
{\par}

\newcommand\GL{{\mathrm{GL}}}

\newcommand\eps\varepsilon
\newcommand\ph\varphi

\newcommand\F{{\mathbb F}}
\newcommand\Q{{\mathbb Q}}

\newcommand\PPP{{\mathbb P}}

\newcommand\Z{{\mathbb Z}}

\title{Regular models of modular curves in prime level over $\Z_p^{\mathrm{ur}}$}

\author{Bas \stepcounter{footnote}Edixhoven\thanks{While this work was in progress, B.\,E. passed away, on January 16, 2022. 
He is deeply missed. (P.\,P.)}{\ } and Pierre Parent}

4

\begin{document}

\maketitle 
\begin{abstract}
We give regular models for modular curves associated with (normalizer of) split and non-split Cartan subgroups of 
$\GL_2 (\F_p )$ (for $p$ any prime). We then compute the group of connected components of the fiber at~$p$ of the N\'eron 
model of their Jacobians. 

AMS 2020 Mathematics Subject Classification  11G18 (primary), 11G20, 14G35 (secondary). 
\end{abstract}

\tableofcontents

\section{Introduction}
\label{Introduction} 

For $p$ a prime number, semistable models for modular curves associated with all maximal subgroups of $\GL_2 (\Z /p\Z )$ were 
determined by the present authors in \cite{EdixP20}. Recall those maximal subgroups are either Borel, 
normalizer of split Cartan and normalizer of non-split Cartan, whose corresponding curves are generally denoted by $X_0 (p)$,
$X_{\mathrm{s}}^+ (p)$ and $X_{\mathrm{ns}}^+ (p)$, respectively; to these one must add three exceptional subgroups (and the 
relevant curves). In the present article, we propose to describe {\em regular} models over $\Z$ (or, actually, over the 
maximal unramified extension $\Z_p^{\mathrm{ur}}$ of $\Z_p$, for convenience). Such models are well-known for $X_0 (p)$ since 
at least the foundational work of Deligne and Rapoport (\cite{DR73}), and we decided not to consider the case of exceptional 
subgroups, which we feel to be of lesser arithmetic interest. We have therefore focussed on the case of curves 
associated with either split or non-split Cartan subgroups in characteristic $p$, and their normalizers. What was previously 
known in that context was the only 
case of the curves $X_{\mathrm{s}} (p)$ associated with split Cartan subgroups (but not their normalizer), which had been 
studied under the form $X_0 (p^2 )$ by the first-named author (in what is probably his first published work, coming out from 
his PhD thesis (see 
\cite{Edix89b}, \cite{Edix89})). Here we recover in that case the results of op. cit. with our different method, and 
we cover the other situations. We actually give two versions, for each case, of regular models: the first ones are 
the {\em minimal regular models with normal crossings}, which come out naturally from taking Galois quotients of our previous 
semistable models and resolving singularities; then after contracting, we obtain the {\em minimal regular models}. 
(Surprisingly enough, the latter appear to be 
totally reduced, whereas the Galois quotients we started with were not - anywhere, at least in the non-split case). A look at 
Figures~\ref{figureNSgrossier} and \ref{figureSgrossier} 
(models for $X_{\mathrm{ns}} (p)$ vs. $X_{\mathrm{s}} (p)$, respectively), or Figures~\ref{figureNSgrossier+} and 
\ref{figureS+grossier} ($X_{\mathrm{ns}}^+ (p)$ vs. $X_{\mathrm{s}}^+ (p)$), shows the morality of the situation, or at least 
the strong geometric parallelism between the split and non-split cases: the special 
fibers are essentially the same, except that reduced projective lines in the split case have been cut into small pieces in 
the non-split case. Those $\PPP^1$s are $j$-lines in the split situations, whereas their broken equivalents in the non-split 
cases have no modular interpretation, as they just result from blow-ups in the course of the regularization process.

   Our models finally allow to compute the component groups of the N\'eron models of the Jacobians of those curves. 
Here, another noteworthy feature is that the Jacobian over $\Z$ of the curve $X_{\mathrm{s}}^+ (p)$ 
(therefore also $X_0 (p^2 )/w_{p^2}$) has trivial component group: $J_{\mathrm{s}}^+ (p)$ has connected fibers over $\Z$, 
so the situation is similar to that of $X_1 (p)$, as was proven by B.~Conrad, the second-named 
author and W.~Stein in the article~\cite{CEdSt03} with that precise title.

  We feel modular curves are objects which are classical enough to justify the present work; knowing their regular models is 
for instance required to describe some other important features, such as their height (cf. the recent works~\cite{BBC20} or 
\cite{BMC22}). But as in the case of the semistable 
article \cite{EdixP20}, we have also been led by diophantine motivations. Indeed, when working out the recent quadratic 
Chabauty method for modular curves as developped by J. Balakrishnan and her coauthors (see e.g. \cite{BM17}, \cite{BM21}, 
\cite{BD23}), one needs to have semistable models at one's disposal. On 
the other hand, the geometric version of quadratic Chabauty, developped by the first-named author and G. Lido (see 
\cite{EdLi22}) requires regular models. In particular, one needs to know the exponent of the component groups of the 
N\'eron models of the Jacobians (in order to map every component of the curve to the Jacobian's neutral component: check the 
$m$ in Section~2 of \cite{EdLi22}). In the case of $X_{\mathrm{ns}}^+ (p)$, we find it to be just $(p-1)$ (see Proposition~
\ref{componentsXns+}). Beyond that numerical datum, it seems that the complete description of our regular models shall in 
fact be useful for modular versions of non-abelian Chabauty (see forthcoming work by G.~Lido et al.).

\medskip

{\bf Note from the second-named author.}
While that text was in progress we learnt, in the last weeks of 2021, of the illness of the first-named author. Bas Edixhoven 
passed away on January 16, 2022, a few months before his 60th birthday. The present work was ripe enough, in the summer 
before, for Bas to give a talk about it in Oberwolfach (July 22, 2021); yet it was immature enough for him to mention then 
that ``{\em we [were] not yet happy with all proofs/computations}\footnote{Handwritten notes for his lecture (private 
communication).}''. I (P.\,P.) have tried to complete the writing-up and make (me) happy enough with it. 
However, as a colleague and friend kindly put, ``{\em  Bas had his own standards}'', which would have been ridiculous to 
mimic. Not only do I therefore claim all mistakes that could remain, but I also apologize to those who could feel 
disappointed with the style. Beste Bas, bedankt, bedankt voor alles.

\medskip  
 
{\bf Acknowledgements.} The second-named author would like to thank Bill Allombert for his advices on safety checks with 
PARI/GP, Qing Liu and Dino Lorenzini for helpful discussions on regular models, and Reinie Ern\'e for her help to improve the 
presentation of this article. Thanks are also due to the referee for his acute reading and for suggesting nice examples which 
illustrate and corroborate our computations. This work benefited from the financial support of ANR-20-CE40-0003 Jinvariant. 

\medskip

{\bf Notation.} Through all the text, we safely assume that our primes $p$ are greater than $5$, unless explicitely stated 
otherwise. (Note by the way that for $p=2, 3$ all our modular curves have genus $0$, so there is no question about regular 
models.) The notation $\Z_{p^2}$ stands for the ring of integers of the unique non-ramified quadratic extension $\Q_{p^2}$ of 
$\Q_p$ (Witt vectors of $\F_{p^2}$), and $\Z_p^{\mathrm{ur}}$ for the ring of integers of the maximal unramified extension 
$\Q_p^{\mathrm{ur}}$ of $\Q_p$ (Witt vectors of $\overline{\F}_p$).

\section{Some structural results on $X(p)$}
\label{lesStructuraux}

All through this text, we freely use results displayed in~\cite{EdixP20}, to which we refer for details about notation or definitions. In this section we sum up and develop some of the basics of op. cit.

\bigskip

Fix once and for all a prime number $p(>3)$. 
Consider some representable finite \'etale moduli problem ${\cal P}$ over $({\mathrm{Ell}})_{\Z_p}$. The moduli problem 
$({\cal P},[\Gamma (p)])$ classifies triples $(E/S,\mathrm{a},\phi)$ for $S$ a $\Z_p$-scheme, $E/S$ an elliptic curve, 
$\mathrm{a}\in{\cal P} (E/S)$ and $\phi\in[\Gamma (p)]
(E/S)$ (in the sense of \cite{KM85}). Katz--Mazur's theorems about $\Gamma(p)$-structures (\cite{KM85}, Theorems~3.6.0, 
5.1.1 and 10.9.1) assert that
$({\cal P},[\Gamma (p)])$ is representable by a regular $\Z_p$-scheme ${\cal M} ({\cal P} ,[\Gamma (p)])$, having a 
compactification we denote by $\overline{{\cal M}} ({\cal P} ,[\Gamma (p)])$. The existence of Weil's pairing $e_p (\cdot ,
\cdot )$ shows that the morphism $\overline{\cal M} ({\cal P} ,[\Gamma (p)]) \to {\mathrm{Spec}} (\Z_p)$ factorizes through 
${\mathrm{Spec}}(\Z_p[\zeta_p])$, with $\zeta_p$ some primitive $p^{\mathrm{th}}$ root of unity. For all integers $i$ 
in $\{ 1,\dots ,p-1 \}$, set 
\[
X_i :=\overline{\cal M} ({\cal P} ,[\Gamma (p)^{\zeta_p^i{\mathrm{-can}} } ])
\]
for the sub-moduli problem over $({\mathrm{Ell}} )_{\Z_p[\zeta_p ]}$ representing triples
\[
  (E/S/\Z_p[\zeta_p],\mathrm{a}, \phi) \quad \text{such that}
  \quad e_p (\phi (1,0) ,\phi (0,1)) =\zeta_p^i\,.
\]  
The obvious morphism
$$
\coprod_{i\in \F_p^*} X_i \to \overline{{\cal M}} ({\cal P} ,[\Gamma (p)])_{\Z_p[\zeta_p ]}
$$
induces, by normalization, an isomorphism of schemes over
$\Z_p[\zeta_p]$:
\begin{eqnarray}
\label{X_i}
\coprod_{i\in \F_p^*} X_i \stackrel{\sim}{\longrightarrow} 
\overline{{\cal M}} ({\cal P} ,[\Gamma (p)])^\sim_{\Z_p[\zeta_p ]}\,,
\end{eqnarray}
where ${{\cal M}} ({\cal P} ,[\Gamma (p)])^\sim_{\Z_p[\zeta_p ]}\to \overline{{\cal M}} ({\cal P} ,
[\Gamma (p)])_{\Z_p[\zeta_p ]}$ is the normalization. The triviality of $p^{\mathrm{th}}$ roots  
of unity in characteristic~$p$ shows that, after the base change
$\Z_p[\zeta_p]\to\F_p$, the $X_{i,\F_p}$ are not only isomorphic to
each other but actually equal. Moreover, the modular interpretation of
a $\Gamma (p)$-structure $\phi \colon (\Z /p\Z )^2 \to E(k)$, in the
generic case of an ordinary elliptic curve $E$ over a field $k$ of
characteristic $p$, amounts to choosing some line $L$ in $(\Z /p\Z
)^2$ which plays the role of ${\mathrm{Ker}} (\phi )$, then some point
$P$ in $E(k)$ which defines the induced isomorphism $(\Z /p\Z )^2 /L
\stackrel{\sim}{\longrightarrow} E[p](k)$. More precisely, we have:
\begin{theo}
\label{KatzMazur1}
{\bf (Katz--Mazur \cite{KM85}, 13.7.6).} Each curve $X_{i,\F_p}$
obtained from $X_i$ over $\Z_p[\zeta_p]$ via $\Z_p[\zeta_p]\to\F_p$, is
the disjoint union, with crossings at the supersingular points, of
$p+1$ copies of the $\overline{\cal M} ({\cal P})$-schemes
$\overline{\cal M}({\cal P}, [{\mathrm{ExIg}} (p,1 )])$
(cf. Figure~\ref{FigureX_i}). 
We label those Igusa schemes ${\mathrm{Ig}}_{i,L}$ for $(i,L)$ running through $\F_p^\times \times
\PPP^1 (\F_p )$.
\end{theo}

Making that model go through the semistabilization process of~\cite{EdixP20}, we do a base change and 
perform a few blow-ups to obtain a semistable model over $\Z_p^{\mathrm{ur}} [p^{1/(p^2 -1)}]$, with special fiber as 
represented in Figure~\ref{FigureX_i} (cf Section~2.2 of op. cit.). In particular, 
the $X_i$ now have two kinds of parts: the vertical (or Igusa) ones, which already showed up in Katz--Mazur model, 
and the new horizontal (Drinfeld) ones, with projective models
\begin{eqnarray}
\label{laDrinfeldOriginale}
\alpha^p \beta -\alpha \beta^p =z^{p+1} .
\end{eqnarray}
%


\begin{figure}
\begin{center}
\scalebox{1}{
\begin{picture}(180,250)(50,-80)


\put(52,159){$p+1$ (vertical) copies of}
\put(61,147){$\overline{\cal M}({\cal P}, [{\mathrm{ExIg}} (p,1 )])$}

\put(40,150){\vector(-2,-1){20}}

\put(175,155){\vector(3,-2){20}}

 
\qbezier(-30,135)(45,110)(10,90)
\qbezier(10,90)(-20,75)(10,60)
\qbezier(10,60)(40,45)(10,30)
\qbezier(10,30)(0,25)(-5,20)

\qbezier(50,135)(-25,110)(10,90)
\qbezier(10,90)(40,75)(10,60)
\qbezier(10,60)(-20,45)(10,30)
\qbezier(10,30)(20,25)(25,20)

\put(0,130){$\cdot$}
\put(9,130){$\cdot$}
\put(18,130){$\cdot$}

\put(4,101){$\cdots$}

\put(5,74){$\cdots$}

\put(5,44){$\cdots$}

 
\qbezier(-30,130)(55,110)(10,90)
\qbezier(10,90)(-30,75)(10,60)
\qbezier(10,60)(50,45)(10,30)
\qbezier(10,30)(-2,25)(-9,20)

\qbezier(50,130)(-35,110)(10,90)
\qbezier(10,90)(50,75)(10,60)
\qbezier(10,60)(-30,45)(10,30)
\qbezier(10,30)(22,25)(29,20)


\qbezier(-30,145)(36,110)(10,90)
\qbezier(10,90)(-10,75)(10,60)
\qbezier(10,60)(30,45)(10,30)
\qbezier(10,30)(2,25)(-1,20)

\qbezier(50,145)(-16,110)(10,90)
\qbezier(10,90)(30,75)(10,60)
\qbezier(10,60)(-10,45)(10,30)
\qbezier(10,30)(18,25)(21,20)



\qbezier(10,-25)(0,-20)(-5,-15)
\qbezier(-30,-70)(45,-45)(10,-25)

\qbezier(10,-25)(20,-20)(25,-15)
\qbezier(50,-70)(-25,-45)(10,-25)


\qbezier(10,-25)(-2,-20)(-9,-15)
\qbezier(-30,-65)(55,-45)(10,-25)

\qbezier(10,-25)(22,-20)(29,-15)
\qbezier(50,-65)(-35,-45)(10,-25)


\qbezier(10,-25)(2,-20)(-1,-15)
\qbezier(-30,-78)(37,-45)(10,-25)

\qbezier(10,-25)(18,-20)(21,-15)
\qbezier(50,-75)(-15,-45)(10,-25)

\put(5,-40){$\cdots$}

\put(0,-70){$\cdot$}
\put(9,-70){$\cdot$}
\put(18,-70){$\cdot$}


\put(29,10){$\cdot$}
\put(29,2){$\cdot$}
\put(29,-6){$\cdot$}

\put(-12,10){$\cdot$}
\put(-12,2){$\cdot$}
\put(-12,-6){$\cdot$}


\put(-70,3){supersingular points}

\put(-40,15){\vector(1,2){18}} 
\put(-40,13){\vector(2,1){19}}
\put(-40,0){\vector(1,-1){20}}

\put(-40,24){\vector(1,3){24}}

\put(-21,63){$\cdot$}
\put(-21,60){$\cdot$}
\put(-21,57){$\cdot$}



%

\put(90,37){\vector(-1,0){39}}


\qbezier(105,115)(215,110)(305,117)

\put(322,116){$D_{i,s}$}

\put(195,120){$\cdot$}
\put(215,119){$\cdot$}
\put(230,120){$\cdot$}

\qbezier(103,92)(215,90)(305,97)

\put(323,95){$D_{i,s}$}

\put(195,96){$\cdot$}
\put(215,96){$\cdot$}
\put(230,97){$\cdot$}

\qbezier(104,62)(215,60)(305,68)

\put(195,66){$\cdot$}
\put(215,66){$\cdot$}
\put(230,67){$\cdot$}

\qbezier(103,32)(215,32)(305,38)

\put(195,36){$\cdot$}
\put(215,36){$\cdot$}
\put(230,37){$\cdot$}

\put(323,67){$\cdot$}
\put(322,47){$\cdot$}
\put(321,27){$\cdot$}
\put(322,7){$\cdot$}
\put(323,-16){$\cdot$}


\qbezier(103,-26)(215,-29)(305,-26)

\put(195,-22){$\cdot$}
\put(215,-23){$\cdot$}
\put(230,-23){$\cdot$}

\qbezier(104,-53)(215,-54)(305,-51)

\put(320,-53){$D_{i,s}$}
\put(195,-49){$\cdot$}
\put(215,-49){$\cdot$}
\put(230,-50){$\cdot$}


\qbezier(130,143)(135,0)(125,-80)

\put(122,13){$\cdot$}
\put(122,2){$\cdot$}
\put(121,-9){$\cdot$}

\qbezier(160,142)(165,0)(158,-82)

\put(152,13){$\cdot$}
\put(152,2){$\cdot$}
\put(151,-9){$\cdot$}

\qbezier(180,145)(184,0)(182,-79) 

\put(174,13){$\cdot$}
\put(174,2){$\cdot$}
\put(173,-9){$\cdot$}


\qbezier(245,145)(242,0)(252,-80) 
 
\put(249,13){$\cdot$}
\put(249,2){$\cdot$}
\put(250,-9){$\cdot$}

\qbezier(265,144)(261,0)(272,-80) 

\put(269,13){$\cdot$}
\put(269,2){$\cdot$}
\put(270,-9){$\cdot$}

\qbezier(285,145)(279,0)(292,-80)

\put(289,13){$\cdot$}
\put(289,2){$\cdot$}
\put(290,-9){$\cdot$}


\end{picture}}
\end{center}
\caption{Katz--Mazur and stable fibers of ${X}_i$}
\label{FigureX_i}
\end{figure}
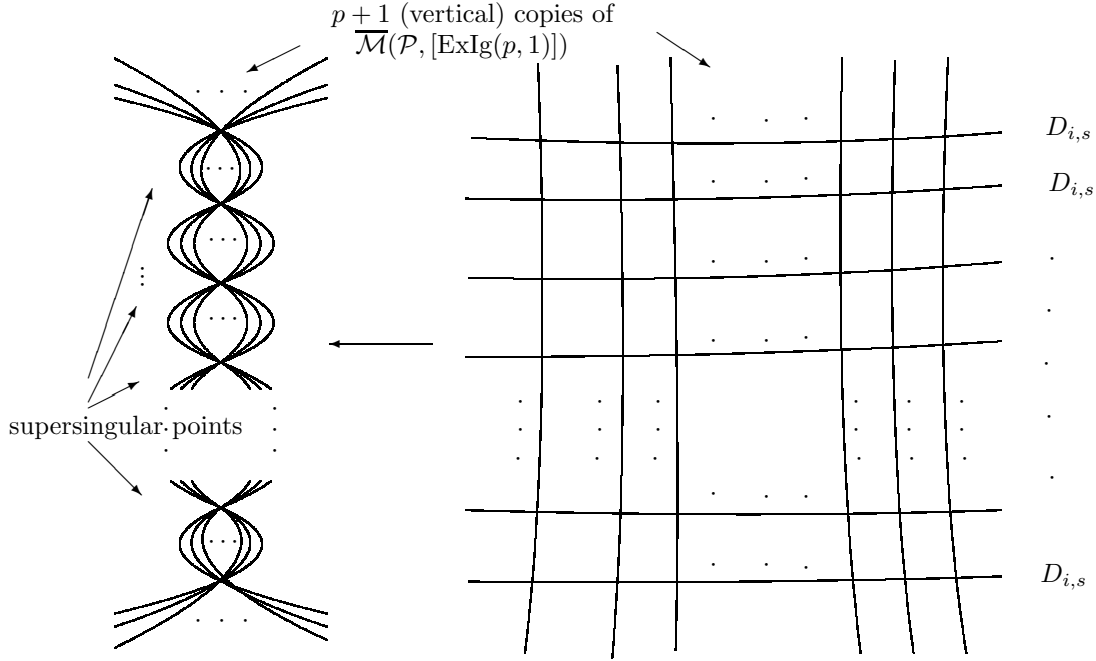

The Galois group ${\mathrm{Gal}} (\Q_{p^2} (p^{1/(p^2 -1)} )/
\Q_{p^2} )$ will be identified with $\mu_{p^2 -1} (\overline{\Q} )$ via the map
$$
\sigma \mapsto \frac{\sigma (p^{1/(p^2 -1)})}{p^{1/(p^2 -1)}} =:\zeta ,
$$
and we will also identify that group with $\mu_{p^2 -1} (\F_{p^2} )\simeq \F_{p^2}^*$, for which we choose a generator $u$. 
It acts on $\overline{{\cal M}} ({\cal P} ,[\Gamma (p)])^\sim_{\Z_p[\zeta_p ]}$ (or more precisely on each $X_i$).

   On each Drinfeld part $D_{i,s}$, with affine model $\alpha^p \beta -\alpha \beta^p =1$, the subgroup of 
${p+1}^{\mathrm{st}}$-roots of unity in $\F_{p^2}^*$ acts by $v(:= u^{p-1} )\colon (\alpha ,\beta )\mapsto (v^{-1} 
\alpha ,v^{-1} \beta)$. On the projective model $\alpha^p \beta -\alpha \beta^p =z^{p+1}$, that Galois action extends 
uniquely to
$$
(\alpha ,\beta ,z )\mapsto (v^{-1} \alpha ,v^{-1} \beta ,z),
$$  
and its fixed points are the $p+1$ points at infinity $\{ (1,0,0) ,(a,1, 0), a\in \F_p  \}$. 

\begin{rema}
\label{al'infini}
{\rm Those fixed points (at infinity) under Galois on the Drinfeld components 
are precisely the intersection points of the $D_{i,s}$ with Igusa parts (on which, on their side, the same Galois action of 
${p+1}^{\mathrm{st}}$-roots of unity is trivial).}  
\end{rema}

\medskip 

On the other hand, $\overline{{\cal M}} ({\cal P} ,[\Gamma (p)])^\sim_{\Z_p^{\mathrm{ur}} [p^{1/(p^2 -1)}]}$ is also endowed with 
a modular action of $\GL_2 (\F_p )$, also described in Section~2 of~\cite{EdixP20}. It will actually be more convenient to 
consider the action of the full group $\F_{p^2}^* \times \GL_2 (\F_p )$.
\begin{propo}
\label{stabilisateur}
The action of $\F_{p^2}^* \times \GL_2 (\F_p )$ on a given Igusa part ${\mathrm{Ig}}_{i,P}$ has stabilizer 
$$
\{(u,g), N(u):=u^{p+1} =\det (g), g\cdot P=P\} ,
$$
which acts on ${\mathrm{Ig}}_{i,P}$ by the diamond operator $\langle u^{p+1} \chi_P (g^{-1} ) \rangle_p$, for 
$\chi_P$ the character on the Borel subgroup of $\GL_2 (\F_p )$ defined by the action on the line $P$. 

   On the Drinfeld parts $D_{i,s}$ of the special fiber of $\overline{{\cal M}} ({\cal P} ,[\Gamma (p)])^\sim$, the action of 
$\F_{p^2}^* \times \GL_2 (\F_p )$ has stabilizer $\Sigma_0 :=\{(u,g), N(u)=u^{p+1}=\det (g)\}$. On each $D_{i,s}$, this 
stabilizer acts like
$$
(u,\left( {a\atop c} \ {b\atop d}\right) ) (\alpha ,\beta )=u^{-1} (a\alpha +c\beta ,b\alpha +d\beta ).
$$
\end{propo}

\begin{proof} We know from~\cite{EdixP20}, Sections~2.2.3 and 2.2.4, that the two actions of $u\in \F_{p^2}^*$ and 
$g\in \GL_2 (\F_p )$ have modular interpretation:
$$
u \colon (E/S \stackrel{f}{\longrightarrow} {\mathrm{Spec}} (\Z [\zeta_p ]) ,{\mathrm{a}},\phi )_i 
\stackrel{\overline{\gamma}(u)}{\longmapsto} (E/S \stackrel{f'}{\longrightarrow} {\mathrm{Spec}} (\Z [\zeta_p ])  ,
{\mathrm{a}},\phi )_{iu^{-p-1}}
$$  
for $f' ={\mathrm{Spec}} (u^{p+1}) \circ f$ and
$$
g \colon (E/S,{\mathrm{a}},\phi )_i \mapsto (E/S,{\mathrm{a}},\phi\circ g)_{i\det (g)} ,
$$
respectively (where the index shift from $i$ to $i\det (g)$ on the line above comes from the properties of 
Weil's pairing). It follows that $u$ and $g$ induce isomorphisms:
\begin{eqnarray}
\label{gammabar}
\overline{\gamma} (u)\colon
\left\{
\begin{array}{rcl}
{\mathrm{Ig}}_{i,P} & \stackrel{\sim}{\longrightarrow} & {\mathrm{Ig}}_{iu^{-p-1},P} \\
D_{i,s} & \stackrel{\sim}{\longrightarrow} & D_{iu^{-p-1},s}  \hspace{0.1cm} 
\end{array}
\right. 
\end{eqnarray}
and
\begin{eqnarray}
r(g)\colon
\left\{
\begin{array}{rcl}
{\mathrm{Ig}}_{i,P} & \stackrel{\sim}{\longrightarrow} & {\mathrm{Ig}}_{i\det (g),g^{-1} P} \\
D_{i,s}& \stackrel{\sim}{\longrightarrow} & D_{i\det (g), s} ,
\end{array}
\right. 
\end{eqnarray}
respectively. The two actions commute, and the stabilizers of both vertical and horizontal parts obviously are as in 
the proposition. 

  For the Igusa parts, the modular interpretation shows that if $(u,g)$ is such that $u^{p+1} =\det (g)$, then on a basis
which is adapted to the line $P$ (which is to play the role of the kernel of Frobenius) such that $g(P)=P$, $g$ needs to act 
as $\left( {\chi_P (g) \atop 0}{\ast \atop \chi_P (g)^{-1} \det (g) } \right) =\left( {\chi_P (g) \atop 0}{\ast \atop \chi_P 
(g)^{-1} u^{p+1}} \right)$, whence the diamond action $\langle \chi_P (g)^{-1} u^{p+1} \rangle$ on the (generically) \'etale 
line $(E[p]/P)_{\F_p}$.

   As for the Drinfeld parts $D_{i,s}$, one knows from~\cite{EdixP20} that the coordinates $\alpha , \beta$ of op. cit. are 
such that $\alpha =\pi^{-1} x$, $\beta =\pi^{-1} y$ and $u(\pi )=u\cdot \pi$ for $\pi$ some uniformizer of $\Z_p^{\mathrm{ur}} [p^{1/(p^2 -1)}]$, e.g. $\pi$ some $p^{1/(p^2 -1)}$. (Recall that $x=Z(\phi (1,0))$ and $y=Z(\phi (0,1))$,
for $Z$ some parameter of the formal group of the supersingular elliptic curve $E_s$ underlying $D_{i,s}$.) Moreover
$(x,y)\mapsto (x,y)\cdot \left( {a \atop c} {b\atop d} \right)$ under $\left( {a \atop c} {b\atop d} 
\right)$ in $\GL_2 (\F_p )$, whence the action of $\F_{p^2}^* \times \GL_2 (\F_p )$ given in the proposition. $\Box$
\end{proof}

\section{Regular models for non-split Cartan structures}
\label{NonSplitCartan}

\subsection{Regular model for $\overline{{\cal M}} ({\cal P}, \Gamma_{\mathrm{ns}} (p))$}

We first compute a regular model for modular curves associated with a non-split Cartan group $\Gamma_{\mathrm{ns}} (p)
\subseteq \GL_2 (\F_p )$ ({\em not} its normalizer for now), endowed with some additional level structure ${\cal P}$.
 
\begin{theo} 
\label{thNS} 
Let $p>3$ be a prime, and let $[\Gamma_{\mathrm{ns}} (p)]$ be the moduli problem over $\Z [1/p]$ associated with 
$\Gamma_{\mathrm{ns}} (p)$.
Let ${\cal P}$ be a representable moduli problem, which is finite \'etale over $({\mathrm{Ell}})_{/\Z_p}$  
(for instance ${\cal P} =[\Gamma (N) ]$  for some $N\ge 3$ not divisible by $p$). 
Let $\overline{{\cal M}} ({\cal P}, \Gamma_{\mathrm{ns}} (p)) =\overline{{\cal M}} ({\cal P}, \Gamma (p))/ 
\Gamma_{\mathrm{ns}} (p)$ be the associated compactified fine moduli space. 

   The curve $\overline{{\cal M}} ({\cal P}, \Gamma_{\mathrm{ns}} (p))$ has a regular model over $\Z_p^{\mathrm{ur}}$ 
whose special fiber is made of one vertical Igusa part, from which a horizontal chain of Drinfeld 
components rises at each supersingular point.

\medskip 
 
  That vertical part is a copy of $\overline{{\cal M}} ({\cal P} )_{\overline{\F}_p} $, with multiplicity $p-1$.  

\medskip

  If $s_{\cal P}$ is the number of supersingular points of $\overline{{\cal M}} ({\cal P})
(\overline{\F}_p )$, the $s_{\cal P}$ horizontal chains of (Drinfeld) components are all copies of branches 
of the following shape: a rational curve $\PPP^1$ with multiplicity $p+1$, on which, at two points, stems 
another $\PPP^1$ with multiplicity $1$. All singular points in the special fiber are regular normal crossings; more 
precisely, the completed local rings at the geometric former intersection points (between Igusa and Drinfeld parts) are   
$$
\Z_p^{\mathrm{ur}} [[X,Y]]/(X^{p+1} \cdot Y^{p-1} -p )
$$
and those at the latter two intersection points (purely on each Drinfeld part), are 
$$
\Z_p^{\mathrm{ur}} [[X,Y]]/(X^{p+1} \cdot Y -p ).
$$
See Figure~\ref{figureNS} below. 
%

%
%
\end{theo}

\begin{rema}
\label{bisrepetitas}
{\rm As in Remark~2.2 of~\cite{EdixP20}, we would have liked to call our ``vertical parts''  $\overline{{\cal M}} ({\cal P} 
)_{\overline{\F}_p} $ ``vertical components'', but were formally prevented from doing so because they may be not 
irreducible.} 
\end{rema}

\begin{figure}
\begin{center}
\begin{picture}(320,200)(-30,-80)


 
\qbezier(-50,60)(-35,62)(-22,65)
\qbezier(-50,62)(-35,66)(-22,65)

\qbezier(-22,65)(-2,69)(22,64)
\qbezier(-22,65)(-2,64)(22,64)

\qbezier(22,64)(33,64)(45,60)
\qbezier(22,64)(33,61)(45,58)

\qbezier(-50,61)(5,71)(45,59)

\put(-73,55){$D_s$}
\put(62,58){$p+1$}

\qbezier(-49,38)(-35,43)(-19,43)
\qbezier(-49,36)(-35,39)(-19,43)

\qbezier(-19,43)(0,49)(23,45)
\qbezier(-19,43)(0,44)(23,45)

\qbezier(23,45)(39,45)(46,41)
\qbezier(23,45)(32,42)(46,39)

\qbezier(-49,37)(5,52)(46,40)
\put(-73,35){$D_s$}
\put(62,35){$p+1$}

\qbezier(-48,-45)(-30,-41)(-18,-42)
\qbezier(-48,-47)(-28,-44)(-18,-42)

\qbezier(-18,-42)(9,-38)(35,-46)
\qbezier(-18,-42)(-0,-43)(35,-46)

\qbezier(35,-46)(45,-44)(59,-50)
\qbezier(35,-46)(45,-49)(59,-52)

\qbezier(-48,-46)(5,-36)(59,-51)
\put(-74,-45){$D_s$}
\put(67,-53){$p+1$}


\qbezier(0,105)(-5,0)(15,-70) 
\qbezier(1,105)(-4,0)(16,-70)
\qbezier(-1,105)(-6,0)(14,-70)

\put(-9,112){$p-1$}



\put(-20,-90){Galois quotient}


\put(-50,27){$\cdot$}
\put(-50,5){$\cdot$}
\put(-50,-17){$\cdot$}
\put(-50,-36){$\cdot$}

\put(50,27){$\cdot$}
\put(55,5){$\cdot$}
\put(57,-17){$\cdot$}
\put(60,-36){$\cdot$}


\put(150,6){\vector(-1,0){56}}

\put(107,-20){blow-up}


\qbezier(180,69)(245,76)(296,60)
\qbezier(180,68)(245,75)(295,59)
\qbezier(180,67)(245,74)(295,58)
\put(305,58){$D_s$}
\put(150,64){$p+1$}

\qbezier(182,36)(245,46)(295,31)
\qbezier(182,35)(245,45)(295,30)
\qbezier(182,34)(245,44)(295,29)
\put(305,27){$D_s$}
\put(150,32){$p+1$}

\qbezier(188,-44)(245,-34)(295,-44)
\qbezier(188,-45)(245,-35)(295,-45)
\qbezier(188,-46)(245,-36)(295,-46)
\put(305,-50){$D_s$}
\put(153,-49){$p+1$}

\put(222,27){$\cdot$}
\put(221,5){$\cdot$}
\put(222,-17){$\cdot$}
\put(224,-36){$\cdot$}


\qbezier(205,53)(205,65)(201,82)
\put(196,86){$1$}

\qbezier(202,23)(203,35)(198,52)
\put(192,54){$1$}

\qbezier(208,-63)(211,-45)(208,-22)
\put(202,-17){$1$}


\qbezier(275,53)(270,65)(271,82)
\put(266,86){$1$}

\qbezier(272,23)(266,35)(268,52)
\put(263,54){$1$}

\qbezier(275,-63)(263,-45)(268,-22)
\put(270,-17){$1$}


\put(232,112){$p-1$}

\qbezier(241,105)(236,0)(246,-70)
\qbezier(240,105)(235,0)(245,-70)
\qbezier(239,105)(234,0)(244,-70)

%


\put(215,-90){Regular model}


\end{picture}
\end{center}
\caption{Special fiber at $\overline{\F}_p$ of the regular model 
$\overline{{\cal M}} ({\cal P}, \Gamma_{\mathrm{ns}} (p))$}  
\label{figureNS}
\end{figure}
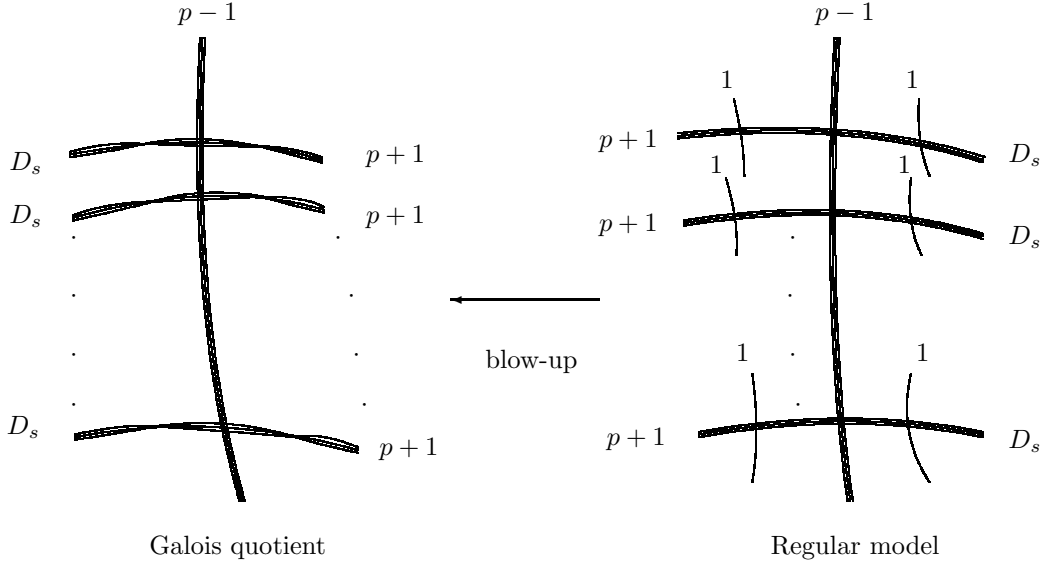

\begin{proof}
We start by considering the semistable model $\overline{{\cal M}} ({\cal P}, \Gamma_{\mathrm{ns}} (p))^{\mathrm{st}}$
over $\Z_p^{\mathrm{ur}} [\pi_0 ]$, for $\pi_0 :=p^{2/(p^2 -1)}$, 
given in Theorem~3.1 of~\cite{EdixP20}. So set $G := {\mathrm{Gal}} (\Q_p^{\mathrm{ur}} (\pi_0 )/
\Q_p^{\mathrm{ur}} ) \simeq \F_{p^2}^* /\{\pm 1\}$. 

\medskip

  First we determine the dual graph of the quotient arithmetic surface (and {\it not} its regularization yet). 
Proposition~\ref{stabilisateur} shows the Drinfeld components are stable under the action of Galois; as for the 
two Igusa parts ${\mathrm{Ig}}_i (p ,{\cal P}) \simeq \overline{\cal M} ({\cal P}, {\mathrm{Ig}}  (p)/\{\pm 
1\})_{\overline{\F}_p}$, $i=1$ or $d$, their stabilizer is the index $2$ subgroup of squares $\F_{p^2}^{*,2} /\{\pm 1\}$, 
and they are switched by the quotient $(\F_{p^2}^* /\{\pm 1\})/(\F_{p^2}^{*,2} /\{\pm 1\})$. 

\medskip

Now for the quotients of Igusa parts. The subgroup of Galois with order $(p+1)/2$, that is, $\mu_{p+1} (\F_{p^2}^* 
)/\{\pm 1\}$, in their stabilizer $(\F_{p^2}^{*,2} )/\{\pm 1\}$,  acts trivially on any of those Igusa parts, 
by Proposition~\ref{stabilisateur}. 
So taking the quotient by $\mu_{p+1} (\F_{p^2}^* )/\{\pm 1\}$ shows that over $\Z_p^{\mathrm{ur}} [p^{1/(p-1)}]$, 
$\overline{{\cal M}} ({\cal P}, \Gamma_{\mathrm{ns}} (p))$ has a special fiber with the
same Igusa parts $\overline{\cal M} ({\cal P}, {\mathrm{Ig}}  (p)/\{\pm 1\})_{\overline{\F}_p}$, with multiplicity
$1$. We then quotient out by the Galois group $\F_{p^2}^{*,2} /\mu_{p+1} (\F_{p^2} )$.
As non-zero homotheties belong to our Cartan subgroup, 
Proposition~\ref{stabilisateur} shows that the quotient of any of the two Igusa parts ${\mathrm{Ig}}_i (p ,
{\cal P})$, $i=1$ or $d$, of the semistable model (Theorem~3.1 of~\cite{EdixP20}) can be read as the quotient of 
$\overline{\cal M} ({\cal P}, {\mathrm{Ig}}  (p)/\{\pm 1\})_{\overline{\F}_p}$ by $\{ (v^2 , H_{v^{p+1}} ), 
v\in \F_{p^2}^* \}$, where for any $\lambda$ in $\F_p^*$, $H_{\lambda}$ is the homothety with ratio $\lambda$. 
Each such morphism $(v^2 , H_{v^{p+1}} )$ has modular interpretation 
\begin{eqnarray}
\label{diamond}
(E, \mathrm{a} ,P)\mapsto (E, \mathrm{a},\langle v^{p+1} \rangle_p P)
\end{eqnarray}
(Proposition~\ref{stabilisateur} again, and its proof). Therefore, ${\mathrm{Gal}} 
(\Q_p^{\mathrm{ur}} (p^{1/(p-1)} )/\Q_p^{\mathrm{ur}} (\sqrt{p}) ) \simeq\F_{p^2}^{*,2} /\mu_{p+1} (\F_{p^2} )$ 
acts generically freely on the ordinary locus of each Igusa part, with 
quotient a copy of $\overline{{\cal M}} ({\cal P})_{\overline{\F}_p}$. 
Looking at generic local rings of closed points on the arithmetic surface, we see that the Igusa parts occur with 
multiplicity $(p-1)/2$ in the closed fiber over $\Z_p^{\mathrm{ur}} [\sqrt{p}]$. 
Finally, the last step of order $2$, that is, taking the quotient under $\F_{p^2}^* /\F_{p^2}^{*,2}$, identifies the 
two Igusa parts, so that their multiplicity in the special fiber over $\Z_p^{\mathrm{ur}}$ is eventually $p-1$.

\medskip 

   Now we focus on Drinfeld components. Recall from Section 3.1 of~\cite{EdixP20} that one can choose parameters
$\tilde{\alpha}$, $\tilde{\beta}$ of the general Drinfeld component of the semistable $\overline{{\cal M}} ({\cal P} ,
[\Gamma (p)])^\sim$ such that it has model 
\begin{eqnarray}
\label{modeleDrX(p)}
\tilde{\alpha}^{p+1} -\tilde{\beta}^{p+1} =2A 
\end{eqnarray}
for $A=aN/2\ne 0$, with notation as in (18) in~\cite{EdixP20}. Identifying our non-split Cartan subgroup 
$\Gamma_{\mathrm{ns}} (p)$ mod $p$ with $\F_{p^2}^*$, one can choose a generator $\kappa$ of the elements with norm $1$ 
in the latter group such that  
%
$\kappa \colon (\tilde{\alpha}, \tilde{\beta} )\mapsto (\kappa^p \tilde{\alpha}, \kappa \tilde{\beta} )$. The 
coordinates $U:=\tilde{\alpha}^{p+1} -A$ and $V:=\tilde{\alpha} \tilde{\beta}$ are therefore clearly invariant under Galois
action, and as $[\overline{\F}_p ((\tilde{\alpha} ,\tilde{\beta})) \colon \overline{\F}_p ((U,V))]=p+1$, those $U$ and $V$
do generate the function field of the Drinfeld parts of the semistable $\overline{{\cal M}} ({\cal P}, \Gamma_{\mathrm{ns}} 
(p))$. Those Drinfeld parts therefore have models of shape 
\begin{eqnarray}
\label{DrinfeldXns}
U^2 =V^{p+1} +A^2  
\end{eqnarray}
(cf. (21) in~\cite{EdixP20}). Now Proposition~\ref{stabilisateur} above says that if $u$ belongs to $G={\mathrm{Gal}} 
(\Q_p^{\mathrm{ur}} (p^{2/(p^2 -1)} )/\Q_p^{\mathrm{ur}} )$, any $g$ in $\Gamma_{\mathrm{ns}} (p)$ with $N(u)=u^{p+1}\equiv 
\det g\mod p$ is such that $(u,g)$ stabilizes any Drinfeld component for $\overline{{\cal M}} ({\cal P} ,[\Gamma (p)])^\sim$. 
Therefore, the Galois action is induced, on the Drinfeld components $D_s$ of the quotient, by $u\colon (\tilde{\alpha} ,
\tilde{\beta} )\mapsto (u^{-1}g\cdot \tilde{\alpha} , u^{-1} g\cdot \tilde{\beta})$. In other words,  if $u$ is
a generator of $\F_{p^2}^*$, and $\kappa$ is in $\Gamma_{\mathrm{ns}} (p)$ with $\kappa (\tilde{\alpha} )=u^p 
\tilde{\alpha}$ and  $\kappa (\tilde{\beta})=u \tilde{\beta}$, then $u\circ \kappa (\tilde{\alpha} ,\tilde{\beta}) 
=(u^{p-1} \tilde{\alpha} ,\tilde{\beta})$. As for the Drinfeld parts of the semistable
$\overline{{\cal M}} ({\cal P}, \Gamma_{\mathrm{ns}} (p))$, the Galois action becomes 
\begin{eqnarray}
\label{actionDeGalois}
u\colon U\mapsto U, V\mapsto u^{p-1} V
\end{eqnarray}
on the parameters $U, V$ of the above model~(\ref{DrinfeldXns}).

  One therefore checks that the subgroup $\mu_{p -1} (\F_{p^2})  /\{\pm 1\} =\langle u^{p+1} \rangle  
/\{\pm 1\}$ of $\F_{p^2}^* /\{ \pm 1\}$ is the kernel of the Galois action on the Drinfeld parts in the 
fiber, so the special fiber of our descended curve over $\Z_p^{\mathrm{ur}} [p^{1/(p +1)}]$ has Drinfeld 
components with same model as above, and mutiplicity $1$. Then the quotient group $\F_{p^2}^* /\mu_{p -1} (
\F_{p})$, which is isomorphic to $\mu_{p+1} (\F_{p^2})$, acts generically freely on those Drinfeld part,
so the special fiber of the quotient modular curve over $\Z_p^{\mathrm{ur}}$ has Drinfeld components with
multiplicity $p+1$. Moreover, parameters for the Galois quotient of Drinfeld curves are now $X:=U$, $Y :=V^{p+1}$, so 
those quotients can be given equation
\begin{eqnarray}
\label{modelecon}
X^2 =Y +A^2 ;
\end{eqnarray}
that is, they are projective lines.

   As for intersection points between Drinfeld and Igusa components, their equation can be computed directly to be as in our 
theorem, which shows our curve is regular at those points (a fact which also follows from Serre's 
pseudo-reflection criterion; cf.~\cite{CEdSt03}, Theorem~2.3.9). The generic points on the descended Drinfeld 
components, on their side, have completed local rings of shape 
$$ 
\Z_p^{\mathrm{ur}} [[X,y]]/(y^{p+1} -p)
$$
hence are regular. However, regularity may fail at points of the Drinfeld curves coming from points which are fixed 
by Galois. Those are the points at infinity (which are exactly the intersection points with Igusa 
components - see Remark~\ref{al'infini} above), on one hand, which have just been discussed. And, on the 
other hand, they are the points at finite distance for which $V=0$ on the model~(\ref{DrinfeldXns}), i.e. the two points 
$(X,Y)=(\pm A,0)$ (on the model~(\ref{modelecon})) coming from  $(U,V)=(\pm A,0)$, which themselves come from the points on 
the model~(\ref{modeleDrX(p)})
$$
(\tilde{\alpha} ,\tilde{\beta} ) =(0, \zeta_{p+1}^k (-2A)^{\frac{1}{p+1}} ) \hspace{1cm} {\mathrm{and}} 
\hspace{1cm} (\tilde{\alpha} ,\tilde{\beta} ) =( \zeta_{p+1}^k 
(2A)^{\frac{1}{p+1}} ,0)
$$
(where $\zeta_{p+1}$ is some primitive $p+1^{\mathrm{st}}$-root of unity). Below those latter fixed points under Galois the 
situation demands blow-ups.

So set $S=\Z_p^{\mathrm{ur}} [\pi_1 ]$ for $\pi_1 =p^{1/(p +1)}$. 
At the two fixed points $(\pm A, 0)$ of~(\ref{DrinfeldXns}), $V$ 
can be chosen as a local parameter for the Drinfeld curve, on which Galois acts via 
$\lambda \colon V\mapsto \lambda V$ for $\lambda$ a generator of $\F_{p^2}^* /\F_p^*  \simeq \mu_{p+1} 
(\overline{\Q}_p )$. A straightforward adaptation of~\cite{CEdSt03}, Lemma~2.3.4 and its proof, shows that $V$ can 
be lifted to a parameter $x$ of the completed local rings $\Z_p^{\mathrm{ur}} [\pi_1 ][[x]]$ of our arithmetic 
surface at that point such that the action of Galois is given, for $\lambda$ in $\mu_{p+1}(\overline{\Q}_p )$, by 
\begin{eqnarray}
\label{laclassiquederegularisation}
\lambda \colon \pi_1 \mapsto \lambda \pi_1 , \hspace{1cm} x\mapsto \lambda x.
\end{eqnarray}
Therefore, the local ring of Galois invariants here has the same completion as 
$$
\bigoplus_{k\geq 0} \Z_p^{\mathrm{ur}} \, (\pi_1^{p+1-[k]} x^k )=\bigoplus_{k\geq 0} \Z_p^{\mathrm{ur}} p \pi_1^{-[k]} x^k 
$$ 
where $[k]$ denotes the integer within the range $\{1,p+1\}$ that is congruent to $k$ mod $p+1$. Setting 
$Y_k :=\pi_1^{p+1-k} x^k =p (x/\pi_1 )^k$, for $1\leq k\le p+1$, that ring can be written as
$$
A:=\Z_p^{\mathrm{ur}} [Y_1 ,\cdots , Y_{p+1}  ]/(Y_1^k -p^{k-1} Y_k )_{1\le k\le p+1} .
$$
The ring can be seen not to be regular at the maximal ideal ${\cal M}=(p, Y_1 , \dots ,Y_{p+1})$ (e.g. applying again
Serre's pseudo-reflections criterion to the ring $\Z_p^{\mathrm{ur}} [\pi_1 ][[x]]$
above), so we blow it up  
at $\cal M$. For $1\leq r\leq p+1$, we need a priori study $A_r :=A[ Y_r, p/Y_r, Y_k /Y_r ]_{1\le k\le p+1}$, 
that is, setting $Y_{k,r} :=Y_k /Y_r$ and ${\cal P}_r :=p/Y_r$,
$$
A_r =\Z_p^{\mathrm{ur}} [Y_{1,r} , Y_{2,r} ,\dots ,Y_{r-1,r} ,Y_r ,Y_{r+1,r} ,\dots ,Y_{p+1,r} , {\cal P}_r ].   
$$
We also define $A_0 :=A[Y_k /p ]_{1\le k\le p+1}$, so that, setting $Y_{k,0} :=Y_k /p$,
$$
A_0 =\Z_p^{\mathrm{ur}} [Y_{1,0}  , Y_{2,0} ,\dots ,Y_{p+1,0} ].
$$
As for the $A_r$ in the former set of rings one remarks that if $1<r<p+1$, one has $Y_{r+1,r} =x/\pi_1$, so 
that $Y_{k,r} =Y_{r+1,r}^{k-r}$ (respectively, $Y_{k,r} =Y_{r-1,r}^{r-k}$) if $k\geq r$ (respectively, $k\leq 
r$) and that ${\cal P}_r =Y_{r-1,r}^r$, $Y_r =pY_{r+1,r}^{r}$ and $Y_{r+1,r} \cdot Y_{r-1,r} =1$. 
So $A_r =\Z_p^{\mathrm{ur}} [Y_{r+1,r} ,  Y_{r+1,r}^{-1} ]$. One moreover checks that  $Y_{r+1,r} =Y_{r' +1,r'}$ for any 
$r, r'$ within the range $\{1,\dots ,p\}$, so that the $A_r$, $1< r< p+1$, are equal as subrings of the fraction 
field of $A$. When $r=1$, one still has $Y_{k,1} =Y_{2,1}^{k-1}$ for $1\le k\le p+1$,  ${\cal P}_1 Y_{2,1}=1$ 
and $Y_1 =pY_{2,1}$, so that eventually the $p$ previous rings $A_r$, $1\le r< p+1$, are all equal, setting $X:= \pi_1 /x$, 
to 
\begin{eqnarray}
\label{voiciA1}
A_1 =\Z_p^{\mathrm{ur}} [{\cal P}_1 ,Y_{2,1}] =  \Z_p^{\mathrm{ur}} [X, X^{-1} ]. 
\end{eqnarray}
When $r=p+1$, one has $Y_{k,p+1} =Y_{p,p+1}^{p+1-k}$ for $1\le k\le p+1$ and  ${\cal P}_{p+1} =Y_{p,p+1}^{p+1}$, so 
that $Y_{p,p+1}^{p+1} \cdot Y_{p+1} =p$. Also note that, with notation as in~(\ref{voiciA1}) above, 
$Y_{p,p+1} =Y_{2,1}^{-1} =X$; therefore, setting now $Y:=Y_{p+1}$:
$$
A_{p+1} \simeq \Z_p^{\mathrm{ur}} [X,Y]/(X^{p+1} \cdot Y -p ).
$$
As for $A_0$, finally, one has $Y_{1,0}^k =Y_{k,0}$ for all $1\le k\le p+1$, so that 
$$
A_0 \simeq \Z_p^{\mathrm{ur}} [Z],
$$
where, with notation as in~(\ref{voiciA1}), $Z=Y_{1,0} =Y_1 /p=Y_{2,1} = x/\pi_1 =1/X$.
   Our first blow-up of $A$ is therefore already regular and has the simple shape described in Figure~\ref{figureNS}.  This completes the proof of Theorem~\ref{thNS}.  $\Box$
\end{proof}

\subsection{Regular model for $\overline{{\cal M}} ({\cal P}, \Gamma_{\mathrm{ns}}^+ (p))$}

\begin{theo} 
\label{thNS+} 
Let $p>3$ be a prime, and let $[\Gamma_{\mathrm{ns}}^+ (p)]$ be the moduli problem over $\Z [1/p]$ 
associated with $\Gamma_{\mathrm{ns}}^+ (p)$.
Let ${\cal P}$ be a representable moduli problem, which is finite \'etale over $({\mathrm{Ell}})_{/\Z_p}$  as in 
Theorem~\ref{thNS}.  
Let $\overline{{\cal M}} ({\cal P}, \Gamma_{\mathrm{ns}}^+ (p)) =\overline{{\cal M}} ({\cal P}, \Gamma (p))/ 
\Gamma_{\mathrm{ns}}^+ (p)$ be the associated compactified fine moduli space. 
%

   Then $\overline{{\cal M}} ({\cal P}, \Gamma_{\mathrm{ns}}^+ (p))$ has a regular model over $\Z_p^{\mathrm{ur}}$ 
whose special fiber is made of one vertical Igusa part, from which grow horizontal chains of Drinfeld 
components above each supersingular point.


\medskip 
 
  That vertical part is a copy of $\overline{{\cal M}} ({\cal P} )_{\overline{\F}_p} $, with multiplicity $(p-1)/2$.  
%
%

\medskip

  If $s_{\cal P}$ is the number of supersingular points of $\overline{{\cal M}} ({\cal P})
(\overline{\F}_p)$, the $s_{\cal P}$ horizontal chains of (Drinfeld) components are all copies of branches of the 
following shape: a rational curve $\PPP^1$ with multiplicity $p+1$, on which, at two points, another $\PPP^1$ stands up, 
one with multiplicity $1$, and one with multiplicity $(p+1)/2$. All singularities in the special fiber which are 
intersection points between Igusa and Drinfeld parts have completed local rings with same completion as
$$
\Z_p^{\mathrm{ur}} [[X,Y]]/(X^{p+1} \cdot Y^{\frac{p-1}{2}} -p );
$$
as for those at the exceptional divisors on the Drinfeld parts, they are
$$
\Z_p^{\mathrm{ur}} [[X,Y]]/(X^{p+1} \cdot Y -p )
\hspace{0.5cm} and \hspace{0.5cm}
\Z_p^{\mathrm{ur}} [[X,Y]]/(X^{p+1} \cdot Y^{\frac{p+1}{2}} -p );
$$
see Figure~\ref{figureNS+} below. 
%

%
%
\end{theo}

\begin{proof}

Start from the semistable model $\overline{{\cal M}} ({\cal P}, \Gamma_{\mathrm{ns}}^+ (p))^{\mathrm{st}}$ 
of Theorem~3.3 in~\cite{EdixP20}, over $\Z_p^{\mathrm{ur}} [p^{2/(p^2 -1)}]$. 
As in the proof of Theorem~\ref{thNS}, one sees that the graph of the special fiber of the quotient scheme by 
Galois remains unchanged over $\Z_p^{\mathrm{ur}} [\sqrt{p}]$. Then over $\Z_p^{\mathrm{ur}}$, there is only 
one ``vertical'', Igusa part, crossed at supersingular points by Drinfeld components (whether $p\equiv 1$ or $-1$
mod $4$). 

\medskip

   If $p\equiv -1$ mod $4$, the only Igusa part of $\overline{{\cal M}} ({\cal P}, \Gamma_{\mathrm{ns}}^+ (p))^{\mathrm{st}}$ 
over $\Z_p^{\mathrm{ur}} [p^{2/(p^2 -1)}]$ is $\overline{\cal M} 
({\cal P}, {\mathrm{Ig}}  (p)/\{\pm 1\})_{\overline{\F}_p }$. With notation as in~\cite{EdixP20}, Theorem~3.3,
one can assume that the action on $\F_p^2$ of our non-split Cartan subgroup has a basis of $\F_{p^2} /
\F_p$-conjugate eigenvectors $P_1$, $P_2$, with respect to which the normalizer of Cartan can be written as
$$
\left\{ \left( {r\atop 0}\ {0\atop r^p }\right) , w_r :=\left( {0\atop r^p}\ {r \atop 0 }\right) , r\in \F_{p^2}^*
\right\} .
$$
So writing $P$ for the $\F_p$-line spanned in $\F_p^2$ by $P_1 +P_2$, one might consider a representative 
${\mathrm{Ig}}_{i,P}$ in $\overline{{\cal M}} ({\cal P}, \Gamma (p))_{\overline{\F}_p}$ above $\overline{\cal M} ({\cal P}, {\mathrm{Ig}}  
(p)/\{\pm 1\})_{\overline{\F}_p }$, which shows that in our normalizer-of-non-split-Cartan curve, the Igusa part
is acted on trivially exactly by the Galois subgroup $\mu_{2(p+1)} (\F_{p^2})/\{ \pm 1\}$ of $\F_{p^2}^* /\{ \pm 
1\}$. (Indeed, just take the quotient by the subgroup $\langle (u^{(p-1)/2} , \left( {0\atop 1 }\ {1 \atop 0 }\
\right) )\rangle$ (for $u$ a generator of $\F_{p^2}^*$) of $\F_{p^2}^* \times \Gamma_{\mathrm{ns}}^+ (p)$, and
apply Proposition~\ref{stabilisateur}.) Taking the quotient by this Galois subgroup shows that over $\Z_p^{\mathrm{ur}} 
[p^{2/(p -1)}]$, $\overline{{\cal M}} ({\cal P}, \Gamma_{\mathrm{ns}}^+ (p))$ has a
special fiber with two Igusa parts which are still isomorphic to $\overline{\cal M} ({\cal P}, {\mathrm{Ig}}  
(p)/\{\pm 1\})_{\overline{\F}_p }$, with multiplicity $1$. Now the quotient Galois group $\F_{p^2}^* /\mu_{2(p+1)} 
(\F_{p^2})$ has {\it odd} order $(p-1)/2$; therefore, any of its elements can be given a representative $v^2$ in 
$\F_{p^2}^{*,2}$, so that it acts as $(v^2 , H_{v^{p+1}} )$, which is the diamond operator $\langle v^{p+1} \rangle_p$, on 
our ${\mathrm{Ig}}_{i,P}$ (notation as for~(\ref{diamond}), proof of Theorem~\ref{thNS}). This eventually shows that, over 
$\Z_p^{\mathrm{ur}}$, $\overline{{\cal M}} ({\cal P}, \Gamma_{\mathrm{ns}}^+ (p))_{\overline{\F}_p}$ has an Igusa 
part which is isomorphic to $\overline{\cal M} ({\cal P} )_{\overline{\F}_p }$, with multiplicity $(p-1)/2$.

\medskip

  If $p\equiv 1$ mod $4$, each of the two Igusa parts $\overline{\cal M} ({\cal P}, {\mathrm{Ig}}  (p)/C_4 
)_{\overline{\F}_p }$ of the special fiber of $\overline{{\cal M}} ({\cal P}, \Gamma_{\mathrm{ns}}^+ (p))^{\mathrm{st}}$ 
over $\Z_p^{\mathrm{ur}} [p^{2/(p^2 -1)}]$, has Galois stabilizer 
the index $2$ subgroup $\F_{p^2}^{*,2} /\{ \pm 1\}$ of squares in $\F_{p^2}^* /\{ \pm 1\}$. Inside of that,
the subgroup with trivial action can be seen to be $\mu_{2(p+1)} (\F_{p^2})/\{ \pm 1\}$, in exactly the same fashion
as for the $p\equiv -1$ mod $4$ case. Therefore, taking the quotient by this Galois subgroup shows that
over $\Z_p^{\mathrm{ur}} [p^{2/(p -1)}]$, $\overline{{\cal M}} ({\cal P}, \Gamma_{\mathrm{ns}}^+ (p))$ has a
special fiber with two Igusa parts which are still isomorphic to $\overline{\cal M} ({\cal P}, {\mathrm{Ig}}  
(p)/C_4 )_{\overline{\F}_p }$, with multiplicity $1$. The quotient Galois group $\F_{p^2}^{*,2} /\mu_{2(p+1)} 
(\F_{p^2})$, with order $(p-1)/4$, again acts as $\{ (v^2 , H_{v^{p+1}} ), v\in \F_{p^2}^* \}$, hence generically
faithfully on the Igusa part. Thus over $\Z_p^{\mathrm{ur}} [\sqrt{p}]$ we obtain a special fiber whose Igusa parts are 
copies of $\overline{\cal M} ({\cal P} )_{\overline{\F}_p }$, with multiplicity 
$(p-1)/4$. As before, those copies are switched when landing to $\Z_p^{\mathrm{ur}}$, whence the final multiplicity
$(p-1)/2$, as in the $p\equiv 1$ mod $4$ case. 

\begin{figure}
\begin{center}
\begin{picture}(320,200)(-30,-80)


 
\qbezier(-50,60)(-35,62)(-22,65)
\qbezier(-50,62)(-35,66)(-22,65)

\qbezier(-22,65)(-2,69)(22,64)
\qbezier(-22,65)(-2,64)(22,64)

\qbezier(22,64)(33,64)(45,60)
\qbezier(22,64)(33,61)(45,58)

\qbezier(-50,61)(5,71)(45,59)

\put(-73,55){$D_s^+$}
\put(62,58){$p+1$}

\qbezier(-49,38)(-35,43)(-19,43)
\qbezier(-49,36)(-35,39)(-19,43)

\qbezier(-19,43)(0,49)(23,45)
\qbezier(-19,43)(0,44)(23,45)

\qbezier(23,45)(39,45)(46,41)
\qbezier(23,45)(32,42)(46,39)

\qbezier(-49,37)(5,52)(46,40)
\put(-73,35){$D_s^+$}
\put(62,35){$p+1$}

\qbezier(-48,-45)(-30,-41)(-18,-42)
\qbezier(-48,-47)(-28,-44)(-18,-42)

\qbezier(-18,-42)(9,-38)(35,-46)
\qbezier(-18,-42)(-0,-43)(35,-46)

\qbezier(35,-46)(45,-44)(59,-50)
\qbezier(35,-46)(45,-49)(59,-52)

\qbezier(-48,-46)(5,-36)(59,-51)
\put(-74,-45){$D_s^+$}
\put(67,-53){$p+1$}


\qbezier(0,105)(-5,0)(15,-70) 
\qbezier(1,105)(-4,0)(16,-70)
\qbezier(-1,105)(-6,0)(14,-70)

\put(-15,112){$(p-1)/2$}


\put(-20,-95){Galois quotient}


\put(-50,27){$\cdot$}
\put(-50,5){$\cdot$}
\put(-50,-17){$\cdot$}
\put(-50,-36){$\cdot$}

\put(50,27){$\cdot$}
\put(55,5){$\cdot$}
\put(57,-17){$\cdot$}
\put(60,-36){$\cdot$}


\put(150,6){\vector(-1,0){56}}

\put(107,-20){blow-up}


\qbezier(180,69)(245,76)(296,60)
\qbezier(180,68)(245,75)(295,59)
\qbezier(180,67)(245,74)(295,58)
\put(305,58){$D_s^+$}
\put(150,64){$p+1$}

\qbezier(182,36)(245,46)(295,31)
\qbezier(182,35)(245,45)(295,30)
\qbezier(182,34)(245,44)(295,29)
\put(305,27){$D_s^+$}
\put(150,32){$p+1$}

\qbezier(188,-44)(245,-34)(295,-44)
\qbezier(188,-45)(245,-35)(295,-45)
\qbezier(188,-46)(245,-36)(295,-46)
\put(305,-50){$D_s^+$}
\put(153,-49){$p+1$}

\put(222,27){$\cdot$}
\put(221,5){$\cdot$}
\put(222,-17){$\cdot$}
\put(224,-36){$\cdot$}


\qbezier(205,53)(205,65)(201,82)
\put(196,86){$1$}

\qbezier(202,23)(203,35)(198,52)
\put(192,54){$1$}

\qbezier(208,-63)(211,-45)(208,-22)
\put(202,-17){$1$}


\qbezier(276,53)(271,65)(272,82)
\qbezier(275,53)(270,65)(271,82)
\put(267,86){$\frac{p+1}{2}$}

\qbezier(273,23)(267,35)(269,52)
\qbezier(272,23)(266,35)(268,52)
\put(252,53){$\frac{p+1}{2}$}

\qbezier(276,-63)(264,-45)(269,-22)
\qbezier(275,-63)(263,-45)(268,-22)
\put(270,-17){$\frac{p+1}{2}$}


\put(226,112){$(p-1)/2$}

\qbezier(241,105)(236,0)(246,-70)
\qbezier(240,105)(235,0)(245,-70)
\qbezier(239,105)(234,0)(244,-70)

%


\put(215,-95){Regular model}


\end{picture}
\end{center}
\caption{Special fiber at $\overline{\F}_p$ of the regular model $\overline{{\cal M}} ({\cal P}, 
\Gamma_{\mathrm{ns}}^+ (p))$}  
\label{figureNS+}
\end{figure}  
  
\bigskip

    Now for the Drinfeld part. With notation as in proof of Theorem~\ref{thNS} above, we know from Theorem~3.3 
of~\cite{EdixP20} that the Drinfeld components of $\overline{{\cal M}} ({\cal P},\Gamma_{\mathrm{ns}}^+ (p))^{\mathrm{st}}$ 
over $\Z_p^{\mathrm{ur}} [p^{2/(p^2 -1)}]$ have equation
\begin{eqnarray}
\label{DrinfeldNS+1}
Y^2 =X( X^{\frac{p+1}{2}} +\left( \frac{aN}{2}\right)^2 ) 
\end{eqnarray}
for $X=V^2 =(\tilde{\alpha} \tilde{\beta})^2$ and $Y=U\times V =\tilde{\alpha} \tilde{\beta}(\tilde{\alpha}^{p+1} 
-aN/2 )$, with Galois action $u \colon 
X\mapsto u^{2(p-1)} X , Y\mapsto u^{p-1} Y$ if $u\in \F_{p^2}^* /\{\pm 1\}$ (see~(\ref{actionDeGalois})). The 
Drinfeld parts therefore descend to the same model~(\ref{DrinfeldNS+1}), with multiplicity $1$, over $\Z_p^{\mathrm{ur}} 
[p^{1/(p+1)}]$. Then the Galois group acts generically faithfully on closed points in the special fiber,
so that the quotient scheme has special fiber some projective line with multiplicity $p+1$. Indeed parameters for the 
quotient scheme in the special fiber can be chosen as $R:= X^{\frac{p+1}{2}}$ and $S:=Y^2 /X$ (whence a trivial model 
$S=R+(\frac{aN}{2})^2$). The only closed points on the model~(\ref{DrinfeldNS+1}) on which the 
Galois action has inertia are either those at infinity (which correspond to intersection point(s) with Igusa 
part(s)) or points for which  $Y=0$, that is,
$$
(X,Y) =(0,0) \hspace{1cm} {\mathrm{or}}  \hspace{1cm} (X,Y)=( \left( \frac{p+1}{2}\right)^{\mathrm{nd}} 
{\mathrm{roots\ of\ }} -\left( \frac{aN}{2}\right)^2 ,0) ,
$$
which give, on the model $S=R+\left( \frac{aN}{2}\right)^2$ of the Drinfeld components downstairs (over $\Z_p^{\mathrm{ur}}$),
$$
(R,S)=(0, (\frac{aN}{2})^2 )     \hspace{1cm} {\mathrm{or}}  \hspace{1cm} (R,S)=(-(\frac{aN}{2})^2 ,0).
$$
At that set of points, Serre's pseudo-reflection criterion (Theorem~2.3.9 of~\cite{CEdSt03}) implies the quotient 
is not regular. 

\medskip
 
   (Here we may remark that if $p\equiv 1$ mod $4$, then (\ref{DrinfeldNS+1}) defines a hyperelliptic equation of 
type ``$Y^2 =X^{2g+2} +\cdots $'', which has two points at infinity: they correspond exactly to the intersection with the 
two Igusa parts on $\overline{{\cal M}} ({\cal P},\Gamma_{\mathrm{ns}}^+ (p))^{\mathrm{st}}$ over $\Z_p^{\mathrm{ur}} 
[p^{2/(p^2 -1)}]$. Whereas if $p\equiv -1$ mod $4$, then (\ref{DrinfeldNS+1}) defines an equation of type ``$Y^2 =X^{2g+1} +
\cdots $'', which has one points at infinity, and indeed in 
that case there is only one Igusa part intersecting Drinfeld. That numerology shows up again, in a slightly more 
subtle guise, in the proof of the split case 
(Theorem~\ref{thS+}) below). 

\medskip

     Let us complete the study of the above singular points (at finite distance). At the first one, $(X,Y) =(0,0)$, 
the completed local ring can be expressed as series in $Y$ and $p^{1/(p+1)}$, and the same computation 
as in the proof of Theorem~\ref{thNS} above (see (\ref{laclassiquederegularisation}) and what follows) 
shows that blowing up makes a $\PPP^1_{\overline{\F}_p}$ with multiplicity $1$ pop up. That is enough to regularize the 
situation there.  

    On the set of $(p+1)/2$ points of shape $(X,Y)=((\frac{p+1}{2})^{\mathrm{nd}} {\mathrm{\ roots\ of\ }} -
(\frac{aN}{2})^2 ,0)$, Galois acts transitively. Set $\pi_1 :=p^{1/(p+1)}$, and descend  the 
model~(\ref{DrinfeldNS+1}) to $\Z [\pi_1 ]$. Then the quotient Galois group $\F_{p^2}^* /(\F_{p^2}^* )^{p+1} 
\simeq \mu_{p+1} (\F_{p^2})$, which acts generically freely on the Drinfeld component, acts with inertia its 
order $2$ subgroup at each such point $((\frac{p+1}{2})^{\mathrm{nd}} {\mathrm{\ roots\ of\ }} 
-(\frac{aN}{2})^2 ,0)$. If $\varepsilon$ is a generator of that order $2$ subgroup and $x$ is a good lift of 
parameter of the ring mod $\pi_1$ (in the sense of Lemma~2.3.4 of~\cite{CEdSt03}, 
cf.~(\ref{laclassiquederegularisation}) above), one has 
$$
\varepsilon \colon \pi_1 \mapsto -\pi_1 , x\mapsto -x,
$$ 
and the ring of Galois invariants of the completed local ring is therefore the same as that of
$$
A =\bigoplus_{k\ge 0} \Z_p^{\mathrm{ur}} [\pi_1^2 ] x^{2k}     \bigoplus_{k\ge 0} \Z_p^{\mathrm{ur}} [\pi_1^2 ] 
\pi_1 x^{2k+1}  \simeq \Z_p^{\mathrm{ur}} [t,X,Y] /( t^{\frac{p+1}{2}} -p , X^2 -tY). 
$$ 
Once more, in a similar fashion as in the proof of Theorem~\ref{thNS} ((\ref{laclassiquederegularisation}) and around), one 
sees that $A$ is singular, but blowing up once at ${\cal M}=(t,X,Y)$ crafts a scheme
which is regular, whose special fiber is the union of  the special fiber of the Drinfeld part (projective 
line with multiplicity $p+1$) with another $\PPP^1_{\overline{\F}_p}$ having multiplicity $(p+1)/2$, which intersects 
it. The intersection point has ring of coordinates
$$
\Z_p^{\mathrm{ur}} [a,b]/(a^{p+1} b^{\frac{p+1}{2}} -p). \hspace{2cm} 
$$ 
As for the singular points at the intersection of Drinfeld and Igusa parts, note that if $p\equiv 1$ mod $4$, they are 
{\em{not}} regular in the semistable model over $\Z_p^{\mathrm{ur}} [p^{2/(p^2 -1)}]$ of Theorem~3.3 in~\cite{EdixP20}. 
However, their image in the Galois quotient is regular over $\Z_p^{\mathrm{ur}}$ (or even over the quadratic subextension of  
$\Z_p^{\mathrm{ur}} [p^{2/(p^2 -1)}]$). $\Box$   
%
 
\end{proof}

\subsection{The coarse case: $X_{\mathrm{ns}} (p)$}

\begin{theo} 
\label{thNSgrossier} 
For $p\geq 13$ a prime, let $X_{\mathrm{ns}} (p)$ be the coarse modular curve over $\Q$ associated with 
$\Gamma_{\mathrm{ns}} (p)$. 
%
%

   The {\em minimal regular model with normal crossings} of the curve $X_{\mathrm{ns}} (p)$ over $\Z_p^{\mathrm{ur}}$ has a 
special fiber which is made of one vertical Igusa component, which intersects horizontal chains of Drinfeld components 
above each supersingular point.



\medskip 
 
  The vertical part is a (smooth) projective line ($j$-line), with multiplicity $p-1$.  
%

\medskip

  At each supersingular point of that Igusa component, 
unless the corresponding $j$-invariant has extra geometric automorphisms (that 
is, $j\equiv 0$ or $1728 \mod p$), the crossing Drinfeld component is a copy of  a rational curve $\PPP^1$ with 
multiplicity $p+1$, on which,
at two points, another $\PPP^1$ springs up with multiplicity $1$. All intersection points in the special fiber are regular 
normal crossings; more precisely, the completed local rings at the former intersection points (between Igusa and those 
Drinfeld components) are   
$$
\Z_p^{\mathrm{ur}} [[X,Y]]/(X^{p+1} \cdot Y^{p-1} -p ),
$$
and those at the latter two intersection points (purely on each Drinfeld part) are
$$
\Z_p^{\mathrm{ur}} [[X,Y]]/(X^{p+1} \cdot Y -p ).
$$
Moreover, the following hold:
\begin{itemize}
\item {\bf If $p\equiv 1$ mod $12$}, the vertical Igusa component intersects transversally an extra component at 
each of the two (ordinary) $j$-invariants $1728$ and 
$0$, which are a projective line with multiplicity $(p-1)/2$ (ring at the intersection point: $\Z_p^{\mathrm{ur}} 
[[a,b]]/
(a^{p-1} b^{\frac{p-1}{2}} -p)$) and a projective line with 
multiplicty $(p-1)/3$ (ring at the intersection point: $\Z_p^{\mathrm{ur}} [[a,b]]/(a^{p-1} b^{\frac{p-1}{3}} -p)$), 
respectively. 
\item  {\bf If $p\equiv 11$ mod $12$}, the Drinfeld components above the supersingular $j\equiv 1728$ and $j\equiv 0 
\mod p$ are both projective lines, with 
multiplicity $(p+1)/2$ and $(p+1)/3$, respectively, both endowed with two extra components which are copies of $\PPP^1$ with 
multiplicity $1$
(intersection local rings: $\Z_p^{\mathrm{ur}} [[a,b]]/(a^{\frac{p+1}{2}} b -p)$ and $\Z_p^{\mathrm{ur}} [[a,b]]/
(a^{\frac{p+1}{3}} b-p)$ respectively). The rings of the intersection between the Igusa component and the special Drinfeld 
components are
$$
\Z_p^{\mathrm{ur}} [[a,b]]/(a^{p-1} b^{\frac{p+1}{2}} -p)
$$ 
and
$$
\Z_p^{\mathrm{ur}} [[a,b]]/(a^{p-1} b^{\frac{p+1}{3}} -p),
$$ 
respectively.
\item  {\bf If $p\equiv 5$ mod $12$}, we have the relevant mix between the above two situations: $j\equiv 1728$ is 
ordinary ${\mathrm{mod}}\ p$, so see the case $p\equiv 1\mod 12$, whereas $j\equiv 0$ is supersingular ${\mathrm{mod}}\ p$, 
so see the case $p\equiv 11\mod 12$.
\item  {\bf If $p\equiv 7$ mod $12$}, the situation is mutatis mutandis the same as before: this time $j\equiv 1728$ is supersingular ${\mathrm{mod}}\ p$, and $j\equiv 0$ is ordinary ${\mathrm{mod}}\ p$.
\end{itemize}
%
%

\medskip
  
  See Figure~\ref{figureNSgrossier} below. 
\end{theo}

\begin{rema}
\label{petitsminimauxNsplit}
{\rm{Theorem~\ref{thNSgrossier} actually gives regular models for all $p\geq 5$, but we have excluded primes $p\leq 11$ 
from its statement because of the non-minimality of our regular model with normal crossings in those cases. However, 
$X_{\mathrm{ns}} (p)$ is just a projective line if $p\leq 5$ (including $p=2, 3$), and for $p=7$ and $p=11$, the minimal 
regular model with normal crossings is obtained by contracting the central Igusa component (with multiplicity $(p-1)$).}} 
\end{rema}

\begin{proof}
As before we start from the semistable model $X_{\mathrm{ns}} (p)^{\mathrm{st}}$ of the coarse curve over 
$\Z_p^{\mathrm{ur}} [p^{2/(p^2 -1)}]$ found in Theorem~3.4 of~\cite{EdixP20}
and quotient out by the Galois group $G ={\mathrm{Gal}} (\Q_p^{\mathrm{ur}} (p^{2/(p^2 -1)})/\Q_p^{\mathrm{ur}})
\simeq \F_{p^2}^* /\{ \pm 1\}$. 
The only new feature with respect to the above Theorem~\ref{thNS} 
appears at points with extra automorphisms, that is, corresponding to isomorphism classes of elliptic curves $E$ 
with $j$-invariant $1728$ (where 
${\mathrm{Aut}}_{{\overline{\F}}_p} (E) \simeq \Z /4\Z$) or $j$-invariant $0$ (where ${\mathrm{Aut}}_{{\overline{\F}}_p} (E) 
\simeq \Z /6\Z$). In each situation, one needs to distinguish between the cases where $E$ is ordinary or supersingular. 

\medskip

    So start with the case $j= 1728$ and assume it is ordinary at $p$ ($\equiv 1$ mod $4$). Elliptic curves 
with $j\equiv 1728$ define $(p-1)/4$ points on each of the two Igusa vertical components $\overline{M} ({\mathrm{Ig}} 
(p)/\{\pm 1\} )_{\overline{\F}_p} $ in the special fiber of the semistable $X_{\mathrm{ns}} (p)^{\mathrm{st}}$. First taking 
the Galois quotient under $\mu_{p+1} (\F_{p^2} )/\{ \pm 1\}$ gives a model whose Igusa components remain the same over 
the base $\Z_p^{\mathrm{ur}} [\pi_2 := p^{1/(p -1)}]$.
The local rings at those special points with $j=1728$ have the same completion (at the closed point $(x,\pi_2 )$) as
\begin{eqnarray}
\label{voiciPlain}
A :=\Z_p^{\mathrm{ur}} [\pi_2 ] [x] .
\end{eqnarray}
Then the Galois subquotient $G_1 :={\mathrm{Gal}} (\Q_p^{\mathrm{ur}} (p^{1/(p-1)} )/\Q_p^{\mathrm{ur}} 
(\sqrt{p})) \simeq \F_{p^2}^{*,2} /\mu_{p+1} (\F_{p^2} ) \simeq \F_p^{*,2}$ of 
index $2$ (which is the stabilizer of each Igusa component) acts transitively as 
$$
\F_{p^2}^{*,2} \ni v^2 \colon (E,  P)\mapsto (E,  \langle v^{p+1} \rangle_p P)
$$
(see (\ref{diamond})) on our set of $(p-1)/4$ special 
points for which $j\equiv 1728$ in the special fiber. In particular, taking $v:=u^{(p-1)/4}$ (for $u$ some generator of
$\F_{p^2}^*$), which is a lift of 
a fourth root of unity in $\F_p^*$, we see that the stabilizer in $G_1$ of those special points is the 
order $2$ subgroup $G_1 [2]=:\langle \varepsilon \rangle$, spanned by the order 2 element $\varepsilon =v^{2(p+1)}
= -1$ in $\F_p^{*,2} \simeq G_1$. Again $x$ can be chosen (Lemma~2.3.4 of~\cite{CEdSt03}) so that  
$$
\varepsilon \colon \pi_2 \mapsto -\pi_2 , x\mapsto -x.
$$ 
We therefore consider the ring of Galois invariants:
$$
A^{\langle \varepsilon \rangle} =\bigoplus_{k\ge 0} \Z_p^{\mathrm{ur}} [\pi_2^2 ] x^{2k}     \bigoplus_{k\ge 0} 
\Z_p^{\mathrm{ur}} 
[\pi_2^2 ] \pi_2 x^{2k+1}  \simeq \Z_p^{\mathrm{ur}} [t,X,Y] /( 
t^{\frac{p-1}{2}} -p , X^2 -tY). 
$$ 
In a similar fashion as in the proof of Theorem~\ref{thNS} (cf. also the last lines of that of Theorem~\ref{thNS+}), 
one sees that $A$ is not regular, but 
blowing up once at ${\cal M}=(t,X,Y)$ produces a scheme over $\Z_p^{\mathrm{ur}} [p^{2/(p-1)}]$ which is (regular). 
Finally, taking the full Galois quotient (therefore
identifying the two Igusa parts), we obtain a scheme whose special fiber is the intersection of  the special fiber of 
the single Igusa component ($j$-line with multiplicity $p-1$) with another $\PPP^1_{\overline{\F}_p}$ with 
multiplicity $(p-1)/2$; the intersection point has ring of coordinates
$$
\Z_p^{\mathrm{ur}} [[a,b]]/(a^{p-1} b^{\frac{p-1}{2}} -p).
$$ 
%
%
%
%
%

\begin{figure}
\begin{center}

\begin{picture}(320,410)(-30,-90)

\qbezier(-50,280)(5,290)(45,278)
\qbezier(-50,281)(5,291)(45,279)
\qbezier(-50,282)(5,292)(45,280)

\put(-73,275){\scriptsize$D_2$}
\put(54,275){\scriptsize$p+1$}

\qbezier(-35,293)(-32,283)(-34,270)
\put(-37,296){\scriptsize$1$}
\put(-33,288){\scriptsize$E_2$}

\qbezier(20,301)(18,283)(21,272)
\put(21,302){\scriptsize$1$}
\put(22,292){\scriptsize$F_2$}

\qbezier(-49,306)(5,321)(46,309)
\qbezier(-49,307)(5,322)(46,310)
\qbezier(-49,308)(5,323)(46,311)

\put(-73,305){\scriptsize$D_1$}
\put(53,305){\scriptsize$p+1$}

\qbezier(-27,320)(-24,308)(-26,295)
\put(-29,323){\scriptsize$1$}
\put(-25,317){\scriptsize$E_1$}

\qbezier(30,322)(28,316)(31,295)
\put(30,323){\scriptsize$1$}
\put(33,317){\scriptsize$F_1$}

\qbezier(-48,217)(5,233)(59,218)
\qbezier(-48,216)(5,232)(59,217)
\qbezier(-48,215)(5,231)(59,216)

\put(-71,215){\scriptsize$D_k$}
\put(64,217){\scriptsize$p+1$}

\qbezier(-25,245)(-23,224)(-26,196)
\put(-24,248){\scriptsize$1$}
\put(-21,240){\scriptsize$E_k$}

\qbezier(27,245)(24,222)(26,212)
\put(28,250){\scriptsize$1$}
\put(31,242){\scriptsize$F_k$}


\qbezier(0,393)(-5,270)(-2,200) 
\qbezier(1,393)(-4,270)(-1,200)
\qbezier(-1,393)(-6,270)(-3,200)

\put(-9,398){\scriptsize$p-1$}
\put(-10,388){\scriptsize$A$}



\put(16,384){\small $j\equiv 1728$}
\put(15,385){\vector(-2,-1){12}}
\qbezier(-25,378)(5,376)(29,379)
\qbezier(-25,377)(5,375)(29,378)
\put(34,371){$\frac{p-1}{2}$}
\put(-35,375){\scriptsize$B$}


\put(20,340){\small $j\equiv 0$}
\put(17,342){\vector(-1,1){13}}
\qbezier(-25,358)(5,357)(29,359)
\qbezier(-25,357)(5,356)(29,358)
\put(34,353){$\frac{p-1}{3}$}
\put(-35,355){\scriptsize$C$}


\put(-45,+270){$\cdot$}
\put(-45,253){$\cdot$}
\put(-45,234){$\cdot$}

\put(42,266){$\cdot$}
\put(43,247){$\cdot$}
\put(44,230){$\cdot$}


\put(-20,175){$p=12k+1$}







\put(260,384){\small $j\equiv 1728$}
\put(258,382){\vector(-2,-1){14}}
\qbezier(213,373)(245,371)(269,374)
\qbezier(213,372)(245,370)(269,373)
\put(195,370){$\frac{p-1}{2}$}
\put(273,370){\scriptsize$B$}


\put(216,318){\scriptsize{$j\equiv 0$}}
\put(231,325){\vector(1,2){5}}

\qbezier(188,338)(245,341)(290,333)
\qbezier(188,337)(245,340)(290,332)
\put(297,328){\scriptsize$D_0$}
\put(164,334){$\frac{p+1}{3}$}

\qbezier(210,352)(212,343)(209,328)
\put(211,355){\scriptsize$1$}
\put(214,350){\scriptsize$E_0$}

\qbezier(256,353)(252,342)(253,327)
\put(257,354){\scriptsize$1$}
\put(260,346){\scriptsize$F_0$}


\qbezier(182,306)(245,316)(295,301)
\qbezier(182,305)(245,315)(295,300)
\qbezier(182,304)(245,314)(295,299)
\put(305,297){\scriptsize$D_1$}
\put(159,302){\scriptsize$p+1$}

\qbezier(180,279)(245,286)(296,270)
\qbezier(180,278)(245,285)(295,269)
\qbezier(180,277)(245,284)(295,268)
\put(305,268){\scriptsize$D_2$}
\put(158,274){\scriptsize$p+1$}

\qbezier(188,226)(245,236)(295,226)
\qbezier(188,225)(245,235)(295,225)
\qbezier(188,224)(245,234)(295,224)
\put(305,220){\scriptsize$D_k$}
\put(163,221){\scriptsize$p+1$}

\put(221,271){$\cdot$}
\put(222,253){$\cdot$}
\put(224,234){$\cdot$}


\qbezier(216,264)(216,273)(212,292)
\put(207,296){\scriptsize$1$}
\put(216,288){\scriptsize$E_2$}

\qbezier(202,293)(203,305)(198,322)
\put(192,319){\scriptsize$1$}
\put(201,316){\scriptsize$E_1$}

\qbezier(208,207)(211,225)(208,248)
\put(202,253){\scriptsize$1$}
\put(211,247){\scriptsize$E_k$}


\qbezier(268,267)(262,275)(263,292)
\put(258,296){\scriptsize$1$}
\put(264,288){\scriptsize$F_2$}

\qbezier(272,293)(266,305)(268,322)
\put(263,324){\scriptsize$1$}
\put(269,315){\scriptsize$F_1$}

\qbezier(275,207)(263,225)(268,248)
\put(270,253){\scriptsize$1$}
\put(273,244){\scriptsize$F_k$}


\put(232,392){\scriptsize$p-1$}
\put(225,380){\scriptsize$A$}

\qbezier(241,385)(240,270)(234,200)
\qbezier(240,385)(239,270)(233,200)
\qbezier(239,386)(238,270)(232,200)

%


\put(215,175){$p=12k+5$}






\qbezier(-50,10)(5,20)(45,8)
\qbezier(-50,11)(5,21)(45,9)
\qbezier(-50,12)(5,22)(45,10)

\put(-73,5){\scriptsize$D_2$}
\put(51,5){\scriptsize$p+1$}
\put(-74,-45){\scriptsize$D_k$}

\qbezier(-35,23)(-32,13)(-34,0)
\put(-37,26){\scriptsize$1$}
\put(-34,18){\scriptsize$E_2$}

\qbezier(20,31)(18,13)(21,2)
\put(21,32){\scriptsize$1$}
\put(21,24){\scriptsize$F_2$}




\qbezier(-49,36)(5,51)(46,39)
\qbezier(-49,37)(5,52)(46,40)
\qbezier(-49,38)(5,53)(46,41)

\put(-73,35){\scriptsize$D_1$}
\put(50,35){\scriptsize$p+1$}

\qbezier(-27,50)(-24,38)(-26,25)
\put(-29,53){\scriptsize$1$}
\put(-26,47){\scriptsize$E_{1}$}

\qbezier(30,52)(28,46)(31,25)
\put(30,53){\scriptsize$1$}
\put(33,46){\scriptsize$F_{1}$}




\qbezier(-48,-47)(5,-37)(59,-52)
\qbezier(-48,-46)(5,-36)(59,-51)
\qbezier(-48,-45)(5,-35)(59,-50)

\put(-74,-45){\scriptsize$D_k$}
\put(64,-53){\scriptsize$p+1$}

\qbezier(-25,-25)(-23,-46)(-26,-57)
\put(-24,-22){\scriptsize$1$}
\put(-21,-30){\scriptsize$E_k$}

\qbezier(27,-25)(24,-48)(26,-58)
\put(28,-20){\scriptsize$1$}
\put(31,-28){\scriptsize$F_k$}


\qbezier(0,123)(-5,0)(15,-70) 
\qbezier(1,123)(-4,0)(16,-70)
\qbezier(-1,123)(-6,0)(14,-70)

\put(-9,128){\scriptsize$p-1$}
\put(-12,117){\scriptsize$A$}



\put(16,104){\small $j\equiv 0$}
\put(15,105){\vector(-2,-1){12}}
\qbezier(-28,98)(5,96)(31,99)
\qbezier(-28,97)(5,95)(31,98)
\put(39,95){$\frac{p-1}{3}$}
\put(-36,96){\scriptsize$C$}



\qbezier(-35,70)(5,78)(45,74)
\qbezier(-35,69)(5,77)(45,73)
\put(52,73){$\frac{p+1}{2}$}

\put(4,88){\scriptsize{$j\equiv 1728$}}
\put(8,86){\vector(-1,-2){5}}

\put(-61,70){\scriptsize$D_{-1}$}

\qbezier(-21,82)(-19,66)(-21,61)
\put(-20,82){\scriptsize$1$}
\put(-19,62){\scriptsize$E_{-1}$}

\qbezier(19,82)(14,66)(19,61)
\put(22,79){\scriptsize$1$}
\put(22,61){\scriptsize$F_{-1}$}


\put(-45,-0){$\cdot$}
\put(-45,-17){$\cdot$}
\put(-45,-36){$\cdot$}

\put(42,-4){$\cdot$}
\put(43,-23){$\cdot$}
\put(44,-40){$\cdot$}


\put(-20,-95){$p=12k+7$}








\qbezier(190,99)(245,105)(289,98)
\qbezier(190,98)(245,104)(289,97)
\put(295,93){\scriptsize$D_{-1}$}
\put(169,95){$\frac{p+1}{2}$}

\qbezier(210,114)(212,108)(209,93)
\put(211,117){\scriptsize$1$}
\put(214,108){\scriptsize$E_{-1}$}

\qbezier(256,115)(252,107)(253,92)
\put(257,116){\scriptsize$1$}
\put(260,108){\scriptsize$F_{-1}$}


\qbezier(188,68)(245,71)(290,63)
\qbezier(188,67)(245,70)(290,62)
\put(297,58){\scriptsize$D_0$}
\put(164,64){$\frac{p+1}{3}$}

\qbezier(210,82)(212,73)(209,58)
\put(203,75){\scriptsize$1$}
\put(214,77){\scriptsize$E_{0}$}

\put(206,86){\scriptsize{$j\equiv 1728$}}
\put(218,91){\vector(2,1){15}}

\qbezier(256,83)(252,72)(253,57)
\put(257,79){\scriptsize$1$}
\put(255,71){\scriptsize$F_{0}$}

\put(253,86){\scriptsize{$j\equiv 0$}}
\put(251,86){\vector(-1,-2){8}}


\qbezier(182,36)(245,46)(295,31)
\qbezier(182,35)(245,45)(295,30)
\qbezier(182,34)(245,44)(295,29)
\put(305,27){\scriptsize$D_1$}
\put(160,32){\scriptsize$p+1$}

\qbezier(180,9)(245,16)(296,0)
\qbezier(180,8)(245,15)(295,-1)
\qbezier(180,7)(245,14)(295,-2)
\put(305,-2){\scriptsize$D_2$}
\put(158,4){\scriptsize$p+1$}

\qbezier(188,-44)(245,-34)(295,-44)
\qbezier(188,-45)(245,-35)(295,-45)
\qbezier(188,-46)(245,-36)(295,-46)
\put(305,-50){\scriptsize$D_k$}
\put(162,-49){\scriptsize$p+1$}

\put(221,1){$\cdot$}
\put(222,-17){$\cdot$}
\put(224,-36){$\cdot$}


\qbezier(216,-6)(216,13)(212,22)
\put(207,26){\scriptsize$1$}
\put(214,18){\scriptsize$E_2$}

\qbezier(202,23)(203,35)(198,52)
\put(192,54){\scriptsize$1$}
\put(200,46){\scriptsize$E_1$}

\qbezier(208,-63)(211,-45)(208,-22)
\put(202,-17){\scriptsize$1$}
\put(210,-25){\scriptsize$E_k$}


\qbezier(268,-3)(262,5)(263,22)
\put(258,26){\scriptsize$1$}
\put(264,18){\scriptsize$F_2$}

\qbezier(272,23)(266,35)(268,52)
\put(263,54){\scriptsize$1$}
\put(270,46){\scriptsize$F_1$}

\qbezier(275,-63)(263,-45)(268,-22)
\put(270,-17){\scriptsize$1$}
\put(273,-25){\scriptsize$F_k$}


\put(232,129){\scriptsize$p-1$}
\put(227,120){\scriptsize$A$}

\qbezier(241,122)(236,0)(246,-70)
\qbezier(240,122)(235,0)(245,-70)
\qbezier(239,123)(234,0)(244,-70)

%


\put(215,-95){$p=12k+11$}



\end{picture}
\end{center}
\caption{Special fiber at $\overline{\F}_p$ of the minimal regular model with minimal crossings of $X_{\mathrm{ns}} (p)$}  
\label{figureNSgrossier}
\end{figure}

\medskip

   Now consider ordinary points with $j\equiv 0$ (so $p\equiv 1$ mod $6$). There are $(p-1)/6$ points on each of the two 
Igusa components $\overline{M} ({\mathrm{Ig}} (p)/\{\pm 1\} )_{\overline{\F}_p} $ in the special fiber of the 
semistable $X_{\mathrm{ns}} (p)^{\mathrm{st}}$ over $\Z_ p^{\mathrm{ur}} [p^{2/(p^2 -1)}]$. Those Igusa parts can, as 
before, be descended over 
$\Z_ p^{\mathrm{ur}} [\pi_2 =p^{1/(p -1)}]$, and then the inertia of the Galois action at the points for which 
$j\equiv 0$ is the order 3 subgroup of $\F_{p^2}^* /\mu_{p+1} (\F_{p^2} )\simeq \F_p^*$, 
which is spanned by some $\omega_3 \in \F_p^*$ that is lifted to $u^{(p-1)/3} \mod \mu_{p+1} (\F_{p^2} )$ 
(with $u$ our generator of $\F_{p^2}^*$). Then $u^{(p -1)/3} \colon  \pi_2 \mapsto u^{(p^2 -1)/3} \pi_2$. 
On the other hand, Proposition~\ref{stabilisateur} implies that the Galois action of $\omega_3$ on the Igusa component is 
\begin{eqnarray}
\label{lesordres6}
(E,P)\mapsto (E,\langle u^{(p^2 -1)/6} \rangle_p P) ,
\end{eqnarray}
which is the same as that of the exceptional automorphism $[u^{(p^2 -1)/6} ]$ (action of a $6^{\mathrm{th}}$ root of unity) 
on the curve $E_0$ with $j$-invariant $0$ on $\overline{\F}_p$. From (7) in~\cite{EdixP20}, end of 
Section~2.2.2, one therefore sees that the action of $\omega_3$ on the universal deformation space $\overline{\F}_p 
[[t]]$ at the elliptic curve $E_0$ with $j$-invariant $0$ is given by $t\mapsto (u^{(p^2 -1)/6} )^2 t=\omega_3 t$. 
Summing up, if $\Z_p^{\mathrm{ur}} [\pi_2 ] [[t]]$ is the completed local ring at such a point 
over $\Z_p^{\mathrm{ur}} [\pi_2 ]$, then the inertia subgroup of Galois can be described as
\begin{eqnarray}
\label{lactionomega3}
\omega_3 \colon \pi_2 \mapsto \omega_3 \pi_2 ,\  t\mapsto \omega_3 t.
\end{eqnarray}

Therefore, computations similar to those in the previous case (when $j\equiv 1728$) show that if those points with  
$j\equiv 1728$ are ordinary at $p$ ($\equiv 1$ mod $3$), blowing up gives a regular model over $\Z_p^{\mathrm{ur}}$ 
with one extra rational component above $p$, with multiplicity $(p-1)/3$, intersecting the Igusa component in a 
singularity of shape
$$
\Z_p^{\mathrm{ur}} [[a,b]]/(a^{p-1} b^{\frac{p-1}{3}} -p).
$$

\medskip

    Now go back to the case $j\equiv 1728$, and assume it is supersingular at $p$ ($\equiv -1$ mod $4$). 
Checking~\cite{EdixP20}, Theorem~3.4 and the proof of Theorem~\ref{thNS} above, one sees that the exceptional 
Drinfeld component on the semistable model $X_{\mathrm{ns}} (p)^{\mathrm{st}}$ over $\Z_p^{\mathrm{ur}} [p^{2/(p^2 -1)} ]$ 
has equation
$$
U^2 =W^{\frac{p+1}{2}} +A_{\mathrm{ns}}
$$
(for some nonzero $A_{\mathrm{ns}} \equiv (aN/2)$ mod $p$), with notation for $U$ and $W=V^2$ as 
in~(\ref{DrinfeldXns}). The Galois action is therefore deduced 
from~(\ref{actionDeGalois}) to be $u\colon U\mapsto U, W\mapsto u^{2(p-1)} W$. The corresponding quotient Drinfeld 
component is again a rational curve, but its 
multiplicity is now $(p+1)/2$. Looking at points where Galois has inertia therefore produces singular points in the 
quotients and leads to computations completely
similar to those at the end of the proof of Theorem~\ref{thNS}. We eventually obtain that, as its plain 
Drinfeld siblings, the exceptional Drinfeld component 
with $j\equiv 1728$ is endowed with two extra components which are also copies of $\PPP^1$ with multiplicity $1$; the local 
completed singularity is
$$
\Z_p^{\mathrm{ur}} [[a,b]]/(a^{\frac{p+1}{2}} b-p).
$$ 
The ring of the intersection between the Igusa component and that special Drinfeld component is
$$
\Z_p^{\mathrm{ur}} [[a,b]]/(a^{p-1} b^{\frac{p+1}{2}} -p).
$$ 
Finally, in the case where $j\equiv 0$ is supersingular at $p$ ($\equiv -1$ mod $3$), one obtain mutatis mutandis 
a Drinfeld component which once more is a rational curve, with multiplicity $(p+1)/3$, endowed with two extra components 
which are copies of $\PPP^1$ with multiplicity $1$, and whose completed ring at the intersection point is  
$$
\Z_p^{\mathrm{ur}} [[a,b]]/(a^{\frac{p+1}{3}} b -p);
$$ 
the ring of the intersection between the Igusa component and the special Drinfeld component is
$$
\Z_p^{\mathrm{ur}} [[a,b]]/(a^{p-1} b^{\frac{p+1}{3}} -p).
$$ 
As in the case of Theorem~\ref{thNS+} (see end of its proof), singular points at the intersection of Drinfeld and Igusa 
parts may be {\em{not}} regular in the semistable model of $X_{\mathrm{ns}} (p)^{\mathrm{st}}$  over $\Z_p^{\mathrm{ur}} 
[p^{2/(p^2 -1)}]$ (see Theorem~3.4 of~\cite{EdixP20}). However, their image in the Galois quotient is regular over 
$\Z_p^{\mathrm{ur}}$ (and, more precisely, have the above shape). 

    The statement that the regular models we obtained are the {\em minimal} regular models with normal crossings is easily 
seen from the fact that no component can be contracted without losing the ``normal crossings'' property. Details will be 
displayed in the proof of next Corollary~\ref{minimalXns} (see also Figure~\ref{figureNSminimal5}). $\Box$

\end{proof}
 
\medskip 

We now describe the minimal regular models for our curves. The point is to compute self-intersections of irreducible 
components in the special fiber, then contract the projective lines with self-intersection $-1$ (Castelnuovo criterion), and 
repeat the process on the new fiber thus obtained. It turns out that one finally obtains reduced fibers. 

\medskip

(Let us recall here for the convenience of the reader that for $p\ge 5$ which one writes $p=12k+i$, with $i= 1, 5, 7$ or 
$11$, the number of supersingular $j$-invariants in characteristic $p$ is $k, k+1, k+1$ or $k+2$, respectively. In 
particular, the number of supersingular $j$-invariants which are different from $0$ and $1728$ mod $p$ is $k$ in all four 
cases.)

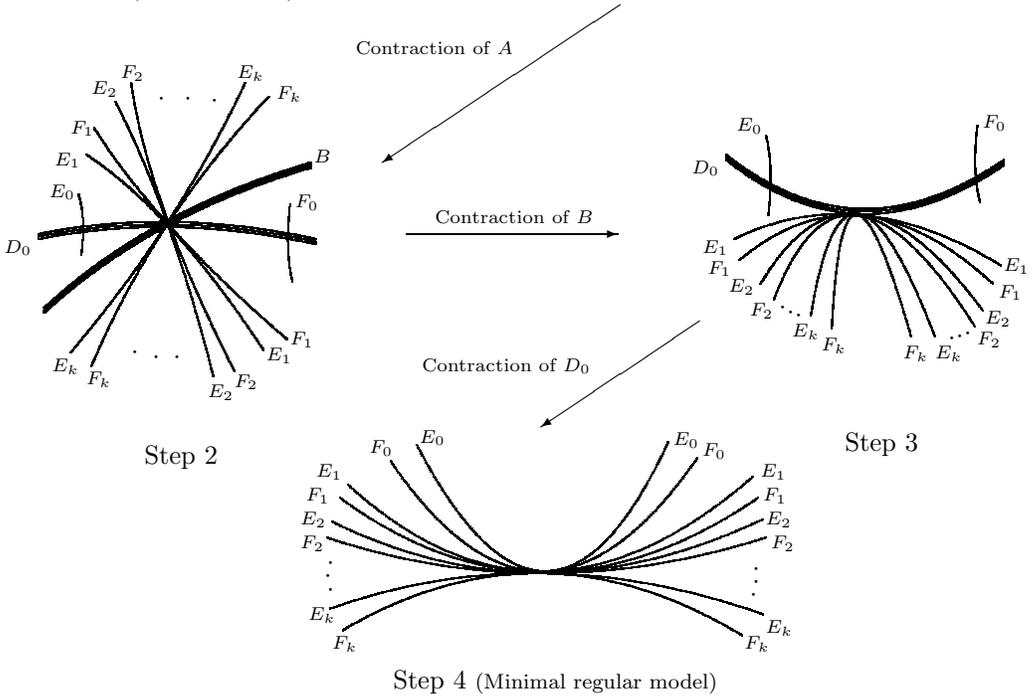
\begin{figure}
\begin{center}

\begin{picture}(320,410)(-30,-90)

\qbezier(-50,280)(5,290)(45,278)
\qbezier(-50,281)(5,291)(45,279)
\qbezier(-50,282)(5,292)(45,280)

\put(-67,275){\scriptsize$D_2$}
\put(54,275){\scriptsize$p+1$}

\qbezier(-35,293)(-32,283)(-34,270)
\put(-37,296){\scriptsize$1$}
\put(-31,272){\scriptsize$E_2$}

\qbezier(20,301)(18,283)(21,272)
\put(21,302){\scriptsize$1$}
\put(24,271){\scriptsize$F_2$}

\qbezier(-51,306)(5,321)(46,309)
\qbezier(-51,307)(5,322)(46,310)
\qbezier(-51,308)(5,323)(46,311)

\put(-67,310){\scriptsize$D_1$}
\put(53,305){\scriptsize$p+1$}

\qbezier(-27,320)(-24,308)(-26,295)
\put(-29,323){\scriptsize$1$}
\put(-24,299){\scriptsize$E_1$}

\qbezier(30,322)(28,316)(31,295)
\put(30,323){\scriptsize$1$}
\put(32,301){\scriptsize$F_1$}

\qbezier(-48,217)(5,233)(59,218)
\qbezier(-48,216)(5,232)(59,217)
\qbezier(-48,215)(5,231)(59,216)

\put(-66,215){\scriptsize$D_k$}
\put(64,217){\scriptsize$p+1$}

\qbezier(-25,245)(-23,224)(-26,206)
\put(-24,248){\scriptsize$1$}
\put(-24,205){\scriptsize$E_k$}

\qbezier(27,245)(24,222)(26,202)
\put(28,250){\scriptsize$1$}
\put(28,203){\scriptsize$F_k$}


\qbezier(0,393)(-5,270)(-2,200) 
\qbezier(1,393)(-4,270)(-1,200)
\qbezier(-1,393)(-6,270)(-3,200)

\put(-6,398){\scriptsize$p-1$}
\put(-11,390){\scriptsize$A$}



\put(-34,375){\scriptsize$B$}
\qbezier(-25,378)(5,376)(29,379)
\qbezier(-25,377)(5,375)(29,378)
\put(34,374){\scriptsize$\frac{p-1}{2}$}


\qbezier(-45,351)(5,350)(49,352)
\qbezier(-45,350)(5,349)(49,351)
\put(57,346){$\frac{p+1}{3}$}
\put(-59,348){\scriptsize$D_0$}

\qbezier(-25,365)(-23,354)(-26,336)
\put(-23,366){\scriptsize$1$}
\put(-22,358){\scriptsize$E_0$}

\qbezier(22,365)(19,342)(21,335)
\put(24,365){\scriptsize$1$}
\put(24,356){\scriptsize$F_0$}


\put(-45,+270){$\cdot$}
\put(-45,253){$\cdot$}
\put(-45,234){$\cdot$}

\put(42,266){$\cdot$}
\put(43,247){$\cdot$}
\put(44,230){$\cdot$}


\put(-70,175){\scriptsize{Minimal regular model with normal crossings}} 
\put(-25,162){\scriptsize{(cf. Theorem~\ref{thNSgrossier})}}


 
\put(83,301){\vector(1,0){76}}

\put(85,290){\scriptsize{Contraction of the $D_i$}}
\put(102,280){\scriptsize{for $i>0$}}





\qbezier(213,373)(245,371)(269,374)
\qbezier(213,372)(245,370)(269,373)
\put(195,370){\scriptsize$B$}


\qbezier(188,338)(245,341)(290,333)
\qbezier(188,337)(245,340)(290,332)
\put(174,334){\scriptsize$D_0$}

\qbezier(208,352)(210,343)(207,328)
\put(211,346){\scriptsize$E_0$}

\qbezier(256,353)(252,342)(253,327)
\put(258,345){\scriptsize$F_0$}


\qbezier(212,304)(246,310)(266,322)
\put(267,296){\scriptsize$F_1$}

\qbezier(211,324)(246,304)(266,302)
\put(207,296){\scriptsize$E_1$}

\qbezier(212,285)(246,266)(266,263)
\put(267,255){\scriptsize$F_2$}

\qbezier(211,265)(246,272)(266,285)
\put(207,255){\scriptsize$E_2$}

\put(221,251){$\cdot$}
\put(222,238){$\cdot$}
\put(224,224){$\cdot$}


\qbezier(212,227)(246,205)(266,205)
\put(267,197){\scriptsize$F_k$}

\qbezier(211,207)(246,217)(266,227)
\put(207,197){\scriptsize$E_k$}


\put(232,392){\scriptsize$A$}

\qbezier(241,385)(240,270)(234,200)
\qbezier(240,385)(239,270)(233,200)
\qbezier(239,386)(238,270)(232,200)

%


\put(215,175){Step 1}




 
\put(160,159){\vector(-3,-2){90}}

\put(60,140){\scriptsize{Contraction of $A$}}


\qbezier(-60,70)(-10,80)(45,68)
\qbezier(-60,71)(-10,81)(45,69)
\qbezier(-60,72)(-10,82)(45,70)

\put(-73,65){\scriptsize$D_0$}

\qbezier(-45,87)(-42,77)(-44,64)
\put(-56,86){\scriptsize$E_0$}

\qbezier(35,83)(33,65)(36,54)
\put(37,82){\scriptsize$F_0$}


\qbezier(-42,102)(-10,81)(25,28)
\put(-54,98){\scriptsize$E_1$}
\put(26,24){\scriptsize$E_1$}

\qbezier(-39,112)(-15,74)(34,32)
\put(-48,110){\scriptsize$F_1$}
\put(35,30){\scriptsize$F_1$}


\qbezier(-31,122)(-10,81)(6,18)
\put(-40,125){\scriptsize$E_2$}
\put(4,10){\scriptsize$E_2$}

\qbezier(-25,129)(-19,85)(14,20)
\put(-29,131){\scriptsize$F_2$}
\put(14,14){\scriptsize$F_2$}

 
\put(-15,121){$\cdot$}
\put(-5,121){$\cdot$}
\put(5,120){$\cdot$}


\qbezier(18,129)(-8,74)(-48,27)
\put(15,131){\scriptsize$E_k$}
\put(-55,20){\scriptsize$E_k$}

\qbezier(27,124)(-17,77)(-40,22)
\put(30,123){\scriptsize$F_k$}
\put(-42,15){\scriptsize$F_k$}


\put(-25,25){$\cdot$}
\put(-17,23){$\cdot$}
\put(-9,23){$\cdot$}


\qbezier(43,97)(-18,79)(-58,42)
\qbezier(43,98)(-18,80)(-58,43)
\qbezier(43,99)(-18,81)(-58,44)
\put(44,99){\scriptsize$B$}


\put(-20,-15){Step 2}




\put(79,72){\vector(1,0){80}}
\put(90,76){\scriptsize{Contraction of $B$}}


\qbezier(200,100)(250,60)(305,98)
\qbezier(200,101)(250,61)(305,99)
\qbezier(200,102)(250,62)(305,100)

\put(187,95){\scriptsize$D_0$}

\qbezier(215,109)(218,99)(216,79)
\put(204,111){\scriptsize$E_0$}

\qbezier(295,113)(293,95)(296,84)
\put(297,112){\scriptsize$F_0$}


\qbezier(203,70)(250,93)(304,60)
\put(191,65){\scriptsize$E_1$}
\put(305,58){\scriptsize$E_1$} 

\qbezier(205,62)(250,100)(301,53)
\put(193,57){\scriptsize$F_1$} 
\put(303,48){\scriptsize$F_1$}


\qbezier(213,53)(247,111)(297,43)
\put(201,50){\scriptsize$E_2$}
\put(298,38){\scriptsize$E_2$}

\qbezier(218,47)(244,116)(294,37)
\put(209,41){\scriptsize$F_2$} 
\put(295,30){\scriptsize$F_2$} 


\put(220,42){$\cdot$}
\put(223,40){$\cdot$} 
\put(226,39){$\cdot$}  


\qbezier(232,41)(245,122)(279,33)
\put(225,33){\scriptsize$E_k$}
\put(279,25){\scriptsize$E_k$}

\qbezier(240,36)(244,124)(270,33)
\put(236,28){\scriptsize$F_k$} 
\put(267,25){\scriptsize$F_k$} 


\put(284,30){$\cdot$}
\put(287,31){$\cdot$} 
\put(290,32){$\cdot$}  


\put(245,-10){Step 3}


 
\put(190,39){\vector(-3,-2){60}}

\put(85,20){\scriptsize{Contraction of $D_0$}}


\put(84,-7){\scriptsize$E_0$}
\qbezier(83,-8)(133,-104)(178,-7)
\put(180,-8){\scriptsize$E_0$}

\put(65,-11){\scriptsize$F_0$}
\qbezier(73,-14)(130,-98)(189,-13)
\put(191,-12){\scriptsize$F_0$}


\put(45,-20){\scriptsize$E_1$}
\qbezier(57,-23)(130,-90)(210,-20)
\put(213,-20){\scriptsize$E_1$}

\put(41,-29){\scriptsize$F_1$}
\qbezier(54,-28)(130,-84)(212,-28)
\put(214,-29){\scriptsize$F_1$}


\put(39,-38){\scriptsize$E_2$}
\qbezier(51,-37)(130,-76)(214,-36)
\put(215,-38){\scriptsize$E_2$}

\put(39,-47){\scriptsize$F_2$}
\qbezier(49,-43)(130,-70)(215,-43)
\put(217,-47){\scriptsize$F_2$}


\put(49,-55){$\cdot$}
\put(48,-60){$\cdot$}
\put(49,-65){$\cdot$}
 

\put(42,-74){\scriptsize$E_k$}
\qbezier(50,-70)(128,-42)(214,-72)
\put(215,-79){\scriptsize$E_k$}

\put(51,-85){\scriptsize$F_k$}
\qbezier(55,-78)(130,-34)(206,-80)
\put(207,-85){\scriptsize$F_k$}


\put(210,-57){$\cdot$}
\put(210,-62){$\cdot$}
\put(209,-67){$\cdot$}
 
  

\put(74,-100){Step 4 \footnotesize{(Minimal regular model)}}

 
\end{picture}
\end{center}
\caption{Minimal resolution of $X_{\mathrm{ns}} (p)$ in the case $p=12k+5$} 
\label{figureNSminimal5}
\end{figure}

\begin{coro}
\label{minimalXns}
Assume $p\geq 5$. Writing $p=12k+i$ with $i=1, 5, 7$ or $11$, set $n(p) =2k$ if $i=1$, $n (p) =2k+2$ if $i=5$ or $i=7$
and $n(p) =2k+4$ if $i=11$. 
    
   The {\em minimal regular model} of $X_{\mathrm{ns}} (p)$ over $\Z_p^{\mathrm{ur}}$
has a special fiber wich is reduced, made of $n(p)$ components intersecting at one common (very) singular point. Those 
components are quotients of the projective line, which might be singular at the common intersection point. See 
Figure~\ref{figureNSminimal5}.
\end{coro}

\begin{proof}

One starts from the regular models obtained in Theorem~\ref{thNSgrossier}. By Castelnuovo's criterion (cf.~\cite{Liu02}, Chapter~9.3), one just needs to iteratively contract the components in the special fiber which are isomorphic to projective lines (in our case, all irreducible components are) and have self-intersection $-1$. One considers separately the four possible cases of congruence for $p$ mod $12$. Namely, take for instance some $p=12k+5$. We have labelled components on Figure~\ref{figureNSgrossier}; the only Igusa one is $A$ (with multiplicity $p-1 =12k+4$), and the horizontal ones are $B$ (with multiplicity $(p-1)/2 =6k+2$), $D_0$ (multiplicity $(p+1)/3=4k+2$), $D_1 , \dots , D_k$ (all with multiplicity $p+1=12k+6$), $E_0$ and $F_0$ (the two multiplicity $1$ components intersecting $D_0$) and similarly $E_i$, $F_i$ (the multiplicity $1$ components intersecting $D_i$) for $1\le i\le k$; see also Figure~\ref{figureNSminimal5}. Any two different components 
$X$ and $Y$ have intersection number $X\cdot Y=1$, so computing self-intersection numbers gives (as intersecting with the whole special fiber gives $0$)
$$
A^2 :=A\cdot A=\frac{-1}{12k+4} \left( (6k+2) A\cdot B + (4k+2) A\cdot D_0 +\sum_{i=1}^k (12k+6)A\cdot D_i \right) =-1-k
$$      
and similarly
$$
B^2 =-2, D_0^2 =-3, E_0^2 = F_0^2 =-4k-2\ {\mathrm{and\ for}}\ 1\le i\le k, D_i^2 =-1 \ {\mathrm{and}}\ E_i^2 =F_i^2 =-12k-6.  
$$
So one first contracts all $D_i$ for $i\ge 1$ (Step 1 in Figure~\ref{figureNSminimal5}). Let $\pi$ be the contraction 
morphism. By the projection formula (see~\cite{Liu02}, Chapter~9, Theorem~2.12), if $A' :=\pi_* (A)$ denotes the image 
of $A$ in the new scheme, then 
$$
{A'}^2 =\pi_* (A)\cdot \pi_* (A)=A\cdot \pi^* \pi_* (A)=A\cdot [A+ \sum_{i=1}^k D_i]
=-1-k +k=-1$$
(where the first two intersection products above occur on the contracted scheme, whereas the others take place in 
the original arithmetic surface). So one further contracts $A'$ (Step~2  in Figure~\ref{figureNSminimal5}). 
Note at this point that, whereas the intersection 
product of any two different components has always been $1$ so far, it is no longer the case after contracting $A'$ (one 
has $E_i \cdot F_i =2$ for $i>0$). 
By abuse of notation, we forget about the `prime' in the new scheme (so now $A'$ is denoted by $A$ again) and keep writing 
$\pi$ for the new contraction, and we let $B' :=\pi_* (B)$ denote the image of $B$, 
which satisfies
$$
{B'}^2 =B\cdot \pi^* \pi_* (B)=B\cdot [B+A] =-1
$$    
so we contract $B^{(')}$ (Step~3). Finally for $D_0' :=\pi (D_0 )$,
$$
{D_0'}^2 =D_0 \cdot [D_0 +B+2A+ \sum_{i=1}^k D_i ] =D_0^2 +2 D_0 \cdot A =-1
$$ 
(this time the first product occurs in the newest contracted arithmetic surface, whereas the last ones are in the scheme we 
started with). We therefore finally contract $D_0$ (Step~4), and we obtain the minimal regular model which is indeed reduced, 
made of $2k+2=(p+7)/6$ projective lines with multiplicity $1$, each of which intersects all the others at one highly 
singular common point. (Note that writing down the intersection matrix of our minimal regular model with normal crossings, 
and performing a Smith normalization process which parallels our successive contractions, one can obtain the intersection 
matrix of the minimal regular model.)

   The computations in the other congruence cases for $p$ are similar. Explicitely, and labelling the components as in 
Figure~\ref{figureNSgrossier}, for $p=12k+1$, one contracts the $D_i$, then $A$, then $B$, finally $C$; when $p=12k+7$, first 
contract $D_i , i\geq 1$, then $A$, $D_{-1}$ and $C$; for $p=12+11$, first the $D_i , i\geq 1$, then $A$, $D_{-1}$ and 
$D_0$. $\Box$

\end{proof}

\subsubsection{Component groups for ${J}_{\mathrm{ns}} (p)$}

\begin{propo}
\label{componentsXns}
Let $p\ge 17$ be a prime, and denote by ${\mathcal J}_{\mathrm{ns}} (p)$ the N\'eron model over $\Z_p^{\mathrm{ur}}$ of the Jacobian of $X_{\mathrm{ns}} (p)$. The group of connected components 
$({\mathcal J}_{\mathrm{ns}} (p)/{\mathcal J}^0_{\mathrm{ns}} (p))(\overline{\F}_p)$ of the special fiber of
${\mathcal J}_{\mathrm{ns}} (p)$ is: 
\begin{itemize}
\item $\Z /\frac{p+1}{2}\Z \times \Z /12(p+1)\Z \times \left( \Z /(p^ 2 -1)\Z \right)^{(p-25)/12}$ if $p=12k+1$;
\item $\Z /\frac{p+1}{6}\Z \times \Z /4(p+1)\Z \times \left( \Z /(p^ 2 -1)\Z \right)^{(p-17)/12}$ if $p=12k+5$;
\item $\Z /\frac{p+1}{4}\Z \times \Z /6(p+1)\Z \times \left( \Z /(p^ 2 -1)\Z \right)^{(p-19)/12}$ if $p=12k+7$;
\item $\Z /\frac{p+1}{12}\Z \times \Z /2(p+1)\Z \times \left( \Z /(p^ 2 -1)\Z \right)^{(p-11)/12}$ if $p=12k+11$.
\end{itemize}
For $p=2, 3, 5,$  the component groups are trivial, and for $p=7, 11$ and $13$ they are $\Z /2\Z$, $\Z /24\Z$ and $\Z /7\Z$, 
respectively.
\end{propo}

\begin{proof} By \cite{BLR}, Chapter~9.6, Corollary~3, using any of our regular models, we just need to write down the 
intersection matrices at the special fiber and put them in Smith normal form in order to compute the elementary divisors. To 
be explicit, let us again treat a particular case, when $p=12k+5$. After Step 1 of the contraction process 
(Figure~\ref{figureNSminimal5}), one obtains an intersection matrix which, in the components basis $\{ E_1 ,F_1 , E_2 , F_2 , 
\dots , E_k , F_k , A, D_0 , E_0 , F_0 , B\}$, has shape 
$$
\left(
\begin{array}{ccccccccccc}
-12k-5 & 1 & 0 & \cdots & \cdots & 0 & 1 & 0 & 0 & 0 & 0 \\
1 & -12k-5 & 0 & \cdots & \cdots & 0 & 1 &0 &0 &0 & 0\\
0 & 0 &  \vdots & \cdots &  \cdots & 0 & 1 &  0 & 0 & 0& 0\\
  \vdots &  \vdots &  \cdots & \cdots &  \cdots &  \vdots & \vdots & \vdots &  \vdots & \vdots & \vdots \\
 \vdots &  \vdots &  \cdots & \cdots &  \cdots &  \vdots & \vdots & \vdots &  \vdots & \vdots & \vdots \\
0  & 0 &  \vdots & \cdots &  0 &  0 & 1 & 0 &  0 & 0 & 0\\
0  & 0 &  \vdots & \cdots & -12k-5 & 1 & 1 & 0 &  0 & 0 & 0 \\
0  & 0 &  \vdots & \cdots & 1 & -12k-5 & 1 & 0 &  0 & 0 & 0\\
1 & 1 & \vdots & \cdots & 1 & 1 & -1 & 1 & 0 & 0 & 1 \\
0 & 0 & \vdots & \cdots & 0 & 0 & 1 & -3 & 1 & 1 & 0 \\ 
0 & 0 & \vdots & \cdots & 0 & 0 & 0 & 1 & -4k-2 & 0 & 0 \\ 
0 & 0 & \vdots & \cdots & 0 & 0 & 0 & 1 & 0 & -4k-2 & 0 \\ 
0 & 0 & \vdots & \cdots & 0 & 0 & 1 & 0 & 0 & 0 & -2 
\end{array}
\right)
$$
(see proof of Corollary~\ref{minimalXns}). After going through a lengthy but elementary manual Smith normal process, this 
does yield\footnote{One finds in GP-Pari a function {\tt matsnf} (for ``matrices Smith normal form'') which helps making 
safety checks with our formulas for small values of $p$.} elementary divisors with quotient group (when $k\ge 1$):
$$
(\Z /(2k+1)\Z )\times (\Z /24(2k+1)\Z )\times (\Z /24(2k+1)(3k+1)\Z )^{k-1} . \hspace{1cm} \Box
$$

\end{proof}

\subsection{$X_{\mathrm{ns}}^+ (p)$}

\begin{theo} 
\label{thNS+grossier} 
Let $p\geq 13$ be a prime, and let $X_{\mathrm{ns}}^+ (p)$ be the coarse modular curve over $\Q$ associated with 
$\Gamma_{\mathrm{ns}}^+ (p)$.

\medskip
 
  The {\em minimal regular model with normal crossings} of the curve $X_{\mathrm{ns}} (p)$ over $\Z_p^{\mathrm{ur}}$ has a 
special fiber made of one vertical Igusa component, intersecting horizontal chains of Drinfeld components above each 
supersingular point.

\medskip 
 
  That vertical part is a projective ($j$-)line, with multiplicity $(p-1)/2$.  

\medskip

If ${\cal S}$ is the number of supersingular $j$-invariants in $\F_{p^2}$, there are ${\cal S}$ horizontal 
chains of (Drinfeld) components. Unless the corresponding $j$-invariant has exceptional geometric automorphisms (that 
is, $j\equiv 0$ or $1728$ {\rm mod} $p$), those Drinfeld curves are all copies of  a rational curve $\PPP^1$ with 
multiplicity $p+1$, on which,
at one point, stems another $\PPP^1$ with multiplicity $1$, and at another point still another $\PPP^1$ rises, that one with 
multiplicity $(p+1)/2$. 

\medskip

  Depending on the residue class of $p\mod 12$, the vertical Igusa components moreover have the following additional 
equipments: 
\begin{itemize}
\item {\bf If $p\equiv 1$ mod $12$}, the vertical Igusa component intersects transversally an additional chain of 
components above the ordinary point $j\equiv 0\mod p$, which is made up of a projective line, with multiplicity $(p-1)/3$, 
followed by another $\PPP^1$, with multiplicity $(p-1)/6$.
\item  {\bf If $p\equiv 5$ mod $12$}, then the vertical Igusa component has one 
exceptional Drinfeld component, at the supersingular $j\equiv 0\mod p$, which is a projective line with multiplicity 
$(p+1)/3$, intersecting itself two extra projective lines, one with multiplicity $1$, the other one with multiplicity 
$(p+1)/6$.
\item  {\bf If $p\equiv 7$ mod $12$}, the vertical Igusa component intersects transversally an additional chain of components 
at the ordinary $j\equiv 0\mod p$, as in the case $p\equiv 1\ {\mathrm{mod}}\ 12$: a projective line, with multiplicity 
$(p-1)/3$, followed by another $\PPP^1$, with multiplicity $(p-1)/6$. The exceptional Drinfeld component at the supersingular 
$j\equiv 1728 \mod p$ is just a projective line with multiplicity $1$. 
\item  {\bf If $p\equiv 11$ mod $12$}, there are two exceptional Drinfeld components at the supersingular $j\equiv 0\mod p$ 
and $j\equiv 1728$, which are as in the respective cases above.   
\end{itemize}
All intersections points in the special fiber between components with multiplicity $a$ and $b$, say, are normal crossings 
points, with local equations
$$
\Z_p^{\mathrm{ur}} [[X,Y]]/(X^a \cdot Y^b -p ).
$$
  
    (See Figure~\ref{figureNSgrossier+}.) 
\end{theo}

\begin{rema}
\label{petitsminimauxNsplit+}
{\rm{As in Remark~\ref{petitsminimauxNsplit}, note that Theorem~\ref{thNS+grossier} in fact gives regular models for all 
$p\geq 5$, but we have discarded primes $p\leq 11$ 
from its statement because of the non-minimality of our regular model with normal crossings in those cases. However, 
$X_{\mathrm{ns}}^+ (p)$ is just a projective line if $p\leq 7$ (including $p=2, 3$), and for $p=11$ the minimal regular 
model with normal crossings is obtained from the regular model above by contracting the vertical Igusa component (wich has 
multiplicity $(p-1)/2=5$).}}
\end{rema}

\begin{rema}
\label{etmaintenantlegenredeNsplit+}
{\rm{As a sanity check, we can notice that Theorem~\ref{thNS+grossier} provides SNC-models (in the terminology 
of~\cite{Dino90}), with genus-$0$ irreducible components in their special fiber. So loc. cit. shows that the genus of the 
$X_{\mathrm{ns}}^+ (p)$ can be read on the dual graph $G$ of their special fiber as
$$
g(X_{\mathrm{ns}}^+ (p)) =\beta +\frac{1}{2} \sum_{v} (r_v -1)(d_v -2),
$$
where $\beta$ is the Betti number of $G$ (which here is $0$), and the above sum runs through the vertices $v$ of $G$, for 
which $d_v$ denotes the degree of $v$ and $r_v$ is the multiplicity of the associated irreducible component 
(see~\cite{Dino90}, formula on top of p. ~150). With that formula in hand, it is then easy to compute genera for each of the 
four congruence cases mod $12$ of Theorem~\ref{thNS+grossier}, and check that one finds back the formula:
$$
g(X_{\mathrm{ns}}^+ (p))=\frac{p^2 -10\, p +23 +6\left( {-1 \atop p} \right) +4\left( {-3 \atop p} \right)}{24} 
$$
obtained on the generic fibers using Riemann-Hurwitz formula (see e.g.~\cite{Ma76}, p. 117). 
}}
\end{rema}

\begin{figure}
\begin{center}

\begin{picture}(320,410)(-30,-90)

\qbezier(-50,280)(5,290)(45,278)
\qbezier(-50,281)(5,291)(45,279)
\qbezier(-50,282)(5,292)(45,280)

\put(-73,275){\scriptsize$D_2$}
\put(52,275){\scriptsize$p+1$}

\qbezier(-32,293)(-29,283)(-31,270)
\qbezier(-31,293)(-28,283)(-30,270)
\put(-44,296){$\frac{p+1}{2}$}

\qbezier(20,301)(18,283)(21,272)
\put(21,302){\scriptsize$1$}

\qbezier(-49,306)(5,321)(46,309)
\qbezier(-49,307)(5,322)(46,310)
\qbezier(-49,308)(5,323)(46,311)

\put(-73,305){\scriptsize$D_1$}
\put(52,305){\scriptsize$p+1$}

\qbezier(-27,320)(-24,308)(-26,295)
\qbezier(-26,320)(-23,308)(-25,295)
\put(-29,323){$\frac{p+1}{2}$}

\qbezier(30,322)(28,316)(31,295)
\put(30,323){\scriptsize$1$}

\qbezier(-48,217)(5,233)(59,218)
\qbezier(-48,216)(5,232)(59,217)
\qbezier(-48,215)(5,231)(59,216)

\put(-66,215){\scriptsize$D_k$}
\put(64,217){\scriptsize$p+1$}

\qbezier(-25,245)(-23,224)(-26,206)
\qbezier(-24,245)(-22,224)(-25,206)
\put(-27,248){$\frac{p+1}{2}$}

\qbezier(27,245)(24,222)(26,212)
\put(28,250){\scriptsize$1$}


\qbezier(0,373)(-5,270)(-2,200) 
\qbezier(1,373)(-4,270)(-1,200)
\qbezier(-1,373)(-6,270)(-3,200)

\put(-9,384){$\frac{p-1}{2}$}



\put(20,340){\small{$j\equiv 0$}}
\put(17,342){\vector(-1,1){13}}
\qbezier(-25,358)(5,357)(29,359)
\qbezier(-25,357)(5,356)(29,358)
\put(34,353){$\frac{p-1}{3}$}

\qbezier(-17,370)(-14,368)(-16,345)
\qbezier(-16,370)(-13,368)(-15,345)
\put(-32,373){$\frac{p-1}{6}$}


\put(-45,+270){$\cdot$}
\put(-45,253){$\cdot$}
\put(-45,234){$\cdot$}

\put(42,266){$\cdot$}
\put(43,247){$\cdot$}
\put(44,230){$\cdot$}


\put(-20,175){$p=12k+1$}







\qbezier(188,338)(245,341)(290,333)
\qbezier(188,337)(245,340)(290,332)
\put(164,334){$\frac{p+1}{3}$}

\qbezier(210,352)(212,343)(209,328)
\qbezier(211,352)(213,343)(210,328)
\put(206,355){$\frac{p+1}{6}$}
\qbezier(256,353)(252,342)(253,327)
\put(257,354){\scriptsize$1$}

\put(214,318){\scriptsize{$j\equiv 0$}}
\put(226,325){\vector(1,1){10}}


\qbezier(182,306)(245,316)(295,301)
\qbezier(182,305)(245,315)(295,300)
\qbezier(182,304)(245,314)(295,299)
\put(305,297){\scriptsize$D_1$}
\put(157,302){\scriptsize$p+1$}

\qbezier(180,279)(245,286)(296,270)
\qbezier(180,278)(245,285)(295,269)
\qbezier(180,277)(245,284)(295,268)
\put(305,268){\scriptsize$D_2$}
\put(157,274){\scriptsize$p+1$}

\qbezier(188,226)(245,236)(295,226)
\qbezier(188,225)(245,235)(295,225)
\qbezier(188,224)(245,234)(295,224)
\put(305,220){\scriptsize$D_k$}
\put(164,221){\scriptsize$p+1$}

\put(221,271){$\cdot$}
\put(222,253){$\cdot$}
\put(224,234){$\cdot$}


\qbezier(216,264)(216,273)(212,292)
\qbezier(217,264)(217,273)(213,292)
\put(207,296){$\frac{p+1}{2}$}

\qbezier(200,293)(201,305)(196,322)
\qbezier(199,293)(200,305)(195,322)
\put(192,324){$\frac{p+1}{2}$}

\qbezier(208,207)(211,225)(208,248)
\qbezier(207,207)(210,225)(207,248)
\put(202,253){$\frac{p+1}{2}$}


\qbezier(268,267)(262,275)(263,292)
\put(258,296){\scriptsize$1$}

\qbezier(272,293)(266,305)(268,322)
\put(263,324){\scriptsize$1$}

\qbezier(275,207)(263,225)(268,248)
\put(270,253){\scriptsize$1$}


\put(232,380){$\frac{p-1}{2}$}

\qbezier(241,370)(240,270)(234,200)
\qbezier(240,370)(239,270)(233,200)
\qbezier(239,370)(238,270)(232,200)

%


\put(215,175){$p=12k+5$}







\qbezier(-50,10)(5,20)(45,8)
\qbezier(-50,11)(5,21)(45,9)
\qbezier(-50,12)(5,22)(45,10)

\put(-73,5){\scriptsize$D_2$}
\put(54,5){\scriptsize$p+1$}

\qbezier(-35,23)(-32,13)(-34,0)
\qbezier(-36,23)(-33,13)(-35,0)
\put(-48,25){$\frac{p+1}{2}$}

\qbezier(20,31)(18,13)(21,2)
\put(21,32){\scriptsize$1$}

\qbezier(-49,36)(5,51)(46,39)
\qbezier(-49,37)(5,52)(46,40)
\qbezier(-49,38)(5,53)(46,41)

\put(-73,35){\scriptsize$D_1$}
\put(54,35){\scriptsize$p+1$}

\qbezier(-27,50)(-24,38)(-26,25)
\qbezier(-26,50)(-23,38)(-25,25)
\put(-29,53){$\frac{p+1}{2}$}

\qbezier(30,52)(28,46)(31,25)
\put(30,53){\scriptsize$1$}

\qbezier(-48,-47)(5,-37)(59,-52)
\qbezier(-48,-46)(5,-36)(59,-51)
\qbezier(-48,-45)(5,-35)(59,-50)

\put(-68,-52){\scriptsize$D_k$}
\put(65,-53){\scriptsize$p+1$}

\qbezier(-24,-25)(-22,-46)(-25,-57)
\qbezier(-25,-25)(-23,-46)(-26,-57)
\put(-24,-22){$\frac{p+1}{2}$}

\qbezier(27,-25)(24,-48)(26,-58)
\put(28,-20){\scriptsize$1$}


\qbezier(0,123)(-5,0)(15,-70) 
\qbezier(1,123)(-4,0)(16,-70)
\qbezier(-1,123)(-6,0)(14,-70)

\put(-9,131){$\frac{p-1}{2}$}



\put(16,104){\small{$j\equiv 0$}}
\put(15,105){\vector(-2,-1){12}}
\qbezier(-28,98)(5,96)(31,99)
\qbezier(-28,97)(5,95)(31,98)
\put(36,92){$\frac{p-1}{3}$}

\qbezier(-17,105)(-14,93)(-16,80)
\qbezier(-16,105)(-13,93)(-15,80)
\put(-32,113){$\frac{p-1}{6}$}



\qbezier(-40,70)(0,78)(40,74)
\put(43,74){\scriptsize$1$}

\put(20,64){\small{$j\equiv 1728$}}
\put(16,67){\vector(-2,1){10}}


\put(-45,-0){$\cdot$}
\put(-45,-17){$\cdot$}
\put(-45,-36){$\cdot$}

 \put(42,-4){$\cdot$}
\put(43,-23){$\cdot$}
\put(44,-40){$\cdot$}


\put(-20,-95){$p= 12k+7$}




\qbezier(190,99)(245,105)(289,98)
\put(178,97){\scriptsize$1$}

\put(258,108){\small{$j\equiv 1728$}}
\put(255,109){\vector(-2,-1){12}}



\qbezier(188,68)(245,71)(290,63)
\qbezier(188,67)(245,70)(290,62)

\put(164,64){$\frac{p+1}{3}$}

\put(213,48){\scriptsize{$j\equiv 0$}}
\put(225,54){\vector(1,1){10}}

\qbezier(210,82)(212,73)(209,58)
\qbezier(211,82)(213,73)(210,58)
\put(208,85){$\frac{p+1}{\mathrm{6}}$}

\qbezier(256,83)(252,72)(253,57)
\put(259,84){\scriptsize$1$}


\qbezier(182,36)(245,46)(295,31)
\qbezier(182,35)(245,45)(295,30)
\qbezier(182,34)(245,44)(295,29)
\put(305,27){\scriptsize$D_1$}
\put(157,32){\scriptsize$p+1$}

\qbezier(180,9)(245,16)(296,0)
\qbezier(180,8)(245,15)(295,-1)
\qbezier(180,7)(245,14)(295,-2)
\put(305,-2){\scriptsize$D_2$}
\put(156,4){\scriptsize$p+1$}

\qbezier(188,-44)(245,-34)(295,-44)
\qbezier(188,-45)(245,-35)(295,-45)
\qbezier(188,-46)(245,-36)(295,-46)
\put(305,-50){\scriptsize$D_k$}
\put(164,-49){\scriptsize$p+1$}

\put(221,1){$\cdot$}
\put(222,-17){$\cdot$}
\put(224,-36){$\cdot$}


\qbezier(216,-6)(216,13)(212,22)
\qbezier(215,-6)(215,13)(211,22)
\put(207,26){$\frac{p+1}{2}$}

\qbezier(202,23)(203,35)(198,52)
\qbezier(201,23)(202,35)(197,52)
\put(186,52){$\frac{p+1}{2}$}

\qbezier(208,-63)(211,-45)(208,-22)
\qbezier(209,-63)(212,-45)(209,-22)
\put(202,-17){$\frac{p+1}{2}$}


\qbezier(268,-3)(262,5)(263,22)
\put(258,26){\scriptsize$1$}

\qbezier(272,23)(266,35)(268,52)
\put(263,54){\scriptsize$1$}

\qbezier(275,-63)(263,-45)(268,-22)
\put(270,-17){\scriptsize$1$}


\put(232,132){$\frac{p-1}{2}$}

\qbezier(241,122)(236,0)(246,-70)
\qbezier(240,122)(235,0)(245,-70)
\qbezier(239,123)(234,0)(244,-70)

%


\put(215,-95){$p=12k+11$}


\end{picture}
\end{center}
\caption{Special fiber above $\overline{\F}_p$ of the minimal regular model with normal crossings of $X_{\mathrm{ns}}^+ (p)$}  
\label{figureNSgrossier+}
\end{figure}
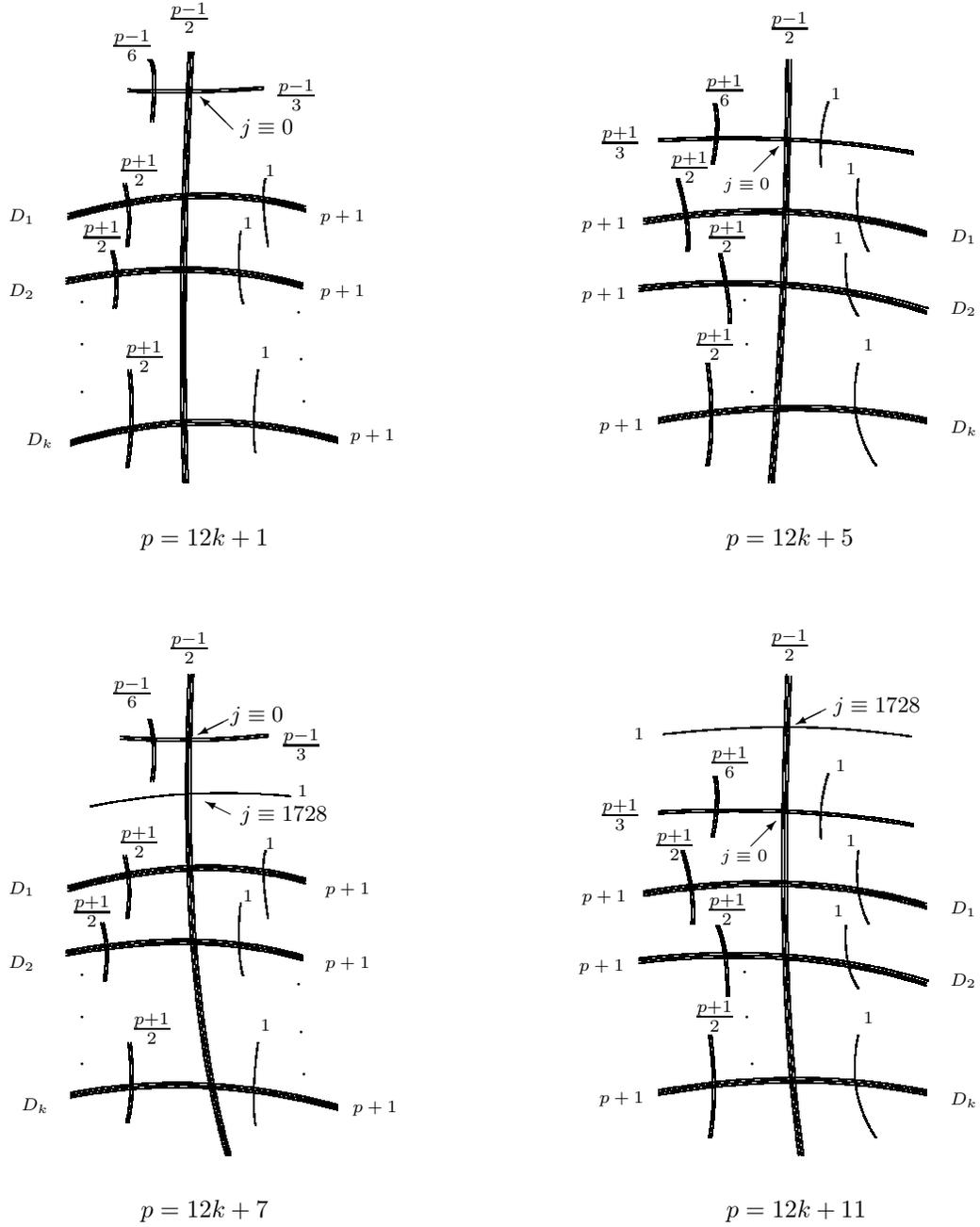

\begin{proof} 

As for Theorem~\ref{thNSgrossier} above, we start from the semistable model $X_{\mathrm{ns}}^+ (p)^{\mathrm{st}}$ over 
$\Z_p^{\mathrm{ur}} [p^{2/(p^2 -1)}]$ of Theorem~3.5 in~\cite{EdixP20}, and quotient out by the Galois 
group $G ={\mathrm{Gal}} (\Q_p^{\mathrm{ur}} (p^{2/(p^2 -1)})/\Q_p^{\mathrm{ur}} )\simeq \F_{p^2}^* /\{ \pm 1\}$, 
taking care of points with exceptional automorphisms. 

\medskip

    The case where $j\equiv 0$ is ordinary $\mod p$ ($\equiv 1\mod 3$) is not quite the same as what happens for 
$X_{\mathrm{ns}} (p)$ in~Theorem~\ref{thNSgrossier}, because of Galois action.
There is still a total of $(p-1)/6$ points with that $j$-invariant on the (either $1$- or $2$-components) Igusa locus of 
our semistable model, which is $\overline{M} ({\mathrm{Ig}} (p)/\{\pm 1\} )_{\overline{\F}_p} $ if $p\equiv -1 
\mod 4$ and $\overline{M} ({\mathrm{Ig}} (p)/C_4 )_{\overline{\F}_p}$ if $p\equiv 1 \mod 4$. As in the proof of 
Theorem~\ref{thNS+}, those Igusa parts can be descended over 
$\Z_ p^{\mathrm{ur}} [\pi_2^2 =p^{2/(p -1)}]$, and then the inertia of the Galois action at the points for which 
$j\equiv 0$ is the order $3$ subgroup of $\F_{p^2}^* /\mu_{2(p+1)} (\F_{p^2} )\simeq \F_p^* /\{ \pm 1\}$, 
which is spanned by $\zeta_3 := u^{(p-1)/3} \mod \mu_{2(p+1)} (\F_{p^2} )$ 
for $u$ our generator of $\F_{p^2}^*$. Then $\zeta_3 \colon  \pi_2^2 \mapsto u^{2(p^2 -1)/3} \pi_2^2$. 
On the other hand, Proposition~\ref{stabilisateur} implies that the Galois action of $\zeta_3$ on the Igusa component is 
$(E,P)\mapsto (E,\langle u^{(p^2 -1)/6} \rangle_p P)$ (see (\ref{lesordres6})),
which is the same as that of the exceptional automorphism $[u^{(p^2 -1)/6} ]$ (action of a $6^{\mathrm{th}}$ root of unity) 
on the curve $E_0$ with $j$-invariant $0$ on $\overline{\F}_p$. As in the proof of Theorem~\ref{thNSgrossier}, from (7) 
in~\cite{EdixP20}, end of Section~2.2.2, one sees that the action of $\zeta_3$ on the universal deformation space 
$\overline{\F}_p [[t]]$ at the elliptic curve $E_0$ with $j$-invariant $0$ is given by $t\mapsto (u^{(p^2 -1)/6} )^2 
t=u^{(p^2 -1)/3} t$. Summing up, if $\Z_p^{\mathrm{ur}} [\pi_2^2 ] [[t]]$ is the completed local ring at such a point 
over $\Z_p^{\mathrm{ur}} [\pi_2^2 ]$, then the inertia subgroup of Galois can be described as
\begin{eqnarray}
\label{lactionzeta3} 
\zeta_3 \colon \pi_2^2 \mapsto \zeta_3^2 \pi_2^2 ,\  t\mapsto \zeta_3 t.
\end{eqnarray} 
Notice the difference with the analogous situation~(\ref{lactionomega3}) of Theorem~\ref{thNS+} as now the Galois action on 
the tangent space at our exceptional points is {\em not} diagonal. In any case, that action is not by pseudo-reflections 
either, so the quotient is not regular, and we blow it up. Over $\Z_p^{\mathrm{ur}} [\pi_6 :=\pi_2^6 =p^{6/(p-1)}]$, the 
Igusa quotient has multiplicity $3$, and looking at invariants shows that, putting $X=t^3$ and $Y=\pi_2^2 t$, each of our 
singularity quotients with $j\equiv 0$ there has local ring with same completion as 
$$
\Z_p^{\mathrm{ur}} [\pi_6 ][X,Y]/(Y^ 3 -\pi_6 X).
$$
Blowing it up at ${\cal M}=(\pi_6 ,X,Y)$ produces three affine schemes with rings 
$$
\Z_p^{\mathrm{ur}} [\pi_6 ][X_1 ,Y_1 ]/(X_1^2 Y_1^3 -\pi_6 ), \hspace{0.5cm} \Z_p^{\mathrm{ur}} [\pi_6 ][X_2 ,Y_2 ]/
(X_2^2 Y_2 -\pi_6 )\hspace{0.2cm} {\mathrm{\ and\ }} \hspace{0.2cm} \Z_p^{\mathrm{ur}} [\pi_6 ][X_3 ]
$$
(with $X_1 :=X$, $Y_1 := Y/X$; $X_2 := \pi_6 /Y$, $Y_2 :=X/Y$; and $X_3 :=Y/\pi_6$).

In other words, our blowing up creates, at each point with $j\equiv 0\mod p$, a chain consisting of a projective line with 
multiplicity $2$, on which stems another projective line with multiplicity $1$. Now going all the way down to the quotient 
over $\Z_p^{\mathrm{ur}}$, 
all our $(p-1)/6$ points are permuted (regardless of whether $p$ is $1$ or $-1$ mod $4$), so that one eventually obtains a 
single point with $j\equiv 0\mod 3$ on the Igusa $j$-line, from which goes out a chain consisting of a $\PPP^1$ with 
multiplicity $(p-1)/3$, followed by another one with multiplicity $(p-1)/6$. The crossing points have the expected shape.  
 
\medskip

     When $j\equiv 1728$ is ordinary mod $p$ ($\equiv 1\mod 4$), then: nothing happens. (The Galois quotient of the 
Igusa component is regular at those special points.) Indeed, compare with the proof of Theorems~\ref{thNSgrossier}, around 
(\ref{voiciPlain}): over $\Z_p^{\mathrm{ur}} [\pi_2 :=p^{1/(p-1)} ]$, the local ring at some point 
with $j\equiv 1728$ on an Igusa component $\overline{M} ({\mathrm{Ig}} (p)/C_4 )_{\overline{\F}_p}$ is now
$\Z_p^{\mathrm{ur}} [\pi_2] [[x]]$, with Galois action of the quadratic subextension: $\pi_2 \mapsto -\pi_2$, $x\mapsto 
x$. Whence the regularity.

\medskip

   Now if $j\equiv 1728$ is supersingular mod $p$ ($\equiv -1\mod 4$), then as in the proof of 
Theorem~\ref{thNSgrossier}, one checks through Theorem~\ref{thNS+}
above and its proof, and Theorem~3.5 of~\cite{EdixP20}, that a parameter on the Drinfeld component (which is a 
projective line) in the special fiber of the semistable model $X_{\mathrm{ns}}^+ (p)^{\mathrm{st}}$ 
over $\Z_p^{\mathrm{ur}} [p^{2/(p^2 -1)}]$ is 
${\cal X} =V^2 =({\tilde{\alpha}} \tilde{\beta} )^2$ (with notation as in~(\ref{DrinfeldXns}), or as in proof of 
Theorem~3.5 of~\cite{EdixP20}). Therefore, if $u$ is a generator of  $\F_{p^2}^*$, 
then on our exceptional Drinfeld component, the Galois action is defined by $u\colon {\cal X}\mapsto u^{2(p-1)} 
{\cal X}$ (cf. (\ref{actionDeGalois})). Therefore, the Galois quotient, over $\Z_p^{\mathrm{ur}}$, 
has a special Drinfeld component which is a projective line with multiplicity $(p+1)/2$. On the Drinfeld  component 
above (over $\Z_p^{\mathrm{ur}} [p^{2/(p^2 -1)}]$) (which again is a projective line with multiplicity $1$), that 
Galois action has exactly two fixed points(at ${\cal X}=0$ and ${\cal X}=\infty$). One has to be the 
intersection point with the Igusa component. The other one creates a non regular point on 
the quotient Drinfeld curve over $\Z_p^{\mathrm{ur}}$. An avatar on the computations seen many 
times now (cf. e.g.~(\ref{laclassiquederegularisation}) and below)
shows that blowing up at that point causes a projective line with multiplicity $1$ to appear. Summing-up: at the 
supersingular point $j\equiv 1728$, some projective line, say $D_{-1}$, with multiplicity $(p+1)/2$, intersects our Igusa 
(vertical) component, and some other projective line (say ${\cal E}_{-1}$) with multiplicity $1$ pops-up at some other point 
of $D_{-1}$. The resulting scheme is now regular with normal crossings, but not {\em minimal} with such properties: indeed, 
one readily checks that $D_{-1}$ has self-intersection $-1$. Contracting it, only ${\cal E}_{-1}$ survives, and we now have 
the desired minimality property.

\medskip

   Finally, let us consider the case when $j\equiv 0$ is supersingular mod $p$ ($\equiv -1\mod 3$). With notation as in the 
proof of Theorem~3.5 of~\cite{EdixP20}, the exceptional Drinfeld component above $j\equiv 0$, for the curve over 
$\Z_p^{\mathrm{ur}} [p^{2/(p^2 -1)}]$, has equation 
$$
{\cal Y}^2 ={\cal X} ({\cal X}^{\frac{p+1}{6}} +A_{\mathrm{ns}} )
$$
with ${\cal X}=X^3 =V^6$ and ${\cal Y}= XY=UV^3$.
With the same $u$ as just before that means that Galois acts by $u\colon {\cal X}\mapsto u^{6(p-1)}{\cal X}, 
{\cal Y}\mapsto u^{3(p-1)}{\cal Y}$ (see~(\ref{DrinfeldXns}) and around). This first shows that our exceptional Drinfeld 
component is a projective line (with parameter ${\cal Z}:={\cal Y}^2 /{\cal X}$) with multiplicity $(p+1)/3$. Then the same 
reasoning as at the end of the proof of Theorem~\ref{thNS+} 
(see~(\ref{DrinfeldNS+1}) and nearby) shows that after blowing up, two extra projective lines, one with multiplicity 
$1$, and one with multiplicity $(p+1)/6$, arise on that last Drinfeld component with $j\equiv 0$.  
 
\medskip

    As at the end of proof of Theorem~\ref{thNSgrossier}, the assertion that the regular models we obtained are the {\em 
minimal} regular models with normal crossings (at least when $p>5$) follows from the fact that no component can be 
contracted without losing the ``normal crossings'' property. Again, details will be displayed in the proof of next 
Corollary~\ref{minimalXns^+}. $\Box$
  
\end{proof}

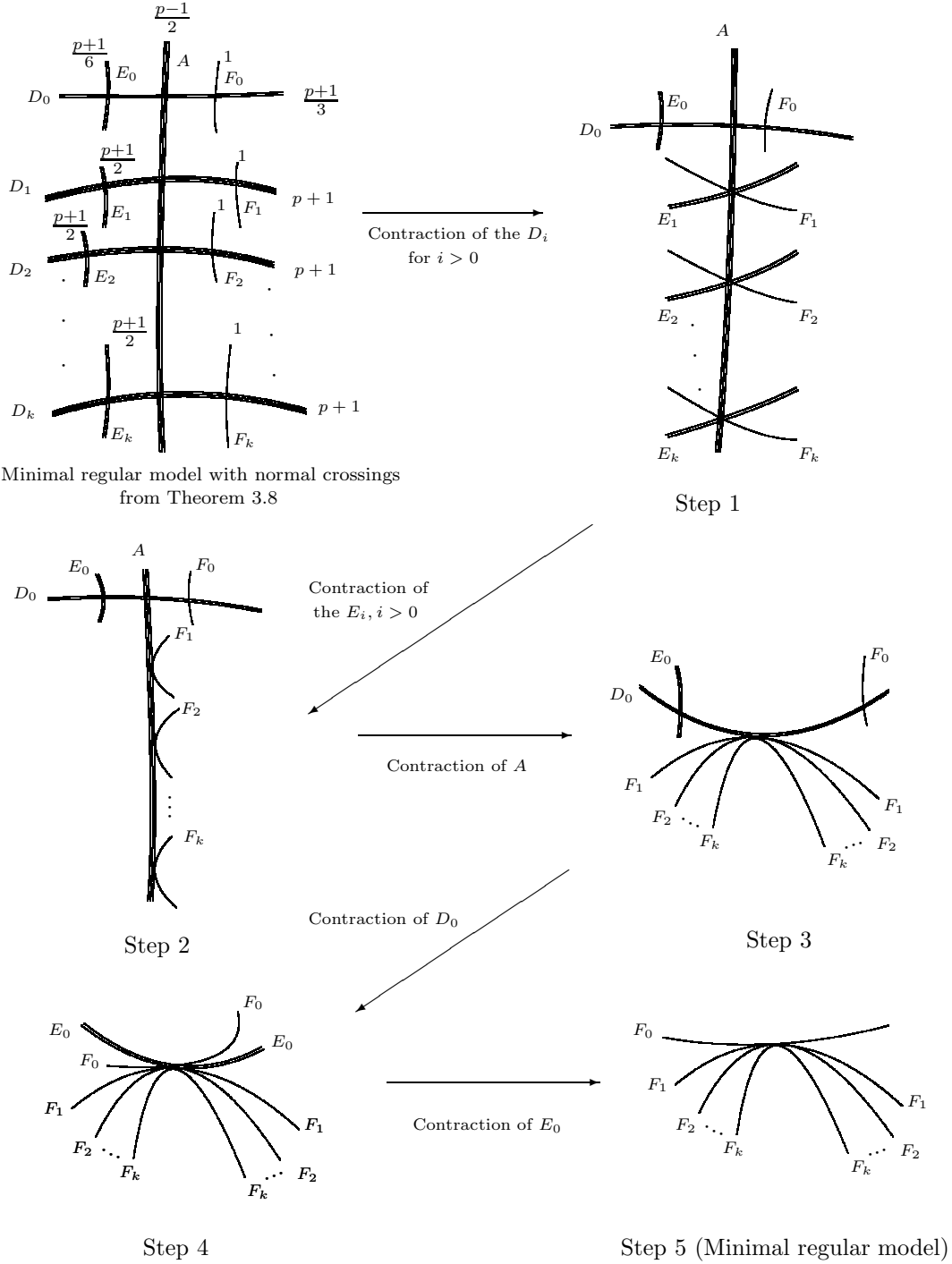
\begin{figure}
\begin{center}



\begin{picture}(320,510)(-30,-140)

\qbezier(-51,306)(5,321)(46,309)
\qbezier(-51,307)(5,322)(46,310)
\qbezier(-51,308)(5,323)(46,311)

\put(-67,310){\scriptsize$D_1$}
\put(53,305){\scriptsize$p+1$}

\qbezier(-27,320)(-24,308)(-26,295)
\qbezier(-28,320)(-25,308)(-27,295)
\put(-29,321){$\frac{p+1}{2}$}
\put(-24,299){\scriptsize$E_1$}

\qbezier(30,322)(28,316)(31,295)
\put(30,323){\scriptsize$1$}
\put(32,301){\scriptsize$F_1$}


\qbezier(-50,280)(5,290)(45,278)
\qbezier(-50,281)(5,291)(45,279)
\qbezier(-50,282)(5,292)(45,280)

\put(-67,275){\scriptsize$D_2$}
\put(54,275){\scriptsize$p+1$}

\qbezier(-35,293)(-32,283)(-34,270)
\qbezier(-36,293)(-33,283)(-35,270)
\put(-49,293){$\frac{p+1}{2}$}
\put(-31,272){\scriptsize$E_2$}

\qbezier(20,301)(18,283)(21,272)
\put(21,302){\scriptsize$1$}
\put(24,271){\scriptsize$F_2$}

\qbezier(-48,217)(5,233)(59,218)
\qbezier(-48,216)(5,232)(59,217)
\qbezier(-48,215)(5,231)(59,216)

\put(-66,215){\scriptsize$D_k$}
\put(64,217){\scriptsize$p+1$}

\qbezier(-26,245)(-24,224)(-27,206)
\qbezier(-25,245)(-23,224)(-26,206)
\put(-24,248){$\frac{p+1}{2}$}
\put(-24,205){\scriptsize$E_k$}

\qbezier(27,245)(24,222)(26,202)
\put(28,250){\scriptsize$1$}
\put(28,203){\scriptsize$F_k$}


\qbezier(0,373)(-5,260)(-2,200) 
\qbezier(1,373)(-4,260)(-1,200)
\qbezier(-1,373)(-6,260)(-3,200)

\put(-6,381){$\frac{p-1}{2}$}
\put(4,363){\scriptsize$A$}

\qbezier(-45,351)(5,350)(49,352)
\qbezier(-45,350)(5,349)(49,351)
\put(57,346){$\frac{p+1}{3}$}
\put(-59,348){\scriptsize$D_0$}

\qbezier(-25,365)(-23,354)(-26,336)
\qbezier(-26,365)(-24,354)(-27,336)
\put(-41,366){$\frac{p+1}{6}$}
\put(-22,358){\scriptsize$E_0$}

\qbezier(22,365)(19,342)(21,335)
\put(24,365){\scriptsize$1$}
\put(24,356){\scriptsize$F_0$}


\put(-45,+270){$\cdot$}
\put(-45,253){$\cdot$}
\put(-45,234){$\cdot$}

\put(42,266){$\cdot$}
\put(43,247){$\cdot$}
\put(44,230){$\cdot$}


\put(-70,188){\footnotesize{Minimal regular model with normal crossings}} 
\put(-20,178){\footnotesize{from Theorem~\ref{thNS+grossier}}}


 
\put(83,301){\vector(1,0){76}}

\put(85,290){\scriptsize{Contraction of the $D_i$}}
\put(102,280){\scriptsize{for $i>0$}}



\qbezier(188,338)(245,341)(290,333)
\qbezier(188,337)(245,340)(290,332)
\put(174,334){\scriptsize$D_0$}

\qbezier(208,352)(210,343)(207,328)
\qbezier(209,352)(211,343)(208,328)
\put(211,346){\scriptsize$E_0$}

\qbezier(256,353)(252,342)(253,327)
\put(258,345){\scriptsize$F_0$}

\qbezier(212,304)(246,310)(266,322)
\qbezier(213,303)(247,309)(267,321)
\put(267,296){\scriptsize$F_1$}

\qbezier(211,324)(246,304)(266,302)
\put(207,296){\scriptsize$E_1$}

\qbezier(212,285)(246,266)(266,263)
\put(267,255){\scriptsize$F_2$}

\qbezier(211,265)(246,272)(266,285)
\qbezier(212,264)(247,271)(267,284)
\put(207,255){\scriptsize$E_2$}

\put(221,251){$\cdot$}
\put(222,238){$\cdot$}
\put(224,224){$\cdot$}


\qbezier(212,227)(246,205)(266,205)
\put(267,197){\scriptsize$F_k$}

\qbezier(211,207)(246,217)(266,227)
\qbezier(212,206)(247,216)(267,226)
\put(207,197){\scriptsize$E_k$}


\put(232,376){\scriptsize$A$}

\qbezier(241,370)(240,270)(234,200)
\qbezier(240,370)(239,270)(233,200)
\qbezier(239,370)(238,270)(232,200)

%


\put(215,175){Step 1}


 
\put(180,169){\vector(-3,-2){120}}

\put(60,140){\scriptsize{Contraction of}}
\put(62,130){\scriptsize{the $E_i , i>0$}}



\qbezier(-50,138)(-0,141)(40,133)
\qbezier(-50,137)(0,140)(40,132)
\put(-65,138){\scriptsize$D_0$}

\qbezier(-30,148)(-25,138)(-29,128)
\qbezier(-29,148)(-24,138)(-28,128)
\put(-42,149){\scriptsize$E_0$}

\qbezier(10,149)(8,138)(11,127)
\put(12,150){\scriptsize$F_0$}


\qbezier(1,122)(-14,107)(3,96)
\put(3,121){\scriptsize$F_1$}

\qbezier(5,91)(-13,78)(2,62)
\put(6,89){\scriptsize$F_2$}

\put(0,51){$\cdot$}
\put(0,47){$\cdot$}
\put(0,43){$\cdot$}


\qbezier(2,37)(-13,22)(4,7)
\put(7,35){\scriptsize$F_k$}

\put(-15,156){\scriptsize$A$}


\qbezier(-10,150)(-5,50)(-8,10)
\qbezier(-9,150)(-4,50)(-7,10)
\qbezier(-8,150)(-3,50)(-6,10)

%


\put(-18,-12){Step 2}

\put(81,80){\vector(1,0){90}}
\put(93,65){\scriptsize{Contraction of $A$}}


\qbezier(200,100)(250,60)(305,98)
\qbezier(200,101)(250,61)(305,99)

\put(187,95){\scriptsize$D_0$}

\qbezier(215,109)(218,99)(216,79)
\qbezier(216,109)(219,99)(217,79)
\put(204,111){\scriptsize$E_0$}

\qbezier(295,113)(293,95)(296,84)
\put(297,112){\scriptsize$F_0$}


\qbezier(205,62)(250,100)(301,53)
\put(193,57){\scriptsize$F_1$} 
\put(303,48){\scriptsize$F_1$}


\qbezier(215,50)(246,113)(297,40)
\put(205,43){\scriptsize$F_2$} 
\put(300,32){\scriptsize$F_2$} 


\put(217,42){$\cdot$}
\put(220,40){$\cdot$} 
\put(223,39){$\cdot$}  


\qbezier(231,41)(245,120)(278,33)
\put(225,33){\scriptsize$F_k$}
\put(279,25){\scriptsize$F_k$}


\put(286,30){$\cdot$}
\put(289,31){$\cdot$} 
\put(292,32){$\cdot$}  


\put(245,-10){Step 3}



 
\put(170,23){\vector(-3,-2){90}}

\put(060,-0){\scriptsize{Contraction of $D_0$}}
 


\qbezier(-35,-42)(3,-72)(40,-52)
\qbezier(-36,-43)(003,-73)(41,-53) 
\put(-50,-47){\scriptsize$E_0$}
\put(44,-52){\scriptsize$E_0$}


\qbezier(-25,-60)(36,-64)(30,-37)
\put(-37,-60){\scriptsize$F_0$}
\put(32,-35){\scriptsize$F_0$}


\qbezier(-40,-78)(5,-40)(56,-87)
\put(-52,-80){\scriptsize$F_1$} 
\put(58,-87){\scriptsize$F_1$}


\qbezier(-30,-90)(1,-27)(48,-100)
\put(-40,-97){\scriptsize$F_2$} 
\put(55,-108){\scriptsize$F_2$} 

\put(-28,-98){$\cdot$}
\put(-25,-100){$\cdot$} 
\put(-22,-101){$\cdot$}  


\qbezier(-14,-99)(-2,-20)(33,-107)
\put(-20,-107){\scriptsize$F_k$}
\put(34,-115){\scriptsize$F_k$}


\put(41,-110){$\cdot$}
\put(44,-109){$\cdot$} 
\put(47,-108){$\cdot$}  



\put(-10,-140){Step 4}


\put(93,-67){\vector(1,0){90}}
\put(104,-87){\scriptsize{Contraction of $E_0$}}


\qbezier(210,-48)(255,-56)(305,-43)

\put(197,-47){\scriptsize$F_0$}


\qbezier(215,-68)(260,-30)(311,-77)
\put(203,-70){\scriptsize$F_1$} 
\put(313,-77){\scriptsize$F_1$}

\qbezier(225,-80)(256,-17)(307,-90)
\put(215,-87){\scriptsize$F_2$} 
\put(310,-98){\scriptsize$F_2$} 


\put(227,-88){$\cdot$}
\put(230,-90){$\cdot$} 
\put(233,-91){$\cdot$}  


\qbezier(241,-89)(253,-10)(288,-97)
\put(235,-97){\scriptsize$F_k$}
\put(289,-105){\scriptsize$F_k$}


\put(296,-100){$\cdot$}
\put(299,-99){$\cdot$} 
\put(302,-98){$\cdot$}



\qbezier(-40,-78)(5,-40)(56,-87)
\put(-52,-80){\scriptsize$F_1$} 
\put(58,-87){\scriptsize$F_1$}


\qbezier(-30,-90)(1,-27)(48,-100)
\put(-40,-97){\scriptsize$F_2$} 
\put(55,-108){\scriptsize$F_2$} 

\put(-28,-98){$\cdot$}
\put(-25,-100){$\cdot$} 
\put(-22,-101){$\cdot$}  


\qbezier(-14,-99)(-2,-20)(33,-107)
\put(-20,-107){\scriptsize$F_k$}
\put(34,-115){\scriptsize$F_k$}


\put(41,-110){$\cdot$}
\put(44,-109){$\cdot$} 
\put(47,-108){$\cdot$}  


\put(192,-140){Step 5 (Minimal regular model)}

 
\end{picture}
\end{center}

\caption{Minimal resolution of $X_{\mathrm{ns}}^+ (p)$ in the case $p=12k+5$} 
\label{figureNS+minimal5}
\end{figure}

\begin{coro}
\label{minimalXns^+}
Let $p\geq 5$. If $p=12k+i$, for $i=1, 5, 7$ or $11$, set $n(p) =k$ if $i=1$, $n (p) =k+1$ if $i=5$ or $i=7$, 
and $n(p) =k+2$ when $i=11$.

    The {\em minimal regular model} of $X_{\mathrm{ns}}^+ (p)$ over $\Z_p^{\mathrm{ur}}$, has a reduced fiber,  
which is made up of $n(p)$ components (which is precisely the number of supersingular invariants), intersecting at one common 
singular point. (As in Corollary~\ref{minimalXns}, components are quotients of the projective line, which might be singular 
at the common intersection point.) See Figure~\ref{figureNS+minimal5}.

\end{coro}

\begin{rema}
\label{reduitpasreduit}
{\rm  As already noticed, $X_{\mathrm{ns}}^+ (p)\simeq \PPP^1_\Q$ for $p=2, 3, 5$ and $7$; moreover, 
$X_{\mathrm{ns}}^+ (11)$ is known to be an elliptic curve of Kodaira type III from the work of Ligozat (\cite{Lg77}), 
which can also be found by our computations.

  Note also that, whereas the Galois quotient of the semistable model for $X_{\mathrm{ns}}^+ (p)$ we started with has no 
reduced component (at all), the minimal regular model is, to the contrary, totally reduced... Of course we knew a priori that 
there had to be reduced components in order to host the specialization of points with values in $\Q$, which are known to 
exists in many cases (CM points). (Actually, for all $p$, CM elliptic curves even furnish an infinite number of CM points 
with values in $\Q_p^{\mathrm{ur}}$ which specialize to any of those reduced components.)}
\end{rema} 

\begin{proof}

This is very similar to the proof of Corollary~\ref{minimalXns}. Again we give a picture of the minimal resolution in the 
case $p=12k+5$ (see 
Figure~\ref{figureNS+minimal5}): using Castelnuovo's criterion, one first contracts the horizontal components with 
multiplicity $(p+1)$, then those with multiplicity $(p+1)/2$, then the vertical Igusa component, the  
horizontal one with multiplicity $(p+1)/3$, and finally that with multiplicity $(p+1)/6$. The other cases for $p$ mod $12$ 
are settled in the same fashion.

    As in the proof of 
Corollary~\ref{minimalXns}, one can remark that writing down the intersection matrix of our minimal regular model with normal 
crossings, and performing a Smith normalization process which parallels our successive contractions, one obtains the 
intersection matrix of our minimal regular model, which for~$X_{\mathrm{ns}}^+ (p)$ say with $p=12k+5$, has shape in the
components basis $\{ F_0 , F_1 , \dots , F_k \}$ of Figure~\ref{figureNS+minimal5}:
\begin{equation*}
\left(
\begin{array}{ccccc}
(5-p)/3 & 4  & 4 &  \cdots  & 4 \\ 
4 & 13-p  & 12 & \cdots  & 12 \\
\vdots & \vdots &  \vdots & \vdots & \vdots \\
4 & 12 &  \cdots & 13-p  & 12 \\
4 & 12 &  \cdots & 12 & 13-p 
\end{array}
\right)
\end{equation*}

$\Box$

\end{proof}

\subsubsection{Component groups for ${J}_{\mathrm{ns}}^+ (p)$}

\begin{propo}
\label{componentsXns+}
Let $p\ge 17$ be a prime, and denote by ${\mathcal J}_{\mathrm{ns}}^+ (p)$ the N\'eron model over $\Z_p^{\mathrm{ur}}$ of the 
Jacobian of $X_{\mathrm{ns}}^+ (p)$. The component group ${\mathcal J}_{\mathrm{ns}}^+ (p)/{\mathcal J}^{+,0}_{\mathrm{ns}} (p)$ 
of the special fiber of ${\mathcal J}_{\mathrm{ns}}^+ (p)$ is:
\begin{itemize}
\item $\Z /12\Z \times \left( \Z /(p -1)\Z \right)^{(p-25)/12}$ if $p=12k+1$;
\item $\Z /4\Z \times \left( \Z /(p -1)\Z \right)^{(p-17)/12}$ if $p=12k+5$;
\item $\Z /6\Z \times \left( \Z /(p -1)\Z \right)^{(p-19)/12}$ if $p=12k+7$;
\item $\Z /2\Z \times \left( \Z /(p -1)\Z \right)^{(p-11)/12}$ if $p=12k+11$.
\end{itemize}
When $p=11$, the component group is $\Z /2\Z $, and it is trivial when $p= 2, 3, 5, 7$ or $13$.
\end{propo}

\begin{proof}
The arguments are the same as those of Proposition~\ref{componentsXns}. Note that for small primes, the zero group of 
components also trivially follows from the fact that the genus of $X_{\mathrm{ns}}^+ (p)$ is $0$ for $p=5, 7$ (and $p=2,3$
of course). It is however $3$ for $p=13$ (and $1$ if $p=11$); see for instance p.~117 of \cite{Ma76}. $\Box$

\end{proof}

\section{Regular models for split Cartan structures}
\label{SplitCartan}

Computations of regular models for split Cartan modular curves in prime level have already been performed, under 
the guise of $\Gamma_0 (p^2 )$-structure, by the first-named author, in~\cite{Edix89}. Here we redo the proofs for 
completeness, and add the normalizer-of-Cartan case, via the different method 
used in the previous non-split Cartan situation (starting from semistable models and computing Galois quotients).  
The next results therefore necessarily closely parallel those of previous section, to which we refer for details which might 
be skipped here.

\subsection{Regular model for $\overline{{\cal M}} ({\cal P}, \Gamma_{\mathrm{s}} (p))$}
\label{S}
  
\begin{theo} 
\label{thS} 
Let $p>3$ be a prime, let $[\Gamma_{\mathrm{s}} (p)]$ be the moduli problem over $\Z [1/p]$ associated with 
$\Gamma_{\mathrm{s}} (p)$, and let ${\cal P}$ be a representable moduli problem which is finite \'etale over 
$({\mathrm{Ell}})_{/\Z_p}$. Let $\overline{{\cal M}} ({\cal P}, \Gamma_{\mathrm{s}} (p)) =\overline{{\cal M}} 
({\cal P}, \Gamma (p))/ \Gamma_{\mathrm{s}} (p)$ be the associated compactified fine moduli space. 

\medskip

   Then $\overline{{\cal M}} ({\cal P}, \Gamma_{\mathrm{s}} (p))$ has a regular model over $\Z_p^{\mathrm{ur}}$ 
whose special fiber is made up of three ``vertical'' Igusa parts, which are linked together, at each supersingular 
point, by Drinfeld components.

\medskip 
 
  All three vertical parts are copies of $\overline{{\cal M}} ({\cal P} )_{\overline{\F}_p} $. One of them (call 
it the {\rm central one}) has multiplicity $p-1$; the other two (call them {\rm external ones}) have multiplicity 
$1$.  

\medskip

  If $s_{\cal P}$ is the number of supersingular points of $\overline{{\cal M}} ({\cal P})
(\overline{\F}_p )$, the $s_{\cal P}$ horizontal chains of (Drinfeld) components are all copies of a rational curve 
$\PPP^1$ with multiplicity $p+1$. 

   At any intersection point in the special fiber between two irreducible components, of multiplicity $a$ and $b$, say (both in 
$\{ 1, p-1, p+1 \}$), the completed local ring is 
$$
\Z_p^{\mathrm{ur}} [[X,Y]]/(X^{a} \cdot Y^{b} -p ).
$$ 
\end{theo}

For a shape of the special fiber of the curve, see Figure~\ref{figureSgrossier} (which, precisely speaking, represents the 
coarse case $X_{\mathrm{s}}^+ (p)$).

\begin{proof}
As in the proof of Theorem~\ref{thNS}, we start by considering the semistable model over $\Z_p^{\mathrm{ur}} 
[\pi_0 =p^{2/(p^2 -1)} ]$ given in Theorem~4.1 of~\cite{EdixP20}. 

\medskip

   We begin with the vertical components. Remark that the two external Igusa parts in the special fiber, which are both 
copies of $\overline{{\cal M}} ({\cal P})_{\overline{\F}_p}$, are acted on trivially by Galois. They
therefore give rise to the same Igusa parts in the quotient modular curve over $\Z_p^{\mathrm{ur}}$. As for the 
other two Igusa parts, the situation is exactly the same as for those in the non-split Cartan case; therefore, the 
same arguments as in the beginning of the proof of Theorem~\ref{thNS} show they give rise to a copy of $\overline{{\cal M}} 
({\cal P} )_{\overline{\F}_p} $ with multiplicity $(p-1)$.

\medskip

  As for the Drinfeld parts, the situation is almost similar to the non-split Cartan case 
(Theorem~\ref{thNS}) but not quite. Recall that those Drinfeld components in the special fiber over 
$\Z_p^{\mathrm{ur}} [p^{2/(p^2 -1)}]$ have models 
\begin{eqnarray}
\label{DrinfeldXs}
U^2 =V^{p+1} +A_{\mathrm{s}},   
\end{eqnarray}
where $U=\alpha^{p} \beta -a/2$, $V=\alpha \beta$ and $A_{\mathrm{s}} =a^2 /4$ in $\overline{\F}_p$ some non-zero
element $a$ in $\overline{\F}_p$, with notation of~\cite{EdixP20}, Theorem~4.1, (27) and its proof. If $u$ is a 
generator of $\F_{p^2}^*$, its Galois action which can be read on the Drinfeld component as that of $(u,\left( {1 \atop 0} 
\ {0\atop u^{p+1}}\right) )\in \F_{p^2}^* \times \Gamma_{\mathrm{s}} (p)$, here induces $U\stackrel{\cdot  u}
{\longrightarrow} U$, $V\stackrel{\cdot u}{\longrightarrow} u^{p-1} V$ (cf. Proposition~\ref{stabilisateur}). Therefore, 
the quotient Drinfeld curves are projective lines, with parameters $X=U$, $Y=V^{p+1}$, linked by equations of the
shape $X^2 =Y +A_{\mathrm{s}}$. 
Looking at local rings at closed points of the arithmetic surface shows the new Drinfeld components appear with 
multiplicity $p+1$ in the special fiber above $\Z_p^{\mathrm{ur}}$. Also note that the quotient curve is regular, 
except perhaps under orbits of points where Galois does not act freely. The latter are, on the 
model~(\ref{DrinfeldXs}), the two points for which $V=0$ (but this implies $\alpha =0$ or $\beta =0$, so those 
points are quotients of points at infinity on the model for the full Drinfeld curves (\ref{laDrinfeldOriginale}) of 
$X(p)$ in Section~\ref{lesStructuraux}). Or perhaps, they are points at infinity for (\ref{DrinfeldXs}), so 
again they come from points at infinity for (\ref{laDrinfeldOriginale}). In all cases, those possibly non-regular 
points are precisely the four intersection points with Igusa parts.  
 
\medskip

   We therefore just check what happens at intersection points between Igusa and Drinfeld components. First we 
treat those at the intersection with the central Igusa parts. In our initial 
semistable model over $\Z_p^{\mathrm{ur}} [\pi_0 =p^{2/(p^2 -1)}]$, they are double points of thickness $1$ with 
complete local ring
\begin{eqnarray}
\label{thickness1}
\Z_p^{\mathrm{ur}} [\pi_0 ][[x,y ]]/(xy -\pi_0 ),
\end{eqnarray}
for $x$ and $y$ parameters corresponding to the Igusa and Drinfeld components, respectively (\cite{EdixP20}, 
Theorem~4.1). The stabilizer of each of those double point is the subgroup of 
index $2$ in the total Galois group, that is,
$$
S :={\mathrm{Gal}} (\Q_p^{\mathrm{ur}} (\pi_0 )/\Q_p^{\mathrm{ur}} (\sqrt{p})) \simeq \F_{p^2}^{*,2} /\{\pm 1\} .
$$
By the straightfowrard adaptation of the Lemma~2.3.4 from~\cite{CEdSt03} already used in our previous proofs, one can assume 
Galois acts on $x$ and $y$ via (Teichm\"uller lifts of) characters. Recall that $u$ denotes a generator of the full Galois
group $\F_{p^2}^* /\{\pm 1\}$, so that $u^2$ is a generator of $S$. The order of the action on $x$ and $y$ and 
the preservation of the equation in~(\ref{thickness1}) imply that Galois acts via 
$$
u^2\colon  x\mapsto u^{2(p+1)} x, \ y\mapsto u^{2(1-p)} y.
$$  
Making a d\'evissage if one wishes, by decomposing $S$ as $S[\frac{p+1}{2}] \times S[\frac{p-1}{2}]$, we see that 
on the two successive subquotients, Galois acts via pseudo-reflections on the cotangent space, which implies that the 
quotient ring is regular, from Serre's criterion (2.3.9 in~\cite{CEdSt03}) -- and indeed that ring is 
$$
\Z_p^{\mathrm{ur}} [[ x^{\frac{p-1}{2}} ,y^{\frac{p+1}{2}},\sqrt{p} ]] /(x^{\frac{p^2 -1}{4}} y^{\frac{p^2 -1}{4}} -\sqrt{p} ).
$$
Going down the last step to $\Z_p^{\mathrm{ur}}$ finally gives a singularity with equation
$$
\Z_p^{\mathrm{ur}} [[ X,Y ]] /(X^{p+1} Y^{p-1} -p) .
$$

Finally, let us check points of intersection with the external Igusa components. Any double point in the 
semistable model 
over $\Z_p^{\mathrm{ur}} [\pi_0 =p^{\frac{2}{p^2 -1}}]$ has completed local ring 
$$
\Z_p^{\mathrm{ur}} [\pi_0 ] [[x,y]]/(xy-\pi_0^{\frac{p-1}{2}} )
$$
and is fixed by the full Galois group (Theorem~4.1 of~\cite{EdixP20}). Taking the quotient by $(\F_{p^2}^* /\{ \pm 1\} 
)^{p-1}$ and setting $\pi_1 :=p^{1/(p+1)}$, we obtain singularities of the shape
$$
\tilde{A} :=\Z_p^{\mathrm{ur}} [\pi_1 ] [[ x,y ]] /(x y -\pi_1 )
$$
which now is a regular ring. If $\lambda =u^{p-1}$ is a generator of the $p+1$-order cyclic quotient $G_1 
:={\mathrm{Gal}} (\Q_p^{\mathrm{ur}} (\pi_1 )/\Q_p^{\mathrm{ur}}) \simeq \F_{p^2}^*
/\F_p^*$, then the Galois action is induced by $\lambda \colon x\mapsto x$, $y\mapsto \lambda y$ and $\pi_1 \mapsto 
\lambda \pi_1$, with
$$
\tilde{A}^{G_1} = \Z_p^{\mathrm{ur}} [[x,Y ]]/(x^{p+1} Y-p) . \hspace{3cm}   \Box
$$
\end{proof}

\subsection{$\overline{{\cal M}} ({\cal P}, \Gamma_{\mathrm{s}}^+ (p))$}

\begin{theo} 
\label{thS+} 
Let $p>3$ be a prime, let $[\Gamma_{\mathrm{s}}^+ (p)]$ be the moduli problem over $\Z [1/p]$ associated with 
$\Gamma_{\mathrm{s}}^+ (p)$, and let ${\cal P}$ be a representable moduli problem which is finite \'etale over 
$({\mathrm{Ell}})_{/\Z_p}$. Let $\overline{{\cal M}} ({\cal P}, \Gamma_{\mathrm{s}}^+ (p)) =\overline{{\cal M}} 
({\cal P}, \Gamma (p))/\Gamma_{\mathrm{s}}^+ (p)$ be the associated compactified fine moduli space. 

\medskip

   Then $\overline{{\cal M}} ({\cal P}, \Gamma_{\mathrm{s}}^+ (p))$ has a regular model over $\Z_p^{\mathrm{ur}}$ 
whose special fiber is made of two ``vertical'' Igusa parts, which are bound together, at each supersingular 
point, by Drinfeld components.

\medskip 
 
  The two vertical parts are copies of $\overline{{\cal M}} ({\cal P} )_{\overline{\F}_p} $. One of them has 
multiplicity $\frac{p-1}{2}$; the other one has multiplicity $1$.  

\medskip

  If $s_{\cal P}$ is the number of supersingular points of $\overline{{\cal M}} ({\cal P})(\overline{\F}_p )$, the 
$s_{\cal P}$ horizontal chains of (Drinfeld) components are all copies of a rational curve $\PPP^1$ with 
multiplicity $p+1$. Each of them also sees one projective line arising, with multiplicity $\frac{p+1}{2}$. 

\medskip

   At any intersection point in the special fiber between two irreducible components, of multiplicity say $a$ and $b (\in \{ 1, (p-1)/2, (p+1)/2, (p+1) \} )$, the completed local ring is 
$$
\Z_p^{\mathrm{ur}} [[X,Y]]/(X^{a} \cdot Y^{b} -p ).
$$ 
\end{theo}

For a picture of the curve's special fiber (actually in the coarse case), see Figure~\ref{figureS+grossier}.
 
\begin{proof}
Use Theorem~4.2 of~\cite{EdixP20}.  As in Theorem~\ref{thNS+} above, we see that the graph of the special fiber of 
the quotient scheme by Galois remains the same over $\Z_p^{\mathrm{ur}} [\sqrt{p}]$ as it is upstairs (on the semistable 
model over $\Z_p^{\mathrm{ur}} [p^{2/(p^2 -1)} ]$). Then on $\Z_p^{\mathrm{ur}}$, the two central Igusa components (... if 
there are two) are switched, whence the graph of the special fiber of our curve over $\Z_p^{\mathrm{ur}}$. Also remark 
that the copy of the Igusa component $\overline{\cal M}({\cal P} )$ in the special fiber of the semistable model 
over $\Z_p^{\mathrm{ur}} [p^{2/(p^2 -1)} ]$ descends to the same object over $\Z_p^{\mathrm{ur}}$.

   Then first assume $p\equiv -1$ mod $4$. As in the beginning of the proof of Theorem~\ref{thNS+} above, one sees that the 
Igusa component $\overline{\cal M}({\cal P} ,{\mathrm{Ig}} (p)/\{\pm 1\} )$ over $\Z_p^{\mathrm{ur}} [p^{2/(p^2 -1)} 
]$ descends to $\overline{\cal M}({\cal P} )$ over $\Z_p^{\mathrm{ur}}$, with multiplicity $(p-1)/2$. 
 If $p\equiv 1$ mod $4$, again the same arguments as in the proof of Theorem~\ref{thNS+} show the two Igusa 
components $\overline{\cal M}({\cal P} ,{\mathrm{Ig}} (p)/C_4 )$ over $\Z_p^{\mathrm{ur}} [p^{2/(p^2 -1)} ]$ 
descend to (a single) $\overline{\cal M} ({\cal P} )$ over $\Z_p^{\mathrm{ur}}$, always with multiplicity $(p-1)/2$.

\medskip

  As for the Drinfeld components over $\Z_p^{\mathrm{ur}} [p^{2/(p^2 -1)} ]$, Theorem~4.2 of~\cite{EdixP20} claims 
they have models 
\begin{eqnarray}
\label{DrinfeldXs+} 
Y^2 =X(X^{\frac{p+1}{2}} +A_{\mathrm{s}} )
\end{eqnarray}
(for some $A_{\mathrm{s}} =a^2 /4 \ne 0$) with $X=V^2 =(\alpha \beta )^2$ and $Y=UV=(\alpha^p \beta -a/2) (\alpha 
\beta )$ (notation as in  the proof of Theorem~4.2 of~\cite{EdixP20}, or~(\ref{DrinfeldXs}) above). As 
in~(\ref{DrinfeldXs}) and after we check that if $u$ is a generator of $\F_{p^2}^*$, Galois is induced by
$$
u\colon X\mapsto u^{2(p-1)} X, \ Y\mapsto u^{p- 1} Y .
$$  
So parameters for the Galois quotient are $R:=X^{(p+1)/2}$, $S:=Y^2 /X$, linked by an equation $S=R+A_{\mathrm{s}}$, 
and we are formally in the same situation as for the Drinfeld component~(\ref{DrinfeldNS+1}) in the proof of 
Theorem~\ref{thNS+} above. However, in the present split Cartan case,
the Drinfeld components are known to have one more intersection point with Igusa parts than in the non-split 
situation of~Theorem~\ref{thNS+}. More precisely,
if $p\equiv 1$ mod $4$, then (\ref{DrinfeldXs+}) defines a hyperelliptic equation of type ``$Y^2 =X^{2g+2} +
\cdots $'', which has two points at infinity, and they are 
intersection points with two Igusa parts on the semistable model over $\Z_p^{\mathrm{ur}} [p^{2/(p^2 -1)}]$. On 
the other hand, if $p\equiv -1$ mod $4$, then (\ref{DrinfeldXs+}) 
defines an equation of type ``$Y^2 =X^{2g+1} +\cdots $'' having one point at infinity, which is the intersection 
point of only one Igusa part with our Drinfeld component. So 
in both cases, there is one missing intersection point with Igusa parts. But one checks that $(X,Y)=(0,0)$ on 
(\ref{DrinfeldXs+}) indeed comes from a point at infinity on the Drinfeld component~(\ref{laDrinfeldOriginale}): so 
here is our missing intersection point with Igusa parts. Then there are $(p+1)/2$ remaining points 
(on~(\ref{DrinfeldXs+})) where Galois has some inertia, of order $2$, and the proof of Theorem~\ref{thNS+} above 
(cf.~(\ref{DrinfeldNS+1}) and the subsequent arguments) shows that taking the Galois quotient and blowing up still makes 
another projective line rise up with multiplicity $(p+1)/2$ in the Drinfeld projective line downstairs, over 
$\Z_p^{\mathrm{ur}}$. $\Box$

\end{proof}

\bigskip

\subsection{$X_{\mathrm{s}} (p)$}

\begin{theo} 
\label{thSgrossier} 
For $p\geq 13$ a prime, let $X_{\mathrm{s}} (p)$ be the coarse modular curve over $\Q$ associated with $\Gamma_{\mathrm{s}} 
(p)$.

   Then the {\em minimal regular model with normal crossings} of $X_{\mathrm{s}} (p)$ over $\Z_p^{\mathrm{ur}}$ has a
special fiber made up of three vertical Igusa parts, which are linked by horizontal chains of Drinfeld components above each 
supersingular point. All irreducible components are projective lines.

\medskip 
 
  The vertical parts are $j$-lines. The central one has multiplicity $p-1$. The two external ones have multiplicity 
$1$.   

\medskip

  If ${\cal S}$ is the number of supersingular $j$-invariants in $\F_{p^2}$, the ${\cal S}$ horizontal chains 
of (Drinfeld) components are almost all (i.e., for $j\not\equiv 0$ or $1728\mod p$) copies of a 
rational curve $\PPP^1$ with multiplicity $p+1$.  
%
According to the residue class of $p$ {\rm mod} $12$, the vertical Igusa components moreover have the following 
extra components: 
\begin{itemize}
\item {\bf If $p\equiv 1$ mod $12$}, the central Igusa component intersects transversally an extra component at 
each of the two (ordinary) $j$-invariants $1728$ and $0$, which are a projective line with multiplicity $(p-1)/2$ 
and a projective line with multiplicity $(p-1)/3$, respectively. 
\item  {\bf If $p\equiv 11$ mod $12$}, the Drinfeld components above the supersingular $j\equiv 1728$ and $j\equiv 
0\ {\mathrm{mod}}\ p$ are both projective lines, with multiplicity $(p+1)/2$ and $(p+1)/3$, respectively. 
\item  {\bf If $p\equiv 5$ mod $12$}, we have the relevant mix between the two above situations: $j\equiv 1728$ is 
ordinary ${\mathrm{mod}}\ p$, so see the case $p\equiv 1\ {\mathrm{mod}}\ 12$, whereas $j\equiv 0$ is supersingular 
${\mathrm{mod}}\ p$, so see the case $p\equiv 11\ {\mathrm{mod}}\ 12$.
\item  {\bf If $p\equiv 7\ {\mathrm{mod}}\ 12$}, symmetric to the previous case: now $j\equiv 1728$ is supersingular
${\mathrm{mod}}\ p$, whereas $j\equiv 0$ is ordinary ${\mathrm{mod}}\ p$.
\end{itemize}
\medskip

   Any intersection point in the special fiber between two irreducible components, of multiplicity say $a$ and $b$ (both in 
$\{ 1, (p-1)/3, (p+1)/3, (p-1)/2, (p+1)/2, (p-1), (p+1) \} )$, has completed local ring
$$
\Z_p^{\mathrm{ur}} [[X,Y]]/(X^{a} \cdot Y^{b} -p ).
$$ 
See Figure~\ref{figureSgrossier}.
%
%
\end{theo}

\begin{rema}
\label{petitsminimauxsplit}
{\rm{Again Theorem~\ref{thSgrossier} also applies to give regular models for all $p\geq 5$, but we have excluded primes 
$p\leq 11$ 
from its statement because of the non-minimality of our regular model with normal crossings for those primes. But 
$X_{\mathrm{s}} (p)$ is just a projective line if $p\leq 5$ (including $p=2, 3$); and for $p=7$ and $p=11$, the minimal 
regular model with normal crossings is obtained by contracting the central Igusa component (with multiplicity $p-1$). 
(See Remarks~\ref{petitsminimauxNsplit} and~\ref{petitsminimauxNsplit+}).}} 
\end{rema}

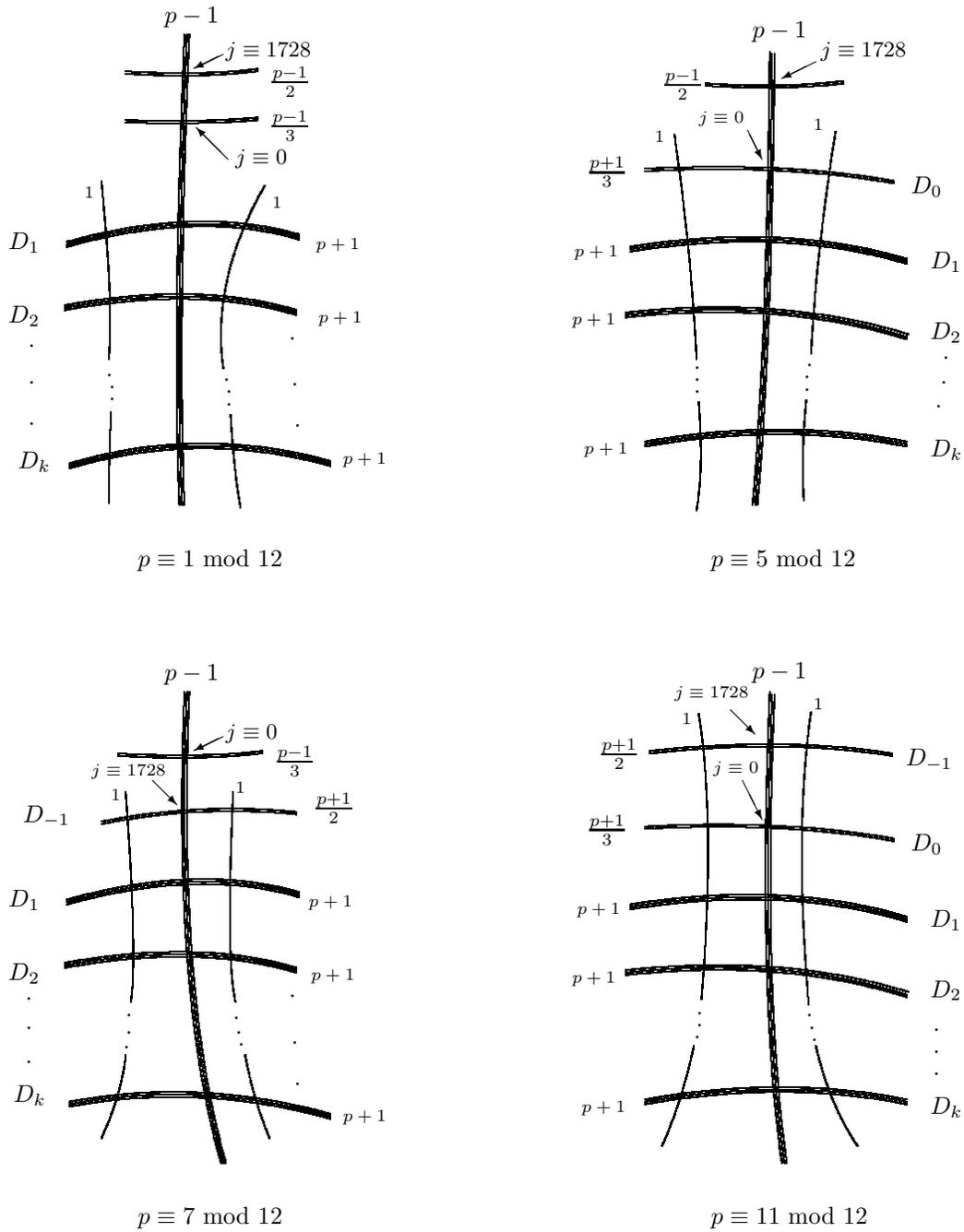
\begin{figure}

\begin{center}


\begin{picture}(320,410)(-30,-90)

\qbezier(-50,280)(5,290)(45,278)
\qbezier(-50,281)(5,291)(45,279)
\qbezier(-50,282)(5,292)(45,280)

\put(-73,275){$D_2$}
\put(54,275){\scriptsize$p+1$}

\qbezier(-35,333)(-30,283)(-32,260)
\put(-42,326){\scriptsize$1$}
\put(-33,251){$\cdot$}
\put(-32,246){$\cdot$}
\put(-32,241){$\cdot$}
\qbezier(-31,238)(-32,218)(-32,201)

\qbezier(32,331)(10,293)(15,256)
\put(35,322){\scriptsize$1$}
\put(15,249){$\cdot$}
\put(16,244){$\cdot$}
\put(16,239){$\cdot$}
\qbezier(18,237)(19,217)(22,199)

\qbezier(-49,306)(5,321)(46,309)
\qbezier(-49,307)(5,322)(46,310)
\qbezier(-49,308)(5,323)(46,311)

\put(-73,305){$D_1$}
\put(53,305){\scriptsize$p+1$}

\qbezier(-48,217)(5,233)(59,218)
\qbezier(-48,216)(5,232)(59,217)
\qbezier(-48,215)(5,231)(59,216)

\put(-69,215){$D_k$}
\put(64,217){\scriptsize$p+1$}


\qbezier(0,393)(-5,270)(-2,200) 
\qbezier(1,393)(-4,270)(-1,200)
\qbezier(-1,393)(-6,270)(-3,200)

\put(-9,398){$p-1$}


\put(16,384){\small $j\equiv 1728$}
\put(15,385){\vector(-2,-1){12}}
\qbezier(-25,378)(5,376)(29,379)
\qbezier(-25,377)(5,375)(29,378)
\put(34,371){$\frac{p-1}{2}$}


\put(20,340){\small $j\equiv 0$}
\put(17,342){\vector(-1,1){13}}
\qbezier(-25,358)(5,357)(29,359)
\qbezier(-25,357)(5,356)(29,358)
\put(34,353){$\frac{p-1}{3}$}


\put(-65,+263){$\cdot$}
\put(-65,248){$\cdot$}
\put(-65,231){$\cdot$}

\put(42,266){$\cdot$}
\put(43,247){$\cdot$}
\put(44,230){$\cdot$}


\put(-20,175){$p\equiv 1$ mod $12$}



 

\put(260,384){\small $j\equiv 1728$}
\put(258,382){\vector(-2,-1){14}}
\qbezier(213,373)(245,371)(269,374)
\qbezier(213,372)(245,370)(269,373)
\put(195,370){$\frac{p-1}{2}$}


\qbezier(188,338)(245,341)(290,333)
\qbezier(188,337)(245,340)(290,332)
\put(297,328){$D_0$}
\put(164,334){$\frac{p+1}{3}$}

\qbezier(200,352)(207,303)(209,260)
\put(192,349){\scriptsize$1$}
\put(208,254){$\cdot$}
\put(208,249){$\cdot$}
\put(208,244){$\cdot$}
\qbezier(210,242)(212,213)(209,198)

\qbezier(266,353)(258,302)(256,262)
\put(257,354){\scriptsize$1$}
\put(254,255){$\cdot$}
\put(253,250){$\cdot$}
\put(253,245){$\cdot$}
\qbezier(254,243)(252,222)(253,202)

\put(210,357){\scriptsize{$j\equiv 0$}}
\put(230,354){\vector(1,-2){6}}


\qbezier(182,306)(245,316)(295,301)
\qbezier(182,305)(245,315)(295,300)
\qbezier(182,304)(245,314)(295,299)
\put(305,297){$D_1$}
\put(159,302){\scriptsize$p+1$}

\qbezier(180,279)(245,286)(296,270)
\qbezier(180,278)(245,285)(295,269)
\qbezier(180,277)(245,284)(295,268)
\put(305,268){$D_2$}
\put(158,274){\scriptsize$p+1$}

\qbezier(188,226)(245,236)(295,226)
\qbezier(188,225)(245,235)(295,225)
\qbezier(188,224)(245,234)(295,224)
\put(305,220){$D_k$}
\put(163,221){\scriptsize$p+1$}

\put(310,258){$\cdot$}
\put(309,248){$\cdot$}
\put(307,238){$\cdot$}


\put(232,392){$p-1$}

\qbezier(241,385)(240,270)(234,200)
\qbezier(240,385)(239,270)(233,200)
\qbezier(239,386)(238,270)(232,200)

%


\put(215,175){$p\equiv 5$ mod $12$}






\qbezier(-50,10)(5,20)(45,8)
\qbezier(-50,11)(5,21)(45,9)
\qbezier(-50,12)(5,22)(45,10)

\put(-73,5){$D_2$}
\put(51,5){\scriptsize$p+1$}

\qbezier(-49,36)(5,51)(46,39)
\qbezier(-49,37)(5,52)(46,40)
\qbezier(-49,38)(5,53)(46,41)

\put(-73,35){$D_1$}
\put(50,35){\scriptsize$p+1$}

\qbezier(-48,-47)(5,-37)(59,-52)
\qbezier(-48,-46)(5,-36)(59,-51)
\qbezier(-48,-45)(5,-35)(59,-50)

\put(-71,-45){$D_k$}
\put(64,-53){\scriptsize$p+1$}

\qbezier(-25,-26)(-27,-43)(-35,-60)
\qbezier(24,-26)(27,-43)(34,-60)

\put(-25,-12){$\cdot$}
\put(-25,-19){$\cdot$}
\put(-26,-25){$\cdot$}

\put(19,-13){$\cdot$}
\put(20,-19){$\cdot$}
\put(21,-24){$\cdot$}


\qbezier(0,123)(-5,0)(15,-70) 
\qbezier(1,123)(-4,0)(16,-70)
\qbezier(-1,123)(-6,0)(14,-70)

\put(-9,128){$p-1$}


\put(16,104){\small $j\equiv 0$}
\put(15,105){\vector(-2,-1){12}}
\qbezier(-28,98)(5,96)(31,99)
\qbezier(-28,97)(5,95)(31,98)
\put(36,92){$\frac{p-1}{3}$}


\qbezier(-35,70)(5,78)(45,74)
\qbezier(-35,69)(5,77)(45,73)
\put(52,73){$\frac{p+1}{2}$}

\put(-67,70){$D_{-1}$}

\qbezier(-25,82)(-20,12)(-23,-4)
\put(-31,79){\scriptsize$1$}

\qbezier(19,82)(16,12)(20,-4)
\put(20,82){\scriptsize$1$}

\put(-39,89){\scriptsize{$j\equiv 1728$}}
\put(-14,87){\vector(1,-1){10}}


\put(-66,-5){$\cdot$}
\put(-66,-17){$\cdot$}
\put(-66,-31){$\cdot$}

\put(42,-4){$\cdot$}
\put(43,-23){$\cdot$}
\put(44,-40){$\cdot$}


\put(-20,-95){$p\equiv 7$ mod $12$}




\qbezier(190,99)(245,105)(289,98)
\qbezier(190,98)(245,104)(289,97)
\put(295,93){$D_{-1}$}
\put(169,95){$\frac{p+1}{2}$}

\qbezier(210,114)(217,68)(211,-3)
\put(203,110){\scriptsize$1$}

\qbezier(256,115)(249,67)(255,-5)
\put(257,116){\scriptsize$1$}

\put(200,121){\scriptsize{$j\equiv 1728$}}
\put(223,116){\vector(1,-1){10}}

\put(216,90){\scriptsize{$j\equiv 0$}}
\put(228,87){\vector(1,-2){8}}


\qbezier(195,-63)(203,-45)(208,-22)
\put(209,-11){$\cdot$}
\put(209,-16){$\cdot$}
\put(208,-21){$\cdot$}

\qbezier(275,-63)(263,-45)(258,-22)
\put(254,-12){$\cdot$}
\put(255,-17){$\cdot$}
\put(256,-22){$\cdot$}


\qbezier(188,68)(245,71)(290,63)
\qbezier(188,67)(245,70)(290,62)
\put(297,58){$D_0$}
\put(164,64){$\frac{p+1}{3}$}


\qbezier(182,36)(245,46)(295,31)
\qbezier(182,35)(245,45)(295,30)
\qbezier(182,34)(245,44)(295,29)
\put(305,27){$D_1$}
\put(160,32){\scriptsize$p+1$}

\qbezier(180,9)(245,16)(296,0)
\qbezier(180,8)(245,15)(295,-1)
\qbezier(180,7)(245,14)(295,-2)
\put(305,-2){$D_2$}
\put(158,4){\scriptsize$p+1$}

\qbezier(188,-44)(245,-34)(295,-44)
\qbezier(188,-45)(245,-35)(295,-45)
\qbezier(188,-46)(245,-36)(295,-46)
\put(305,-50){$D_k$}
\put(162,-49){\scriptsize$p+1$}

\put(306,-18){$\cdot$}
\put(306,-27){$\cdot$}
\put(306,-36){$\cdot$}



\put(232,129){$p-1$}

\qbezier(240,122)(235,0)(245,-70)
\qbezier(241,122)(236,0)(246,-70)
\qbezier(239,123)(234,0)(244,-70)

%


\put(215,-95){$p\equiv 11$ mod $12$}


\end{picture}
\end{center}
\caption{Special fiber over $\overline{\F}_p$ of the minimal regular model with normal crossings of $X_{\mathrm{s}} (p)$}  
\label{figureSgrossier}
\end{figure}

\begin{proof}
We now start from Theorem~4.3 of~\cite{EdixP20}. The only things to adapt with respect to the fine, rigidified case 
(Theorem~\ref{thS} above) is the situation at special points with extra automorphisms. Then the computations are 
essentially the same as those carried out in the proof of Theorem~\ref{thNSgrossier} above. (Compare 
Figure~\ref{figureSgrossier} with Figure~\ref{figureNSgrossier}.) $\Box$

\end{proof}

\begin{coro}
\label{minimalXs}
 For any prime $p$, the {\em minimal regular model} of $X_{\mathrm{s}} (p)$ over $\Z_p^{\mathrm{ur}}$ has a reduced fiber,  
which is made of two components (the external Igusa ones), intersecting (many times) at one common highly singular point 
(unless the curve has genus $0$, that is, $p= 2, 3$ or $5$). 
\end{coro}

See Figure~\ref{figureSminimal5} for the case $p=12k+5$.

\begin{proof}
See Corollary~\ref{minimalXns} and its proof: by Castelnuovo's criterion, one just needs to contract the components which are 
isomorphic to projective lines and have self-intersection $-1$. Namely, in the case $p=12k+5$, for instance, we first 
contract the Drinfeld components with multiplicity $p+1$. Then the central Igusa component (with multiplicity $p-1$) has 
self-intersection $-1$, so we contract it. Next we can contract the multiplicity $(p-1)/2$ component; finally we contract 
the Drinfeld component with multiplicity $(p+1)/3$. What remains is the two external Igusa components, which have 
multiplicity $1$. $\Box$

\end{proof}

\begin{rema}
\label{dejavureduit}
{\rm{Note that, after contracting all components $D_i$ for $i>0$ in the minimal regular model with normal crossings of 
$X_{\mathrm{s}} (p)$ over $\Z_p^{\mathrm{ur}}$ (that is, at the first step of the minimal resolution process), one recovers, 
by a quite different path, the results of \cite{Edix89}, paragraph~1.5, because of the well-known $\Q$-isomorphism 
$X_0 (p^2 )\simeq X_{\mathrm{s}} (p)$.

   As a numerical sanity check, doing the math as above for $p=7$ for example shows that $X_{\mathrm{s}} (7) 
\simeq X_0 (49)$ and ${\mathrm{Jac}} (X_{\mathrm{ns}} (7))$ (which have same special fiber) are elliptic curves of Kodaira 
type III, a fact known e.g. from~\cite{BSW75}.}} 
\end{rema}

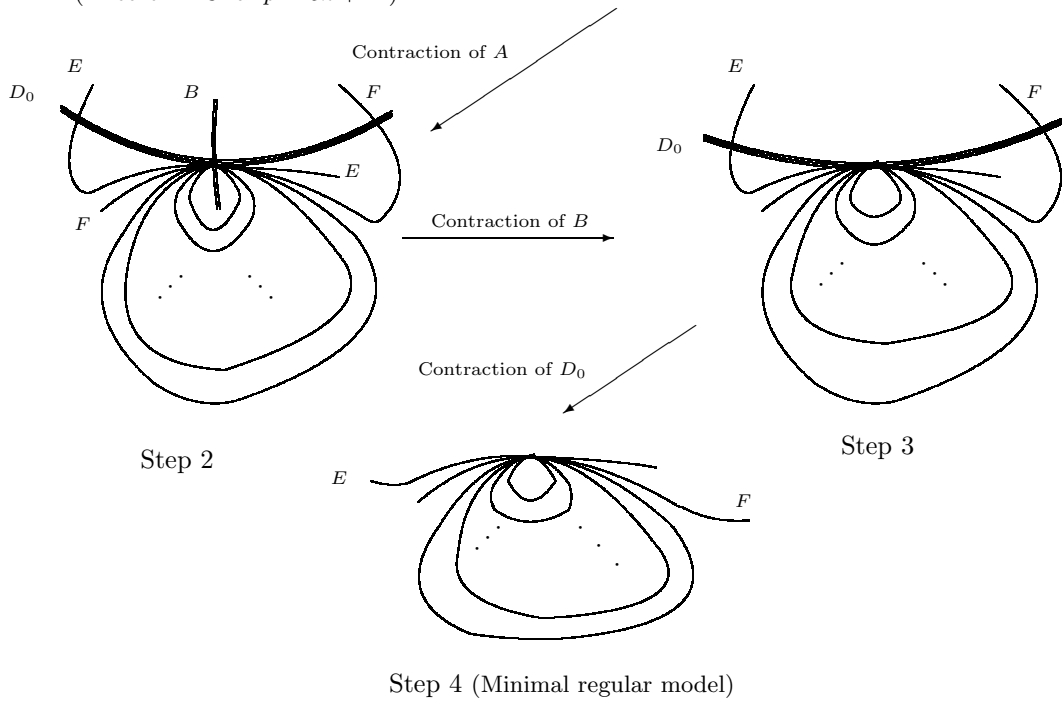
\begin{figure}
\begin{center}

\begin{picture}(320,410)(-30,-90)

\qbezier(-50,280)(5,290)(45,278)
\qbezier(-50,281)(5,291)(45,279)
\qbezier(-50,282)(5,292)(45,280)

\put(-67,275){\scriptsize$D_2$}
\put(54,275){\scriptsize$p+1$}


\put(-21,357){\scriptsize$E$}
\qbezier(-25,362)(-15,313)(-20,260)
\put(-32,360){\scriptsize$1$}
\put(-21,254){$\cdot$}
\put(-21,249){$\cdot$}
\put(-21,244){$\cdot$}
\qbezier(-20,243)(-22,222)(-26,202)

\put(23.5,357){\scriptsize$F$}
\qbezier(35,363)(15,302)(18,262)
\put(38,360){\scriptsize$1$}
\put(17,255){$\cdot$}
\put(18,250){$\cdot$}
\put(18,245){$\cdot$}
\qbezier(20,242)(22,213)(29,198)


\qbezier(-51,306)(5,321)(46,309)
\qbezier(-51,307)(5,322)(46,310)
\qbezier(-51,308)(5,323)(46,311)

\put(-67,310){\scriptsize$D_1$}
\put(53,305){\scriptsize$p+1$}

\qbezier(-48,217)(5,233)(59,218)
\qbezier(-48,216)(5,232)(59,217)
\qbezier(-48,215)(5,231)(59,216)

\put(-66,215){\scriptsize$D_k$}
\put(64,217){\scriptsize$p+1$}


\qbezier(0,393)(-5,270)(-2,200) 
\qbezier(1,393)(-4,270)(-1,200)
\qbezier(-1,393)(-6,270)(-3,200)

\put(-6,398){\scriptsize$p-1$}
\put(-11,390){\scriptsize$A$}


\put(-34,375){\scriptsize$B$}
\qbezier(-25,378)(5,376)(29,379)
\qbezier(-25,377)(5,375)(29,378)
\put(34,374){\scriptsize$\frac{p-1}{2}$}


\qbezier(-45,351)(5,350)(49,352)
\qbezier(-45,350)(5,349)(49,351)
\put(57,346){$\frac{p+1}{3}$}
\put(-59,348){\scriptsize$D_0$}


\put(-57,+265){$\cdot$}
\put(-57,248){$\cdot$}
\put(-57,229){$\cdot$}

\put(42,266){$\cdot$}
\put(43,247){$\cdot$}
\put(44,230){$\cdot$}



\put(-75,175){\small{Minimal regular model with normal crossings}}
\put(-45,162){\small{(Theorem~\ref{thSgrossier} for $p=5k+12$)}}


 
\put(83,301){\vector(1,0){76}}

\put(85,290){\scriptsize{Contraction of the $D_i$}}
\put(102,280){\scriptsize{for $i>0$}}


\qbezier(213,373)(245,371)(269,374)
\qbezier(213,372)(245,370)(269,373)
\put(195,370){\scriptsize$B$}


\qbezier(188,338)(245,341)(290,333)
\qbezier(188,337)(245,340)(290,332)
\put(174,334){\scriptsize$D_0$}

\put(211,346){\scriptsize$E$}

\put(268,345){\scriptsize$F$}



\put(232,392){\scriptsize$A$}

\qbezier(241,385)(243,270)(240,200)
\qbezier(240,385)(242,270)(239,200)
\qbezier(239,385)(241,270)(238,200)

%


\qbezier(200,345)(275,320)(240,300) 
\qbezier(240,300)(225,288)(225,280)

\qbezier(280,345)(205,320)(240,300)
\qbezier(240,300)(255,288)(255,280)

\qbezier(240,240)(230,245)(225,253)
\qbezier(200,195)(275,220)(240,240)

\qbezier(240,240)(250,245)(255,253)
\qbezier(280,195)(205,220)(240,240)

\put(255,259){$\cdot$}
\put(255,264){$\cdot$}
\put(255,269){$\cdot$}

\put(222,259){$\cdot$}
\put(222,264){$\cdot$}
\put(222,269){$\cdot$}



\put(225,175){Step 1}




 
\put(160,159){\vector(-3,-2){70}}

\put(60,140){\scriptsize{Contraction of $A$}}



\qbezier(-50,120)(10,80)(75,118)
\qbezier(-50,121)(10,81)(75,119)
\qbezier(-50,122)(10,82)(75,120)

\put(-70,125){\scriptsize$D_0$}

  

\qbezier(-37,90)(10,113)(64,80)

\qbezier(-35,82)(10,120)(61,73)

\qbezier(61,73)(90,35)(15,10)
\qbezier(15,10)(0,7)(-15,17)
\qbezier(-15,17)(-47,45)(-27,73)

\qbezier(10,99)(39,98)(55,95)


\qbezier(-27,73)(07,131)(57,63)

\qbezier(57,63)(71,41)(12,22)
\qbezier(12,22)(-34,24)(-24,65)

\qbezier(-24,65)(-17,94)(9,101)


\qbezier(-6,89)(07,109)(22,89)

\qbezier(-6,79)(07,55)(22,79)

\qbezier(-6,89)(-8,84)(-6,79)
\qbezier(22,79)(24,84)(22,89)

\qbezier(-1,90)(07,108)(17,90)

\qbezier(-1,85)(07,65)(17,85)

\qbezier(-1,85)(-2,87)(-1,90)
\qbezier(17,85)(18,87)(17,90)


\put(-6,55){$\cdot$}
\put(-10,51){$\cdot$} 
\put(-14,47){$\cdot$}  


\put(20,55){$\cdot$}
\put(24,51){$\cdot$} 
\put(28,47){$\cdot$}  


\put(-48,135){\scriptsize$E$}
\qbezier(-38,130)(-53,100)(-43,90)
\put(57,95){\scriptsize$E$}
\qbezier(-43,90)(-40,88)(-37,90)

\put(65,125){\scriptsize$F$}
\qbezier(55,130)(90,100)(72,80)
\put(-45,75){\scriptsize$F$}
\qbezier(72,80)(69,76)(64,80)


\put(-4,125){\scriptsize$B$}
\qbezier(8,124)(6,107)(09,83)
\qbezier(9,124)(7,107)(10,83)



\put(-20,-15){Step 2}


\put(79,72){\vector(1,0){80}}
\put(90,76){\scriptsize{Contraction of $B$}}


\qbezier(193,109)(263,85)(328,115)
\qbezier(193,110)(263,86)(328,116)
\qbezier(193,111)(263,87)(328,117)

\put(175,105){\scriptsize$D_0$}

\qbezier(213,90)(260,113)(314,80)

\qbezier(215,82)(260,120)(311,73)
 
\qbezier(311,73)(340,35)(265,10)
\qbezier(265,10)(250,7)(235,17)
\qbezier(235,17)(203,45)(223,73)

\qbezier(260,99)(289,98)(305,95)


\qbezier(223,73)(257,131)(307,63)
 
\qbezier(307,63)(321,41)(262,32)
\qbezier(262,32)(226,34)(226,55)

\qbezier(226,55)(233,94)(259,101)


\qbezier(244,89)(257,109)(272,89)

\qbezier(244,79)(257,60)(272,79)

\qbezier(244,89)(242,84)(244,79)
\qbezier(272,79)(274,84)(272,89)

\qbezier(249,90)(257,108)(267,90)

\qbezier(249,85)(257,75)(267,85)

\qbezier(249,85)(248,87)(249,90)
\qbezier(267,85)(268,87)(267,90)


\put(244,60){$\cdot$}
\put(240,56){$\cdot$} 
\put(236,52){$\cdot$}  


\put(275,60){$\cdot$}
\put(279,56){$\cdot$} 
\put(283,52){$\cdot$}  


\put(202,135){\scriptsize$E$}
\qbezier(212,130)(197,100)(207,90)

\qbezier(207,90)(210,88)(213,90)

\put(315,125){\scriptsize$F$}
\qbezier(305,130)(340,100)(322,80)
 
\qbezier(322,80)(319,76)(314,80)
 

\put(245,-10){Step 3}


 
\put(190,39){\vector(-3,-2){50}}

\put(85,20){\scriptsize{Contraction of $D_0$}}


\qbezier(83,-20)(130,3)(184,-30)

\qbezier(85,-28)(130,10)(181,-37)
 
\qbezier(181,-37)(210,-75)(135,-80)
\qbezier(135,-80)(120,-80)(105,-78)
\qbezier(105,-78)(73,-65)(93,-37)

\qbezier(130,-11)(159,-12)(175,-15)


\qbezier(93,-37)(128,19)(177,-47)
 
\qbezier(177,-47)(191,-69)(132,-72)
\qbezier(132,-72)(96,-66)(100,-48)

\qbezier(100,-48)(103,-16)(129,-10)


\qbezier(114,-21)(127,-1)(142,-21)

\qbezier(114,-31)(127,-40)(142,-31)

\qbezier(114,-31)(111,-26)(114,-21)

\qbezier(114,-31)(112,-26)(114,-31)
\qbezier(142,-31)(144,-26)(142,-21)

\qbezier(119,-20)(127,-2)(137,-20)

\qbezier(119,-20)(127,-35)(137,-20)


\put(114,-40){$\cdot$}
\put(110,-44){$\cdot$} 
\put(106,-48){$\cdot$}  


\put(145,-40){$\cdot$}
\put(152,-47){$\cdot$} 
\put(159,-54){$\cdot$}


\put(52,-21){\scriptsize$E$}

\qbezier(67,-20)(77,-23)(83,-20)

\put(205,-30){\scriptsize$F$}
\qbezier(185,-30)(196,-36)(210,-35)

\put(74,-100){Step 4 \small{(Minimal regular model)}}

 
\end{picture}
\end{center}
\caption{Minimal resolution of $X_{\mathrm{s}} (p)$ in the case $p=12k+5$} 
\label{figureSminimal5}
\end{figure}

\subsubsection{Component groups for ${J}_{\mathrm{s}} (p)$}

\begin{propo}
For $p\geq 5$ a prime, denote by ${\mathcal J}_{\mathrm{s}} (p)$ the N\'eron model over $\Z_p^{\mathrm{ur}}$ of the Jacobian of 
$X_{\mathrm{s}} (p)$. The component group ${\mathcal J}_{\mathrm{s}} (p)/{\mathcal J}^{0}_{\mathrm{s}} (p)$ of the special fiber 
of ${\mathcal J}_{\mathrm{s}} (p)$ is 
$$
\Z /\frac{p^2 -1}{24}\Z .$$
\end{propo}

(Here we recover Proposition~2 of~\cite{EdEis}.)

\begin{proof}
We do the same computations as for Proposition~\ref{componentsXns} -- but perhaps easier (here, at step $1$ of the minimal 
resolution process, that is, after contracting all Drinfeld components with multiplicity $p+1$, the type of the intersection 
matrices is fixed (always $5\times 5$)). For instance, for $p=12k+5$, one checks that the intersection matrix of the regular 
model at step $1$ is, in the basis $\{ E, F, A, B, D_0 \}$ of Figure~\ref{figureSminimal5},
$$
\left(
\begin{array}{ccccc}
(-12k^2 -9k -2) & k & k & 0 & 1\\
k & (-12k^2 -9k -2) & k & 0 & 1\\
k & k & -1 & 1  & 1\\
0 & 0 & 1 & -2 & 0\\
1 & 1 & 1 & 0 & -3
\end{array}
\right)
$$
which is easily put under Smith normal form. $\Box$
\end{proof}

\subsection{$X_{\mathrm{s}}^+ (p)$}

\begin{theo} 
\label{thS+grossier} 
For $p>11$ a prime, let $X_{\mathrm{s}}^+ (p)$ be the coarse modular curve over $\Q$ associated with 
$\Gamma_{\mathrm{s}}^+ (p)$.

\medskip

   The {\em minimal regular model with normal crossings} of $X_{\mathrm{s}}^+ (p)$ over $\Z_p^{\mathrm{ur}}$ has a
special fiber made of two vertical 
Igusa parts, which are linked by horizontal chains of Drinfeld components above each supersingular point. All 
irreducible components are (smooth) projective lines.

\medskip 
 
  The vertical parts are $j$-lines. One has multiplicity $(p-1)/2$  (in what follows, we call it the ``central 
one''). The other one has multiplicity $1$.  

\medskip

  If ${\cal S}$ is the number of supersingular $j$-invariants in $\F_{p^2}$, the ${\cal S}$ horizontal chains of 
(Drinfeld) components are almost all copies of branches, at least for $j\not\equiv 0$ or $1728\mod p$, which are 
rational curves $\PPP^1$, with multiplicity $p+1$, from which grows
another exceptional projective line with multiplicity $(p+1)/2$.   
According to the residue class of $p\mod 12$, the central Igusa component moreover has the following supplementary features: 
\begin{itemize}
\item {\bf If $p\equiv 1$ mod $12$}, then the central Igusa component intersects transversally an extra component above the 
(ordinary) invariant $j\equiv 0\mod p$, which is made of a projective line, with multiplicity $(p-1)/3$, followed by another 
$\PPP^1$, with multiplicity $(p-1)/6$.
\item  {\bf If $p\equiv 5$ mod $12$}, the central Igusa component has no extra component. However, there is one exceptional 
Drinfeld component, at the (supersingular) invariant $j\equiv 0$, which is a projective line, with multiplicity $(p+1)/3$, 
intersecting itself one extra projective line, with multiplicity $(p+1)/6$.
\item  {\bf If $p\equiv 7$ mod $12$}, the central Igusa component intersects transversally an extra component above the 
ordinary invariant $j\equiv 0\mod p$, which is as in the case $p\equiv 1$ {\rm mod} $12$: a projective line, with 
multiplicity $(p-1)/3$, followed by another $\PPP^1$, with multiplicity $(p-1)/6$. 
At the supersingular point $j\equiv 1728$, the central Igusa component directly intersects (transversally) the other Igusa 
line (which has multiplicity $1$).
\item  {\bf If $p\equiv 11$ mod $12$}, there are two exceptional supersingular $j$-invariants ($j\equiv 0$ and $j\equiv 
1728$), where the situation is as described in the two respective cases above.   
\end{itemize}
\medskip

   Any intersection point in the special fiber between two irreducible components, of multiplicity say $a$ and $b$ (both in 
$\{ 1, (p-1)/6, (p+1)/6, (p-1)/3, (p+1)/3, (p-1)/2, (p+1)/2, (p+1) \} )$, has completed local ring
$$
\Z_p^{\mathrm{ur}} [[X,Y]]/(X^{a} \cdot Y^{b} -p ).
$$ 
See Figure~\ref{figureS+grossier}. 
\end{theo}

\begin{rema}
\label{petitsminimauxsplit+}
{\rm{Once more, Theorem~\ref{thS+grossier} also gives regular models for all $p\geq 5$, but we 
have excluded primes $p\leq 11$ from its statement because of the non-minimality of our regular model with normal crossings 
for those small primes (cf. Remark~\ref{petitsminimauxNsplit+} etc.). However, and as in loc. cit. again, the 
$X_{\mathrm{s}}^+ (p)$ for $p\leq 7$ are just projective lines. For $X_{\mathrm{s}}^+ (11)$, one sees that its
minimal regular model with normal crossings is found from the regular one given in Theorem~\ref{thS+grossier} by 
contracting the central Igusa component in the special fiber. 

    In fact, note that our models for $X_{\mathrm{ns}}^+ (p)$ and $X_{\mathrm{s}}^+ (p)$ have {\em the same} special fiber 
for $p\le 7$ and $p= 13$. When $p=13$, that illustrates Baran's isomorphism (\cite{Baran}). (The same isomorphism of special fibers holds for $X_{\mathrm{ns}} (p)$ and $X_{\mathrm{s}} (p)$, in the same range $p\leq 7$, $p=13$.)}}
\end{rema}

\begin{rema}
\label{etmaintenantlegenredeSplit+}
{\rm{Extending the comparison between the split and nonsplit cases, one can also use Theorem~\ref{thS+grossier} to check the 
genera of the 
$X_{\mathrm{s}}^+ (p)$ (cf. Remark~\ref{etmaintenantlegenredeNsplit+}). Again Theorem~\ref{thS+grossier} provides us with
SNC-models, in the terminology of~\cite{Dino90}, 
with genus-$0$ irreducible components in their special fiber. So writing $G$ for the dual graph of the special fiber and
$\beta$ the Betti number of $G$ (which here is nonzero), we use the formula $g(X_{\mathrm{s}}^+ (p)) =\beta +\frac{1}{2} 
\sum_{v} (r_v -1)(d_v -2)$ (where the above sum runs through the vertices $v$ of $G$, for which $d_v$ denotes the degree of 
$v$ and $r_v$ is the multiplicity of the irreducible component corresponding to it (\cite{Dino90}, top of p.~150)). Computing 
genera in each of the four congruence cases mod $12$ of Theorem~\ref{thS+grossier} gives back the formula:
$$
g(X_{\mathrm{s}}^+ (p))=\frac{p^2 -8\, p +11 -4\left( {-3 \atop p} \right)}{24}
$$
(\cite{Ma76}, p. 117). Actually, that strong parallelism between the split and nonsplit case (compare 
Figure~\ref{figureNSgrossier+} and Figure~\ref{figureS+grossier}) is reflected in the fact that our Lorenzini's formula 
directly gives
$$
g(X_{\mathrm{s}}^+ (p))= g(X_{\mathrm{ns}}^+ (p)) +\beta ,
$$  
a relation predicted by Chen-Edixhoven's isogeny (\cite{dSE}): 
$$
{\mathrm{Jac}} (X_{\mathrm{s}}^+ (p)) \sim {\mathrm{Jac}} (X_{\mathrm{ns}}^+ (p)) \times {\mathrm{Jac}} (X_0 (p) ),
$$
as ${\mathrm{Jac}} (X_0 (p))$ has purely toric reduction of same rank as that of ${\mathrm{Jac}} (X_{\mathrm{s}}^+ (p)) $. 
Cf. Figure~\ref{XS+vsXNS+}.}} 
\end{rema}

\begin{figure}
\begin{center}

\begin{picture}(320,410)(-30,-90)

\qbezier(-50,280)(5,290)(45,278)
\qbezier(-50,281)(5,291)(45,279)
\qbezier(-50,282)(5,292)(45,280)

\put(-73,275){\scriptsize$D_2$}
\put(52,275){\scriptsize$p+1$}

\qbezier(-32,293)(-29,283)(-31,270)
\qbezier(-31,293)(-28,283)(-30,270)
\put(-44,296){$\frac{p+1}{2}$}

\qbezier(20,327)(13,263)(23,202)
\put(25,326){\scriptsize$1$}

\qbezier(-49,306)(5,321)(46,309)
\qbezier(-49,307)(5,322)(46,310)
\qbezier(-49,308)(5,323)(46,311)

\put(-73,305){\scriptsize$D_1$}
\put(52,305){\scriptsize$p+1$}

\qbezier(-27,320)(-24,308)(-26,295)
\qbezier(-26,320)(-23,308)(-25,295)
\put(-29,323){$\frac{p+1}{2}$}

\qbezier(-48,217)(5,233)(59,218)
\qbezier(-48,216)(5,232)(59,217)
\qbezier(-48,215)(5,231)(59,216)

\put(-66,215){\scriptsize$D_k$}
\put(64,217){\scriptsize$p+1$}

\qbezier(-25,245)(-23,224)(-26,206)
\qbezier(-24,245)(-22,224)(-25,206)
\put(-27,248){$\frac{p+1}{2}$}


\qbezier(0,373)(-5,270)(-2,200) 
\qbezier(1,373)(-4,270)(-1,200)
\qbezier(-1,373)(-6,270)(-3,200)

\put(-9,384){$\frac{p-1}{2}$}


\put(19,373){\scriptsize{$j\equiv 0$}}
\put(16,373){\vector(-1,-1){13}}
\qbezier(-25,358)(5,357)(29,359)
\qbezier(-25,357)(5,356)(29,358)
\put(34,353){$\frac{p-1}{3}$}

\qbezier(-17,370)(-14,368)(-16,345)
\qbezier(-16,370)(-13,368)(-15,345)
\put(-32,373){$\frac{p-1}{6}$}


\put(-45,+270){$\cdot$}
\put(-45,253){$\cdot$}
\put(-45,234){$\cdot$}

\put(42,266){$\cdot$}
\put(43,247){$\cdot$}
\put(44,230){$\cdot$}


\put(-20,175){$p=12k+1$}





\qbezier(188,338)(245,341)(290,333)
\qbezier(188,337)(245,340)(290,332)

\put(164,334){$\frac{p+1}{3}$}

\qbezier(213,354)(215,345)(212,330)
\qbezier(214,354)(216,345)(213,330)
\put(209,357){$\frac{p+1}{6}$}

\qbezier(266,353)(254,282)(267,201)
\put(268,354){\scriptsize$1$}

\put(259,367){\scriptsize{$j\equiv 0$}}
\put(257,366){\vector(-1,-2){13}}


\qbezier(182,306)(245,316)(295,301)
\qbezier(182,305)(245,315)(295,300)
\qbezier(182,304)(245,314)(295,299)
\put(305,297){\scriptsize$D_1$}
\put(157,302){\scriptsize$p+1$}

\qbezier(180,279)(245,286)(296,270)
\qbezier(180,278)(245,285)(295,269)
\qbezier(180,277)(245,284)(295,268)
\put(305,268){\scriptsize$D_2$}
\put(157,274){\scriptsize$p+1$}

\qbezier(188,226)(245,236)(295,226)
\qbezier(188,225)(245,235)(295,225)
\qbezier(188,224)(245,234)(295,224)
\put(305,220){\scriptsize$D_k$}
\put(164,221){\scriptsize$p+1$}

\put(221,271){$\cdot$}
\put(222,253){$\cdot$}
\put(224,234){$\cdot$}


\qbezier(216,264)(216,273)(212,292)
\qbezier(217,264)(217,273)(213,292)
\put(207,296){$\frac{p+1}{2}$}

\qbezier(200,293)(201,305)(196,322)
\qbezier(199,293)(200,305)(195,322)
\put(192,324){$\frac{p+1}{2}$}

\qbezier(208,207)(211,225)(208,248)
\qbezier(207,207)(210,225)(207,248)
\put(202,253){$\frac{p+1}{2}$}


\put(232,380){$\frac{p-1}{2}$}

\qbezier(241,370)(240,270)(234,200)
\qbezier(240,370)(239,270)(233,200)
\qbezier(239,370)(238,270)(232,200)

%


\put(215,175){$p=12k+5$}



\qbezier(-50,10)(5,20)(45,8)
\qbezier(-50,11)(5,21)(45,9)
\qbezier(-50,12)(5,22)(45,10)

\put(-73,5){\scriptsize$D_2$}
\put(54,5){\scriptsize$p+1$}

\qbezier(-35,23)(-32,13)(-34,0)
\qbezier(-36,23)(-33,13)(-35,0)
\put(-48,25){$\frac{p+1}{2}$}

\qbezier(-49,36)(5,47)(46,35)
\qbezier(-49,37)(5,48)(46,36)
\qbezier(-49,38)(5,49)(46,37)

\put(-73,35){\scriptsize$D_1$}
\put(54,35){\scriptsize$p+1$}

\qbezier(-27,50)(-24,38)(-26,25)
\qbezier(-26,50)(-23,38)(-25,25)
\put(-29,53){$\frac{p+1}{2}$}

\qbezier(-48,-47)(5,-37)(59,-52)
\qbezier(-48,-46)(5,-36)(59,-51)
\qbezier(-48,-45)(5,-35)(59,-50)

\put(-68,-52){\scriptsize$D_k$}
\put(65,-53){\scriptsize$p+1$}

\qbezier(-24,-25)(-22,-46)(-25,-57)
\qbezier(-25,-25)(-23,-46)(-26,-57)
\put(-24,-22){$\frac{p+1}{2}$}


\qbezier(0,123)(11,0)(-4,-70) 
\qbezier(1,123)(12,0)(-3,-70)
\qbezier(2,123)(13,0)(-2,-70)

\put(-9,131){$\frac{p-1}{2}$}


\put(18,106){\scriptsize{$j\equiv 0$}}
\put(17,107){\vector(-2,-1){12}}
\qbezier(-28,101)(5,99)(31,102)
\qbezier(-28,100)(5,98)(31,101)
\put(36,93){$\frac{p-1}{3}$}

\qbezier(-15,114)(-12,102)(-14,89)
\qbezier(-14,114)(-11,102)(-13,89)
\put(-32,119){$\frac{p-1}{6}$}



\qbezier(-16,81)(50,36)(22,-64)
\put(-24,80){\scriptsize$1$}

\put(23,70){\scriptsize{$j\equiv 1728$}}
\put(20,69){\vector(-2,-1){12}}


\put(-45,-0){$\cdot$}
\put(-45,-17){$\cdot$}
\put(-45,-36){$\cdot$}

 \put(42,-4){$\cdot$}
\put(43,-23){$\cdot$}
\put(44,-40){$\cdot$}


\put(-20,-95){$p= 12k+7$}





\qbezier(221,115)(268,97)(267,-65)
\put(215,118){\scriptsize$1$}

\put(259,105){\scriptsize{$j\equiv 1728$}}
\put(256,105){\vector(-2,-1){12}}


\qbezier(188,68)(245,71)(290,63)
\qbezier(188,67)(245,70)(290,62)

\put(164,64){$\frac{p+1}{3}$}

\qbezier(210,82)(212,73)(209,58)
\qbezier(211,82)(213,73)(210,58)
\put(208,85){$\frac{p+1}{\mathrm{6}}$}

\qbezier(182,36)(245,46)(295,31)
\qbezier(182,35)(245,45)(295,30)
\qbezier(182,34)(245,44)(295,29)
\put(305,27){\scriptsize$D_1$}
\put(157,32){\scriptsize$p+1$}

\qbezier(180,9)(245,16)(296,0)
\qbezier(180,8)(245,15)(295,-1)
\qbezier(180,7)(245,14)(295,-2)
\put(305,-2){\scriptsize$D_2$}
\put(156,4){\scriptsize$p+1$}

\qbezier(188,-44)(245,-34)(295,-44)
\qbezier(188,-45)(245,-35)(295,-45)
\qbezier(188,-46)(245,-36)(295,-46)
\put(305,-50){\scriptsize$D_k$}
\put(164,-49){\scriptsize$p+1$}

\put(221,1){$\cdot$}
\put(222,-17){$\cdot$}
\put(224,-36){$\cdot$}


\qbezier(216,-6)(216,13)(212,22)
\qbezier(215,-6)(215,13)(211,22)
\put(207,26){$\frac{p+1}{2}$}

\put(210,47){\scriptsize{$j\equiv 0$}}
\put(231,52){\vector(1,2){6}}

\qbezier(202,23)(203,35)(198,52)
\qbezier(201,23)(202,35)(197,52)
\put(185,53){$\frac{p+1}{2}$}

\qbezier(208,-63)(211,-45)(208,-22)
\qbezier(209,-63)(212,-45)(209,-22)
\put(202,-17){$\frac{p+1}{2}$}



\put(232,132){$\frac{p-1}{2}$}

\qbezier(241,122)(244,0)(237,-70)
\qbezier(240,122)(243,0)(236,-70)
\qbezier(239,123)(242,0)(235,-70)

%


\put(215,-95){$p=12k+11$}

 
\end{picture}
\end{center}
\caption{Special fiber over $\overline{\F}_p$ of the minimal regular model with normal crossings of $X_{\mathrm{s}}^+ (p)$}  
\label{figureS+grossier}
\end{figure}

\begin{proof}

Starting from Theorem~4.4 of~\cite{EdixP20}, one checks that the proofs and arguments of Theorem~\ref{thNS+grossier} and 
Theorem~\ref{thS+} above can be transposed with obvious and small modifications. $\Box$

\end{proof}

\begin{figure}

\begin{center}

\begin{picture}(320,480)(-30,-160)

\qbezier(-50,280)(5,290)(45,278)
\qbezier(-50,281)(5,291)(45,279)
\qbezier(-50,282)(5,292)(45,280)

\put(-67,275){\scriptsize$D_2$}
\put(54,275){\scriptsize$p+1$}

\put(-26,291){\scriptsize$C_2$}
\qbezier(-28,294)(-25,280)(-26,268)
\qbezier(-29,294)(-26,280)(-27,268)
\put(-48,292){\scriptsize$\frac{p+1}{2}$}


\put(-21,361){\scriptsize$C_0$}
\qbezier(-25,362)(-22,353)(-23,335)
\qbezier(-26,362)(-23,353)(-24,335)
\put(-47,359){\scriptsize$\frac{p+1}{6}$}

\put(24.5,364){\scriptsize$B$}
\qbezier(35,363)(15,302)(18,262)
\put(38,360){\scriptsize$1$}
\put(17,255){$\cdot$}
\put(18,250){$\cdot$}
\put(18,245){$\cdot$}
\qbezier(20,242)(22,213)(29,198)


\qbezier(-51,306)(5,321)(46,309)
\qbezier(-51,307)(5,322)(46,310)
\qbezier(-51,308)(5,323)(46,311)

\put(-67,310){\scriptsize$D_1$}
\put(53,305){\scriptsize$p+1$}

\put(-21,326){\scriptsize$C_1$}
\qbezier(-25,327)(-22,313)(-23,301)
\qbezier(-26,327)(-23,313)(-24,301)
\put(-43,329){\scriptsize$\frac{p+1}{2}$}

\qbezier(-48,217)(5,233)(59,218)
\qbezier(-48,216)(5,232)(59,217)
\qbezier(-48,215)(5,231)(59,216)

\put(-66,215){\scriptsize$D_k$}
\put(64,217){\scriptsize$p+1$}

\put(-21,232){\scriptsize$C_k$}
\qbezier(-25,232)(-22,218)(-23,200)
\qbezier(-26,232)(-23,218)(-24,200)
\put(-43,233){\scriptsize$\frac{p+1}{2}$}


\qbezier(0,393)(-5,270)(-2,200) 
\qbezier(1,393)(-4,270)(-1,200)
\qbezier(-1,393)(-6,270)(-3,200)

\put(-15,389){\scriptsize$A$}
\put(4,379){$\frac{p-1}{2}$}

\qbezier(-45,351)(5,350)(49,352)
\qbezier(-45,350)(5,349)(49,351)
\put(57,346){$\frac{p+1}{3}$}
\put(-59,348){\scriptsize$D_0$}


\put(-57,+265){$\cdot$}
\put(-57,248){$\cdot$}
\put(-57,229){$\cdot$}

\put(42,266){$\cdot$}
\put(43,247){$\cdot$}
\put(44,230){$\cdot$}


\put(36,200){\scriptsize$B$}



\put(-75,175){\small{Minimal regular model with normal crossings}}
\put(-40,163){\small{(Theorem~\ref{thS+grossier} for $p=12k+5$).}}


 
\put(83,301){\vector(1,0){76}}

\put(85,290){\scriptsize{Contraction of the $D_i$}}
\put(102,280){\scriptsize{for $i>0$}}





\qbezier(188,341)(245,344)(290,336)
\qbezier(188,340)(245,343)(290,335)
 
\put(174,337){\scriptsize$D_0$}

\put(270,348){\scriptsize$B$}

\put(204,351){\scriptsize$C_0$}
\qbezier(200,352)(203,343)(202,325)
\qbezier(201,352)(204,343)(203,325)


\put(227,382){\scriptsize$A$}

\qbezier(241,385)(243,270)(240,200)
\qbezier(240,385)(242,270)(239,200)
\qbezier(239,385)(241,270)(238,200)

%



\qbezier(280,345)(205,320)(240,300)
\qbezier(240,300)(255,288)(255,280)
\qbezier(255,280)(254,278)(252,275)


\qbezier(225,253)(225,256)(227,259)
\qbezier(240,240)(230,245)(225,253)
\qbezier(200,195)(275,220)(240,240)

 
\qbezier(234,335)(242,327)(251,315)
\qbezier(234,336)(242,328)(251,316)
\put(218,333){\scriptsize$C_1$}
 

\qbezier(253,308)(241,300)(232,288)
\qbezier(253,309)(241,301)(232,289)
\put(255,303){\scriptsize$C_2$}



\qbezier(253,248)(241,240)(230,233)
\qbezier(253,249)(241,241)(230,234)
\put(255,243){\scriptsize$C_{k-1}$}


\qbezier(232,218)(242,210)(251,203)
\qbezier(232,219)(242,211)(251,204)
\put(218,217){\scriptsize$C_k$}


\put(255,261){$\cdot$}
\put(255,264){$\cdot$}
\put(255,267){$\cdot$}

\put(222,261){$\cdot$}
\put(222,264){$\cdot$}
\put(222,267){$\cdot$}


\put(193,201){\scriptsize$B$}


\put(215,165){Step 1}




 
\put(160,159){\vector(-3,-2){70}}

\put(60,148){\scriptsize{Contraction of the $C_i$,}} 
\put(90,140){\scriptsize{$i>0$}}



\qbezier(-30,110)(0,95)(45,108)
\qbezier(-30,111)(0,96)(45,109)

\put(-47,105){\scriptsize$D_0$}

\qbezier(-15,125)(-10,105)(-13,85)
\qbezier(-14,125)(-9,105)(-12,85)

\put(-30,126){\scriptsize$C_0$}



\put(35,115){\scriptsize$B$}

\qbezier(35,130)(30,110)(20,100)
\qbezier(20,100)(0,86)(17,75)
\qbezier(17,75)(25,70)(17,65) 
\qbezier(17,65)(1,55)(20,45)  
\qbezier(20,45)(27,40)(21,33)  


\put(19,28){$\cdot$}
\put(17,24){$\cdot$} 
\put(15,20){$\cdot$}  


\qbezier(14,18)(4,8)(22,-2) 


\put(-4,125){\scriptsize$A$}
\qbezier(8,124)(6,87)(09,-0)
\qbezier(9,124)(7,87)(10,-0)
\qbezier(10,124)(8,87)(11,-0)



\put(-20,-15){Step 2}




\put(79,72){\vector(1,0){80}}
\put(90,76){\scriptsize{Contraction of $A$}}


\qbezier(190,105)(250,97)(315,105)
\qbezier(190,106)(250,98)(315,106)

\put(175,103){\scriptsize$D_0$}


\qbezier(203,117)(206,107)(202,87)
\qbezier(204,117)(207,107)(203,87)

\put(192,119){\scriptsize$C_0$}


\put(301,113){\scriptsize$B$}

\qbezier(295,117)(308,99)(284,97)
\qbezier(284,97)(278,96)(272,98)
\qbezier(272,98)(238,108)(212,84)
\qbezier(212,84)(196,60)(219,33)
\qbezier(219,33)(249,7)(286,33)
\qbezier(286,33)(309,57)(281,88)
\qbezier(281,88)(248,113)(223,88)
\qbezier(223,88)(201,68)(226,41)
\qbezier(226,41)(257,19)(281,49)
\qbezier(281,49)(291,59)(281,79)
\qbezier(281,79)(246,122)(227,79)
\qbezier(227,79)(223,68)(227,59)
\qbezier(227,59)(252,27)(271,61)

\put(271,65){$\cdot$}
\put(272,71){$\cdot$} 
\put(270,76){$\cdot$}


\put(257,57){\scriptsize$B$}

\qbezier(249,102)(261,99)(268,85)

\qbezier(247,102)(226,110)(231,117)

\put(237,116){\scriptsize$B$}



\put(245,-10){Step 3}



 
\put(170,23){\vector(-3,-2){90}}

\put(060,-0){\scriptsize{Contraction of $D_0$}}
 


\qbezier(-37,-42)(6,-50)(56,-37) 
\qbezier(-37,-43)(6,-51)(56,-38)
\put(-49,-43){\scriptsize$C_0$}

 
\put(34,-25){\scriptsize$B$}
\qbezier(11,-44)(30,-34)(31,-22)

\qbezier(5,-46)(-12,-48)(-26,-63)
\qbezier(-26,-63)(-43,-87)(-20,-114)

\qbezier(-20,-114)(8,-138)(46,-114)
\qbezier(46,-114)(69,-90)(41,-59)
\qbezier(41,-59)(9,-33)(-17,-59)
\qbezier(-17,-59)(-39,-79)(-13,-106)
\qbezier(-13,-106)(17,-130)(41,-100)
\qbezier(41,-100)(51,-91)(41,-70)
\qbezier(41,-70)(4,-22)(-13,-70)

\put(-15,-77){$\cdot$}
\put(-17,-82){$\cdot$} 
\put(-16,-88){$\cdot$}

\qbezier(-13,-91)(12,-123)(31,-90)
\qbezier(31,-90)(38,-80)(34,-70)
\qbezier(34,-70)(20,-47)(8,-45)

\put(0,-98){\scriptsize$B$}

\qbezier(8,-45)(-17,-38)(-10,-30)

\put(-5,-30){\scriptsize$B$}



\put(-10,-140){Step 4}


\put(98,-67){\vector(1,0){100}}
\put(119,-87){\scriptsize{Contraction of $C_0$}}


\put(314,-20){\scriptsize$B$}
\qbezier(277,-46)(309,-38)(311,-18)

\qbezier(277,-46)(258,-47)(243,-65)
\qbezier(243,-65)(227,-87)(250,-114)

\qbezier(250,-114)(278,-138)(316,-114)
\qbezier(316,-114)(339,-90)(311,-59)
\qbezier(311,-59)(279,-33)(253,-59)
\qbezier(253,-59)(231,-79)(257,-106)
\qbezier(257,-106)(287,-130)(311,-100)
\qbezier(311,-100)(321,-91)(311,-70)
\qbezier(311,-70)(274,-22)(257,-70)

\put(255,-77){$\cdot$}
\put(253,-82){$\cdot$} 
\put(254,-88){$\cdot$}

\qbezier(257,-91)(282,-123)(301,-90)
\qbezier(301,-90)(308,-80)(301,-65)
\qbezier(301,-65)(295,-50)(274,-43)

\put(272,-102){\scriptsize$B$}

\qbezier(274,-43)(260,-39)(253,-35)

\put(241,-36){\scriptsize$B$}





\put(192,-140){Step 5 (Minimal regular model)}


\end{picture}
\end{center}
\caption{Minimal resolution of $X_{\mathrm{s}}^+ (p)$ in the case $p=12k+5$} 
\label{figureS+minimal5}
\end{figure}
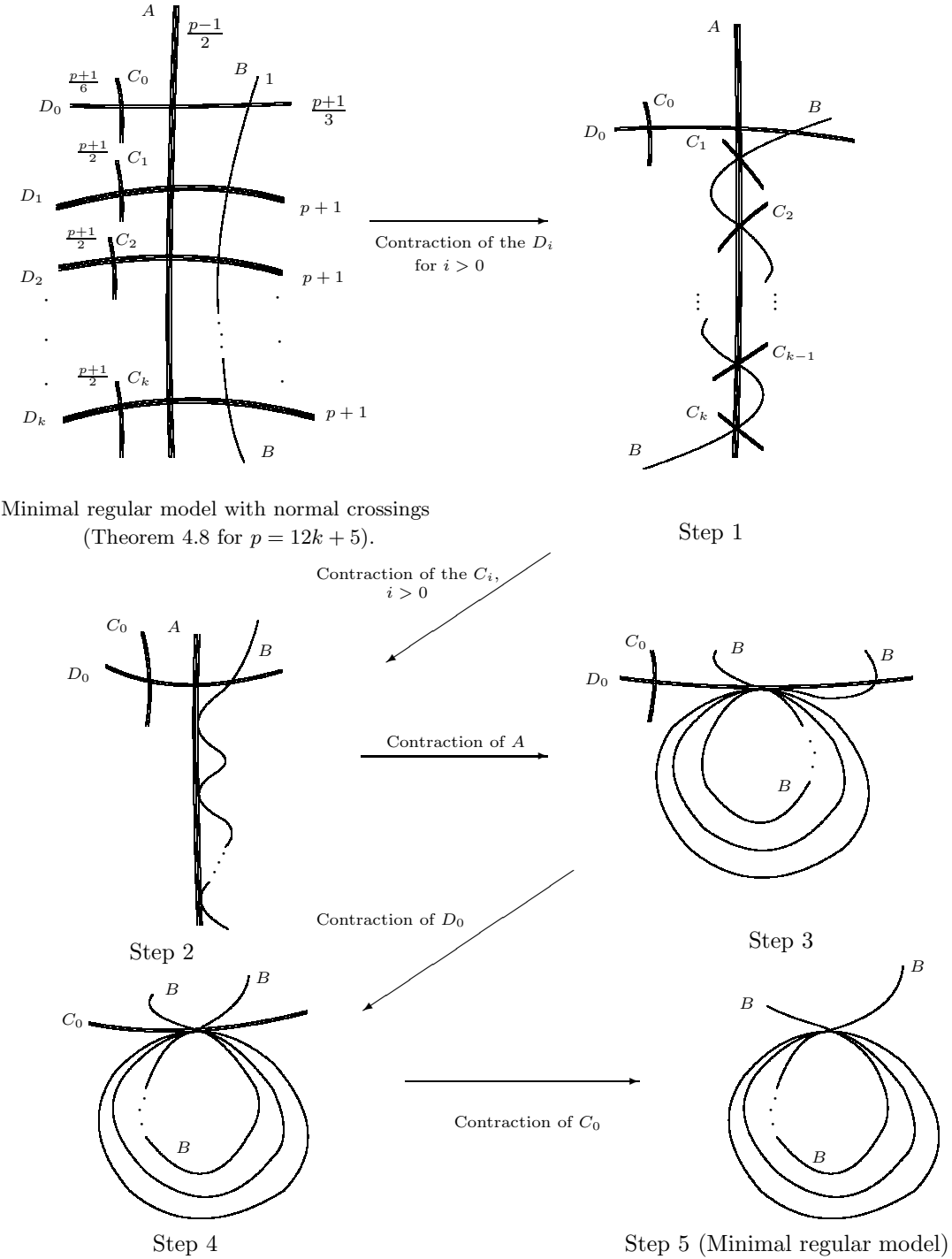

\begin{coro}
\label{minimalXs+}
Let $p$ be a prime number. The special fiber of the {\em minimal regular model} of $X_{\mathrm{s}}^+ (p)$ over 
$\Z_p^{\mathrm{ur}}$ is reduced, made of one single component, which is a projective line with multiplicity $1$ 
intersecting itself (many times) at one singular point (unless $p\leq 7$, where the curve has genus $0$). 
\end{coro}

\begin{proof}

We do the same computations as for Corollaries~\ref{minimalXns^+} or \ref{minimalXs}. To be explicit in the case $p=12k+5$ 
for instance, we contract the Drinfeld components with multiplicity $p+1$; then we contract the components with multiplicity 
$(p+1)/2$, then the central Igusa component (multiplicity $(p-1)/2$), the component with multiplicity $(p+1)/3$, and finally 
the one with multiplicity $(p+1)/6$. See Figure~\ref{figureS+minimal5}. $\Box$
\end{proof}

\begin{figure}
\begin{center}
\begin{picture}(50,7)(80,-60)

\put(5,1){\circle{16}}
\put(1,0){\scriptsize{$A$}}

\put(-3,-2){\line(-3,-1){60}}
\put(-68,-28){\circle{16}}
\put(-74,-29){\scriptsize{$D_{0}$}}

\put(-60,-28){\line(1,0){8}}
\put(-43,-28){\circle{16}}
\put(-49,-29){\scriptsize{$E_0$}}

\put(-68,-35){\line(0,-1){9}}
\put(-68,-53){\circle{16}}
\put(-74,-54){\scriptsize{$F_0$}}

\put(0,-7){\line(-1,-1){14}}
\put(-18,-28){\circle{16}}
\put(-24,-29){\scriptsize{$D_{1}$}}

\put(-10,-28){\line(1,0){8}}
\put(7,-28){\circle{16}}
\put(1,-29){\scriptsize{$E_1$}}

\put(-18,-35){\line(0,-1){9}}
\put(-18,-53){\circle{16}}
\put(-22,-54){\scriptsize{$F_1$}}

\put(5,-58){\small{$(X_{\mathrm{ns}}^+ (p))$}}

\put(31,-30){$\cdots$}

\put(12,-3){\line(3,-1){50}}
\put(62,-28){\circle{16}}
\put(56,-29){\scriptsize{$D_k$}}

\put(70,-28){\line(1,0){8}}
\put(87,-28){\circle{16}}
\put(81,-29){\scriptsize{$E_k$}}

\put(62,-35){\line(0,-1){9}}
\put(62,-53){\circle{16}}
\put(56,-54){\scriptsize{$F_k$}}


\put(205,1){\circle{16}}
\put(201,0){\scriptsize{$A$}}

\put(197,-2){\line(-3,-1){60}}
\put(132,-28){\circle{16}}
\put(126,-29){\scriptsize{$D_{0}$}}

\put(140,-28){\line(1,0){8}}
\put(157,-28){\circle{16}}
\put(151,-29){\scriptsize{$C_0$}}


\put(200,-6){\line(-1,-1){15}}
\put(182,-28){\circle{16}}
\put(176,-29){\scriptsize{$D_{1}$}}

\put(190,-28){\line(1,0){8}}
\put(207,-28){\circle{16}}
\put(201,-29){\scriptsize{$C_1$}}

\put(202,-63){\circle{16}}
\put(198,-65){\scriptsize{$B$}}

\put(194,-65){\line(-2,1){60}}
\put(196,-57){\line(-1,2){11}}
\put(209,-61){\line(3,2){45}}

\put(235,-58){\small{$(X_{\mathrm{s}}^+ (p))$}}

\put(231,-30){$\cdots$}

\put(212,-3){\line(3,-1){50}}
\put(262,-28){\circle{16}}
\put(256,-29){\scriptsize{$D_k$}}

\put(270,-28){\line(1,0){8}}
\put(287,-28){\circle{16}}
\put(281,-29){\scriptsize{$C_k$}}

\end{picture}
 
\end{center}
\caption{Dual graphs of the special fibers of $X_{\mathrm{ns}}^+ (p)$ vs. $X_{\mathrm{s}}^+ (p)$ (case $p=12k+5$).
(The components are labelled as in Figures~\ref{figureNS+minimal5} and~\ref{figureS+minimal5} respectively.)} 
\label{XS+vsXNS+}
\end{figure}
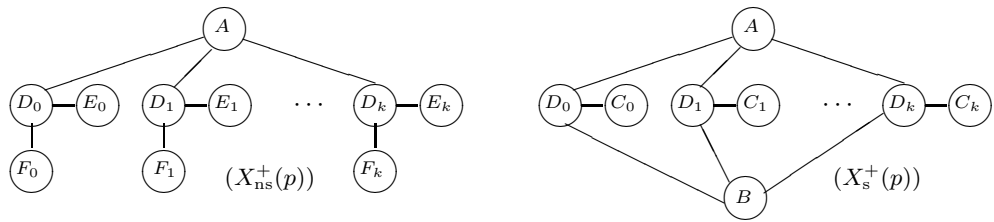

\subsubsection{${J}_{\mathrm{s}}^+ (p)$ has connected fibers}
 
\begin{propo}
Denote by ${\mathcal J}_{\mathrm{s}}^+ (p)$, for any prime $p$, the N\'eron model over $\Z_p^{\mathrm{ur}}$ of the Jacobian 
of $X_{\mathrm{s}}^+ (p)$. The component group ${\mathcal J}_{\mathrm{s}}^+ (p)/{\mathcal J}^{+,0}_{\mathrm{s}} (p)$ of the 
special fiber of ${\mathcal J}_{\mathrm{s}}^+ (p)$ is trivial. 
\end{propo}

\begin{proof} This trivially follows from the fact that the minimal regular model has only one component (and is readily
checked by writing the intersection matrix for the minimal regular model with normal crossings). $\Box$

\end{proof}

\bigskip

\noindent
Bas Edixhoven\\
Mathematisch Instituut, Universiteit Leiden, Niels Bohrweg 1, 2333 CA Leiden, Nederland\\

\bigskip

\noindent 
Pierre Parent\\
Univ. Bordeaux, CNRS, Bordeaux INP, IMB, UMR 5251,  F-33400, Talence, France\\
Pierre.Parent@math.u-bordeaux.fr


{\footnotesize

\begin{thebibliography}{99} 
%
\bibitem{BM17}
%
{\it J. S.~Balakrishnan, N. Dogra, J. S.~M\"uller, J. Tuitman and J. Vonk}, Explicit Chabauty-Kim for the Split Cartan 
Modular Curve of Level $13$, Ann. of Math. {\bf 189} (2019), no. 3, 885--944. 
%
\bibitem{BD23}
%
{\it J. S.~Balakrishnan, N. Dogra, J. S.~M\"uller, J. Tuitman and J. Vonk}, Quadratic Chabauty for modular curves: Algorithms 
and examples, Compos. Math. {\bf 159} (2023), no. 6, 1111--1152.
%
%
\bibitem{BM21}
%
{\it J. S.~Balakrishnan, A. J. Best, F. Bianchi, B. Lawrence, J. S. M\"uller, N. Triantafillou, J. Vonk}, Two recent $p$-adic 
approaches towards the (effective) Mordell conjecture. In {\em Arithmetic L-functions and differential geometric methods}, 
31--74, Progr. Math., {\bf 338}, Birkh\"auser/Springer, 2021.
%
%
%
\bibitem{BBC20}
%
{\it D.~Banerjee, D. Borah, Ch. Chaudhuri}, Arakelov self-intersection numbers of minimal regular models of modular curves 
$X_0 (p^2 )$, Math. Z. {\bf 296} (2020), no. 3--4, 1287--1329.
%
%
%
\bibitem{BMC22}
%
{\it D.~Banerjee, Ch. Chaudhuri, P. Majumder}, The intersection matrices of $X_0(p^r )$ and some applications, preprint
(arxiv.org/abs/2210.08866).
%
%
%
\bibitem{Baran}
%
{\it B. Baran}, An exceptional isomorphism between modular curves of level $13$, J. Number Theory {\bf 145} (2014), 273--300.
%
%
\bibitem{BSW75}
%
{\it B. J. Birch, W. Kuyk (ed.)}, Numerical tables of elliptic curves, in Modular functions of One Variable IV, Lecture Notes 
in Math., Vol. {\bf 476}, pp. 74--144, Springer, Brlin, 1975. 
%
%
\bibitem{BLR}
%
{\it S. Bosch, W. L\"utkebohmert, M. Raynaud}, N\'eron models, Ergebnisse der Mathematik und ihrer Grenzegebiete, vol. 
{\bf 21}. Springer-Verlag (1990).
%
%
\bibitem{CEdSt03}
%
{\it B. Conrad, B. Edixhoven, W. Stein}, $J_1 (p)$ has connected fibers, Documenta Math. {\bf 8} (2003), 325--402. 
%
%
%
\bibitem{dSE}
%
{\it B. de Smit, B. Edixhoven}, Sur un r\'esultat d'Imin Chen, Math. Res. Lett. {\bf 7} (2/3) (2000), 147--153.
%
%
%
\bibitem{DR73}
%
{\it P.~Deligne, M.~Rapoport},
Les sch\'emas de modules de courbes elliptiques. In: Modular functions
of one variable, II (Proc. Internat. Summer School, Univ. Antwerp,
Antwerp, 1972), pp.~143--316. Lecture Notes in Math., Vol.~{\bf 349}, Springer, Berlin, 1973. 

%
%
\bibitem{Edix89}
{\it B. Edixhoven}, Minimal resolution and stable reduction of $X_0 (N)$, Ann. Inst. Fourier {\bf 40} (1990), no. 1, 
31--67.
%
%
\bibitem{Edix89b}
%
{\it B. Edixhoven}, Stable models of modular curves and applications, Ph. D. thesis, Utrecht (1990) (available on 
the author's homepage).
%
%
%
\bibitem{EdEis}
%
{\it B. Edixhoven}, L'action de l'alg\`ebre de Hecke sur les groupes de composantes des jacobiennes des courbes modulaires 
est ``Eisenstein''. Courbes modulaires et courbes de Shimura (Orsay, 1987/1988). Ast\'erisque No. {\bf 196-197} (1991), 
159--170 (1992). 
%
%
%
%
\bibitem{EdLi22}
%
{\it B. Edixhoven, G. Lido}, Geometric Quadratic Chabauty, J. Inst. Math. Jussieu {\bf 22} (2023), no. 1, 279--333.
%
%
\bibitem{EdixP20}
%
{\it B. Edixhoven, P. Parent}, Semistable reduction of modular curves associated with maximal subgroups in prime level,  
Doc.\ Math.~{\bf 26} (2021), 231--269. 
%
%
%
\bibitem{KM85}
%
{\it N.~Katz, B.~Mazur}, Arithmetic moduli of elliptic curves, Annals of Math. Studies {\bf 108}, Princeton University Press, Princeton, NJ, 1985.
%
%
%
\bibitem{Lg77}
%
{\it G.~Ligozat}, Courbes modulaires de niveau 11. Modular functions of one variable, V (Proc. Second Internat. Conf., Univ. Bonn, Bonn, 1976), pp. 149--237. Lecture Notes in Math., Vol. {\bf 601}, Springer, Berlin, 1977. 
%
%
\bibitem{Liu02}
%
{\it Q. Liu}, Algebraic geometry and arithmetic curves, Vol. 6 of Oxford Graduate Texts in Mathematics,
Oxford University Press, Oxford (2002). Translated from the French by Reinie Ern\'e, Oxford Science Publications.
%
%
\bibitem{Dino90}
%
{\it D. Lorenzini}, Groups of components of N\'eron models of Jacobians, Compos. Math. {\bf 73} (1990), 145--160.
%
%
%
\bibitem{Ma76}
%
{\it B. Mazur}, Rational points on modular curves. Modular functions of one variable, V (Proc. Second Internat. Conf., Univ. 
Bonn, Bonn, 1976), pp. 107--148. Lecture Notes in Math., Vol. {\bf 601}, Springer, Berlin, 1977. 
%
%
%
%
%
\end{thebibliography}
}
\end{document}